\documentclass[11pt]{article}

\usepackage{geometry}
\usepackage{comment}
\usepackage{subcaption}
\usepackage{asymptote}
\usepackage{tikz}
\usepackage{svg}
\usepackage{tikz-3dplot}
\usepackage{forest}
\usepackage{pgfplots}
\pgfplotsset{compat=1.18}
\usetikzlibrary{calc}
\usepackage{graphicx}%
\usepackage{multirow}%
\usepackage{amsmath,amssymb,amsfonts}%
\usepackage{amsthm}%
\usepackage{mathrsfs}%
\usepackage[title]{appendix}%
\usepackage{xcolor}%
\usepackage{textcomp}%
\usepackage{manyfoot}%
\usepackage{booktabs}%
\usepackage{algorithm}%
\usepackage{algorithmicx}%
\usepackage{algpseudocode}%
\usepackage{listings}%
\usetikzlibrary{arrows.meta,positioning,automata,quotes}
\usetikzlibrary{external}
\usepgfplotslibrary{fillbetween}
\usepackage{hyperref}
\usepackage{authblk}
\usepackage[numbers]{natbib}

\newtheorem{theorem}{Theorem}%
\newtheorem{proposition}[theorem]{Proposition}%
\newtheorem{example}{Example}%
\newtheorem{corollary}{Corollary}%
\newtheorem{remark}{Remark}%
\newtheorem{definition}{Definition}%
\newtheorem{lemma}{Lemma}%
\newcommand\numberthis{\addtocounter{equation}{1}\tag{\theequation}}

\DeclareMathOperator{\lineends}{\texttt{end}}

\DeclareMathOperator{\prob}{Pr}

\DeclareMathOperator{\proj}{proj}
\DeclareMathOperator{\conv}{conv}

\DeclareMathOperator{\graph}{gr}

\DeclareMathOperator{\vertex}{vert}

\DeclareMathOperator{\R}{\mathbb{R}}
\DeclareMathOperator{\N}{\mathbb{N}}
\DeclareMathOperator{\diam}{diam}

\DeclareMathOperator{\for}{for}

\DeclareMathOperator{\power}{power}

\DeclareMathOperator{\card}{card}

\DeclareMathOperator{\corner}{corner}

\DeclareMathOperator{\AxisAligned}{\mathcal{H}}

\DeclareMathOperator{\Mul}{M}
\DeclareMathOperator{\RCG}{r}

\DeclareMathOperator{\disc}{disc}

\DeclareMathOperator{\rects}{\mathcal{B}}

\DeclareMathOperator{\AP}{\texttt{AP}}
\DeclareMathOperator{\Voxels}{\texttt{Voxels}}

\def \D {{\mathcal{D}}}

\def\mcirc{\mathop{\circ}}

\def \P{\mathcal{P}}
\def \Q{\mathcal{Q}}

\def \point{\mathcal{Q}}
\def \AxisAligned{\mathcal{H}}

\def \state{\mathcal{X}}
\def \transition{T}

\def \G{\mathcal{G}}
\def \Rlx{\mathcal{R}}

\def\limiting{*}

\providecommand{\keywords}[1]
{
  \small\textbf{\textit{Keywords---}} #1
}

\begin{document}

\title{Axis-Aligned Relaxations for Mixed-Integer Nonlinear Programming}
\title{Axis-Aligned Relaxations for Mixed-Integer Nonlinear Programming}
\author[1]{Haisheng Zhu}
\author[2]{Taotao He}
\author[1]{Mohit Tawarmalani\thanks{corresponding author: mtawarma@purdue.edu}}
\affil[1]{Mitch Daniels School of Business, Purdue University, West Lafayette, IN 47906}
\affil[2]{Antai School of Management, Shanghai Jiao-Tong University, Shanghai, China 200030}

\maketitle
\begin{abstract}
    We present a novel relaxation framework for general mixed-integer nonlinear programming (MINLP) grounded in computational geometry. Our approach constructs polyhedral relaxations by convexifying finite sets of strategically chosen points, iteratively refining the approximation to converge toward the simultaneous convex hull of factorable function graphs. The framework is underpinned by three key contributions: (i) a new class of explicit inequalities for products of functions that strictly improve upon standard factorable and composite relaxation schemes; (ii) a proof establishing that the simultaneous convex hull of multilinear functions over axis-aligned regions is fully determined by their values at corner points, thereby generalizing existing results from hypercubes to arbitrary axis-aligned domains; and (iii) the integration of computational geometry tools, specifically voxelization and QuickHull, to efficiently approximate feasible regions and function graphs. We implement this framework and evaluate it on randomly generated polynomial optimization problems and a suite of 619 instances from \texttt{MINLPLib}. Numerical results demonstrate significant improvements over state-of-the-art benchmarks: on polynomial instances, our relaxation closes an additional 20--25\% of the optimality gap relative to standard methods on half the instances. Furthermore, compared against an enhanced factorable programming baseline and Gurobi's root-node bounds, our approach yields superior dual bounds on approximately 30\% of \texttt{MINLPLib} instances, with roughly 10\% of cases exhibiting a gap reduction exceeding 50\%.
\end{abstract}


\maketitle
\section{Introduction}\label{sec1}
Algorithms to solve mixed-integer nonlinear programs (MINLPs) rely predominantly on \emph{factorable programming} (FP), a framework that leverages the recursive structure of nonlinear functions to construct convex relaxations \cite{mccormick1976computability, tawarmalani2004global, misener2014antigone, mahajan2021minotaur, bestuzheva2023global}.
FP represents nonlinear functions using their expression tree where leaves represent variables or constants and internal nodes denote operations (e.g., addition, multiplication). For each internal node, an auxiliary variable is introduced, and the relationship between this variable and its operands is relaxed using properties specific to the operation; notably, McCormick envelopes are employed to relax the multiplication nodes. By systematically relaxing each operation over the known bounds of its arguments, FP transforms a nonconvex problem into a tractable convex relaxation.

Factorable programming (FP) remains the dominant framework due to its scalability: the number of auxiliary variables and constraints introduced is comparable to that of the original formulation. However, this efficiency comes at the cost of tightness. FP relaxes each operation independently over a hypercube defined by variable bounds, thereby ignoring interdependencies among variables induced by linking constraints. This decoupling yields overly conservative relaxations that admit feasible regions significantly larger than the true feasible set, often including combinations of variable values that violate the original problem constraints.

Alternative general-purpose relaxation techniques include composite relaxations~\cite{he2021new}, reformulation-linearization technique (RLT)~\cite{sherali1990hierarchy}, and sum-of-squares (SOS) hierarchies~\cite{lasserre2001explicit}. While RLT and SOS relaxations often require a proliferation of auxiliary variables, composite relaxations maintain a variable count similar to FP. Theoretically, composite relaxations are guaranteed to be tighter than FP and have been shown to be significantly tighter for polynomial optimization problems~\cite{he2024mip}. Furthermore, there is an extensive literature on convexification techniques for specific function structures. This includes results on  multilinear functions~\cite{rikun1997convex,bao2015global,del2021running}, functions with indicator variables~\cite{gunluk2010perspective,atamturk2023supermodularity}, fractional functions~\cite{he2025convexification}, supermodular functions~\cite{tawarmalani2013explicit}, quadratic forms~\cite{burer2009copositive,dey2025second}. However, such analyses have typically been confined to structured domains like hypercubes and simplices. While explicit convex envelopes for bivariate functions over convex polygons exist and have shown promise~\cite{locatelli2016polyhedral,locatelli2018convex,muller2020using}, and FP has been augmented with discretization schemes~\cite{misener2014antigone,nagarajan2019adaptive,mahajan2021minotaur}, no systematic mechanism currently exists to construct convex relaxations that simultaneously avoid excessive auxiliary variables and yields tight envelopes for factorable functions over general non-convex domains. This paper proposes such a scheme, implements it for low-dimensional functions and regions, and demonstrates its efficacy on a large collection of benchmark problems.

Our framework builds on the expression tree representation central to FP, but introduces a fundamentally distinct relaxation strategy to address this limitation. The relaxations are constructed as the convex hull of finite sets of strategically chosen points identified via two approximation steps: (i) outer-approximation of the feasible set by an axis-aligned region, and (ii) piecewise-approximation of the inner functions associated with each operator. Although our approach also decomposes functions into low-dimensional atomic expressions, it leverages computational geometry to capture interdependencies among variables and functions, thereby constructing provably tighter relaxations than factorable programming.

The cornerstone of our approach is a novel relaxation technique for multilinear functions over axis-aligned regions. To frame relaxation construction as such a problem, we approximate feasible regions using voxelization and graphs of functions using tessellations. Voxelization discretizes continuous geometric domains into a grid of voxels indicating feasibility, while tessellation approximates the function graph with a union of polygonal shapes. Standard FP can be viewed as a special case that constructs relaxations over the coarsest voxel, the bounding hypercube. In contrast, general axis-aligned regions offer opportunities for significantly tighter relaxations. Moreover,  unlike FP, which loses the structure of inner functions by replacing them with variables, our tessellation schemes preserve this structure while relaxing outer-functions. By adapting established computational geometry methods--such as discretizing lines~\cite{cohen20023d}, surfaces~\cite{sramek2002alias,stolte1997robust}, and polyhedra~\cite{gorte2016rasterization}--to construct relaxations, our framework enables higher fidelity convex relaxations without adding auxiliary variables, albeit with extra computational effort.

Our technique can also be viewed as a generalization of McCormick relaxation, which constructs the convex hull of four points on a bilinear product's graph. This geometric interpretation has inspired techniques that focus on specific points whose convex extension yields valid relaxations~\cite{crama1993concave,rikun1997convex,tawarmalani2002convex,meyer2005convex}. Our construction represents each point in an axis-aligned region as a convex combination of specific corner points using a discrete Markov chain. Using this representation, we prove that appropriately lifting these corner points into a higher-dimensional space yields the finite collection of points required to convexify multilinear compositions. 

Axis-aligned regions offer greater flexibility than hypercubes; refining the voxelization resolution allows for an arbitrarily close approximation of the feasible set. This approach extends beyond hierarchical convexification to explicitly exploit nonlinear and nonconvex structures. We take advantage of convergent voxelization schemes with computational geometry tools, such as the \texttt{Quick Hull} algorithm~\cite{barber1996quickhull} to enhance relaxation quality. Notably, we prove that our relaxation converges to the tightest possible relaxation, the convex hull of the function over the feasible domain, as the voxelization resolution increases.

We evaluate these relaxations on randomly generated polynomial optimization problems \cite{he2024mip} and $619$ instances from \texttt{MINLPLib} \cite{bussieck2003minlplib}. For polynomial optimization, our relaxation closes 20-25\% of the gap on half of the instances, clearly dominating factorable programming, Gurobi root-node~\cite{gurobi}, and composite relaxation~\cite{he2024mip}. For \texttt{MINLPLib}, we improve upon an enhanced FP baseline and Gurobi on approximately 30\% of instances, closing 50\% of the gap on 10\% of them. Although constructing axis-aligned regions and convexifying point sets incurs computational expense, the resulting relaxations are not significantly harder to solve than FP counterparts, likely due to the absence of extra auxiliary variables and the sparsity of generated inequalities.
Within \texttt{MINLPLib}, we identified six challenging unsolved instances where our relaxation achieves a bound superior to the final bounds reported by solvers after many branch-and-bound iterations.

To illustrate, consider the following nonlinear program:
\begin{alignat*}{2}
\max &\quad& e^{x_1-x_2} x_1 x_2 \\
\text{s.t.}&     & x_1, x_2 \in [0,1] \numberthis \label{example 1}
\end{alignat*}
Reformulating with auxiliary variables $t_1, t_2, t_3$:
\begin{alignat*}{2}
\max&&\quad&t_1 t_2\\
\text{s.t.}&&&t_1 = e^{t_3}\\
    &&&t_3 = x_1 - x_2\\
    &&&t_2 = x_1 x_2\\
    &&&x_1, x_2 \in [0,1] \numberthis \label{factored example 1}
\end{alignat*}
FP relaxes the bilinear term $t_1t_2$ over the variable bounds (dashed rectangle in Figure~\ref{fig1}(a)), which outer-approximates the feasible region. In contrast, our framework approximates the feasible domain by voxelization (Figure~\ref{fig1}(b)), producing a significantly tighter relaxation as depicted in Figure~\ref{fig1}(c).

\begin{figure}[htbp]
\centering

\begin{subfigure}[t]{0.33\linewidth}
\includegraphics{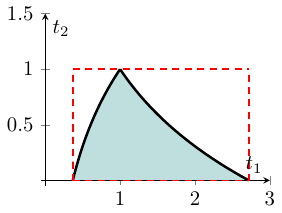}
\caption{}\label{fig1:a}
\end{subfigure}
\begin{subfigure}[t]{0.33\linewidth}
\includegraphics{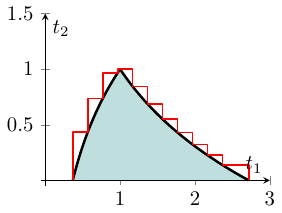}
\caption{}\label{fig1:b}
\end{subfigure}
\begin{subfigure}[t]{0.32\linewidth}
\includegraphics[page=1]{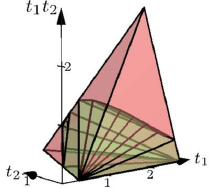}
\caption{}\label{fig1:c}
\end{subfigure}

\caption{Comparison of relaxation over rectangular bounds versus voxelization}
\label{fig1}
\end{figure}

To our knowledge, this paper presents the first general-purpose application of voxelization, tessellation, and \texttt{QuickHull} to systematically create MINLP relaxations. The framework reduces the relaxation of many MINLPs to computational geometry tasks in low-dimensional spaces, guaranteeing convergence to the tightest possible relaxation in structured settings. 
This application is enabled by:
\begin{itemize}
    \item Tools to approximate arbitrary collections of functions using piecewise multilinear functions and voxelization supported with interval-arithmetic techniques~\cite{moore2009introduction}, eliminating the need for extra auxiliary variables;
    \item A foundational convex extension property that permits simultaneous convexification of piecewise multilinear functions over axis-aligned regions, identifying point sets whose convex hull encompasses multiple function graphs simultaneously.
\end{itemize}
The framework is supported by rigorous convergence guarantees, demonstrating that the relaxation converges to the convex hull of a collection of function graphs over arbitrary domains as the resolution increases. We also develop specialized techniques for products of univariate functions that provably outperform existing methods. Our tools are implemented in a \texttt{Julia}-based framework built on \texttt{JuMP}~\cite{lubin2023jump}, enabling efficient generation of relaxations that achieve superior performance on benchmark instances.

Section~\ref{sec:2D} presents two-dimensional results. Section~\ref{sec:prod-over-box} provides new inequalities for the product of two functions defined over a rectangle, improving FP and composite relaxations~\cite{he2021new}. This motivates polyhedral relaxations using computational geometry. Section~\ref{sec:bilinear-over-nonbox} approximates the feasible domain via axis-aligned regions and proves that the convex hull of a bilinear term over such a region is finitely-generated. Section~\ref{sec:prod-over-nonbox} combines these approaches to construct polyhedral relaxations for the products of univariate functions over arbitrary regions. Section~\ref{sec:nD} addresses multilinear compositions in $n$-dimensional space. Section~\ref{sec:axis-alignedMultilinear} abstracts these problems by replacing function approximations with tessellations and constructing axis-aligned regions that incorporate these tessellations. The core theoretical result then identifies a finite set of points to convexify multilinear functions simultaneously over such regions. Section~\ref{sec:convergence} proves convergence to the convex hull as the resolution of the approximation increases. Section~\ref{sec:implementation} discusses our restricted implementation, describing two voxelization methods (Sections~\ref{sec:voxelProject}, ~\ref{sec:voxelQuadtree}). Section~\ref{sec:numerical} evaluates performance on polynomial optimization problems (Section~\ref{sec:polynomial}) and \texttt{MINLPlib} (Section~\ref{sec:minlplib}), demonstrating improvements over prevalent factorable programming relaxations. Finally, we explore hyperparameters controlling approximation resolution to analyze the trade-off between computational time and relaxation tightness.

\emph{Notation:} For a set $S$, we denote its convex hull as $\conv(S)$ and its extreme points as $\vertex(S)$. The corner points of an axis-aligned region $\AxisAligned$ are denoted by $\corner(\AxisAligned)$. The projection of a set $S$ onto the subspace of variables $x$ is written as $\proj_x(S)$. Finally, the cardinality of a finite point set $S$ is denoted as $\card(S)$. The sequence $1,\ldots,m$ is denoted as $[m]$.

\section{Tools: Piecewise Approximation and Axis-Aligned Regions}\label{sec:2D}
This section details the key tools employed to construct relaxations, with particular emphasis on the relaxation of function products, a critical step in relaxing factorable programs. Our approach proceeds in two stages: first, approximating each nonlinear function via a piecewise linear function, and second, replacing the feasible domain with its voxelization. These concepts are introduced in Sections~\ref{sec:prod-over-box} and~\ref{sec:bilinear-over-nonbox}, respectively. Together, these steps define a finite set of points whose convex hull yields a relaxation of the product over an arbitrary domain, as detailed in Section~\ref{sec:prod-over-nonbox}. We demonstrate that this approach can be tuned to balance approximation fidelity and computational efficiency, establishing it as a foundational component for relaxing complex factorable expressions.

\subsection{Multiplication of Functions over a Rectangle}\label{sec:prod-over-box}
In this subsection, we focus on relaxing  the product of two univariate functions over a rectangular domain. Specifically, we analyze the graph:
\[
\G := \bigl\{(x,\mu) \in [0,1]^2 \times \R \bigm| \mu = f_1(x_1)f_2(x_2) \bigr\},
\]
where each $f_i:[0,1] \to \R$ is an arbitrary nonlinear function. Although we restrict the domain of $x_i$ to $[0,1]$, our result generalizes to arbitrary intervals; since the functions are arbitrary, the general case can be reduced to the $[0,1]$ case via an affine transformation. 
Our relaxation relies on \textit{piecewise-linear approximation}, demonstrating that approximating each $f_i$ generates a finite set of points whose convex hull provides a valid relaxation for the graph $\G$. This approach tightens  the recent composite relaxation introduced by \citet{he2021new}. Moreover, by tuning the approximation accuracy, we can effectively balance computational efficiency with the tightness of the resulting relaxation. 

To contextualize our contribution, we first review  the composite relaxation framework proposed by \citet{he2021new}. Their method improves factorable programming relaxations by constructing a polyhedral approximation $U_i$ for each individual function $f_i(x_i)$, $i=1,2$, within a lifted space. The product is then relaxed over the Cartesian product $U_1 \times U_2$. Although the ability to choose the dimension of each $U_i$ offers flexibility to balance accuracy with computational cost as higher-dimensional approximations yield tighter relaxations, this approach suffers from two key limitations:

\begin{enumerate}
    \item The projection of the relaxation fails to converge to the convex hull of the graph of the product function as the dimension of the approximations increases.
    \item The number of inequalities required to convexify the product over $U_1\times U_2$ grows exponentially with $\min\bigl\{\text{dim}(U_1), \text{dim}(U_2)\bigr\}$, regardless of the actual number of inequalities required to describe the convex hull of $f_1(x_1) f_2(x_2)$.
\end{enumerate}
In contrast, our technique yields a tighter relaxation that converges to the convex hull of the product's graph, while eliminating redundant inequalities. We first illustrate this approach with an example, before providing a formal description of the method.

\begin{example}\label{ex:comparison-oneestimator}
Consider the monomial $x_1^2x_2^2$ over the domain $[0,2]^2$. For each $i=1,2$, we define two affine underestimators: $2x_i-1$ and $4x_i-4$. Using these, we construct a convex underestimator for the monomial and compare its strength against the composite relaxation proposed by~\citet{he2021new} and its limiting case in~\citet{he2022tractable}.

We first present our relaxation. For each $i = 1,2$, we define the polyhedral outer-approximation
\begin{equation}\label{eq:examplePi}
P_i:= \Bigl\{(x_i,t_i) \Bigm|  \max\bigl\{0, 2x_i-1, 4x_i-4\bigr\} \leq  t_i \leq 4,\ x_i \in [0,2] \Bigr\}.
\end{equation}
By computing the convex envelope of  $t_1t_2$ over $P_1 \times P_2$ and substituting $x_i^2\leq t_i\leq 4x_i$, we obtain the convex underestimator:
\[
r(x):= \max\{ 0 ,\ 4x_1^2 + 4x_2^2-16 ,\ 8x_1+8x_2-20\}.
\]

The composite relaxation of~\cite{he2021new} cannot utilize the second underestimator $4x_i - 4$, as its upper bound coincides with that of $x_i^2$ over $[0,2]$. Restricting to only the underestimator $2x_i-1$ and its corresponding upper bound $3$, Theorem~1 of~\cite{he2021new} yields the relaxation:
\[
r^1(x) := \max \left\{ \begin{aligned}
&0,\ 4x_1^2+4x_2^2-16,\ 8x_1+3x_2^2-16,\ 3x_1^2+8x_2-16\\
&6x_1+6x_2-15,\ 2x_1+2x_2+3x_1^2+3x_2^2-17
\end{aligned}
\right\},
\]
as derived in Example~2 of~\cite{he2021new}. While the composite relaxation improves as the number of underestimators increases, the limiting case--using infinitely many underestimators for $x_i^2$--is given by Example~4 of~\cite{he2022tractable}:
\[
r^\infty(x):=\max\biggl\{0, \int_{1-\frac{x_1}{2}}^{\frac{x_2}{2}}\Bigl(\frac{4\lambda^2 - 4  +4x_1-x_1^2 }{\lambda^2}\Bigr)\Bigl(\frac{4\lambda^2 - 8\lambda+4x_2-x_2^2 }{(1-\lambda)^2}\Bigr) \mathrm{d}\lambda \biggr\}. 
\]
At $x = (1.5,1.5)$, where the true function value is $\frac{81}{16} = 5.0625$, we compute: $r^1(x) =  6x_1+6x_2 -15 = 3$, $r^\infty(x) \approx 3.274653$, and $r(x) = 8x_1+8x_2-20= 4$. Thus, $r^1(x) < r^\infty(x) < r(x)$, demonstrating that our relaxation constructed using just two underestimators per variable can be tighter than the limiting composite relaxation based on infinitely many underestimators.\hfill \qed
\end{example}
In Example~\ref{ex:comparison-oneestimator}, the set $P_i$ employed to outer-approximate $x_i^2$ is a pentagon. A step in constructing $r(\cdot)$ involved computing the convex envelope of $t_1t_2$ over $P_1 \times P_2$, followed by projecting out the $t_i$ variables utilizing $x_i^2 \le t_i \le 4x_i$. Now, we derive an explicit convex hull description for $t_1t_2$ over an arbitrary pentagon. Assume that  the range of $f_i$ is $[f_i^L, f_i^U]$, and suppose that the graph of $f_i$ is outer-approximated by a pentagon defined by the following vertices:
\begin{equation}\label{eq:pentagon}
v_{i0} = 
	(0 ,f_i^L) \quad
v_{i1} = (p_{i1},	f_i^L)
\quad
v_{i2} = (p_{i2},	b_i)
\quad
v_{i3} = (1,f^U_i) \text{ and }
v_{i4} = (0,f^U_i),
\end{equation}
where $0 \leq p_{i1} < p_{i2} \leq 1$, and $f_i^L < b_i < f_i^U$. Theorem~\ref{thm:pentagon} characterizes the set: 
\[
\conv\bigl\{(x,t, t_1t_2) \bigm| (x_i,t_i) \in  P_i\; \for i = 1, 2 \bigr\}.
\]
\begin{theorem}~\label{thm:pentagon}
Let $P_i$ be the pentagon defined by the vertices in \eqref{eq:pentagon} for $i = 1, 2$ and let $P = P_1\times P_2$. Define the parameters:
\[
r_{i1} := \frac{p_{i2} - p_{i1}}{b_i - f_i^L} \qquad r_{i2} := \frac{1 - p_{i2}}{f_i^U - b_i} \quad \text{and } \quad w_i := \frac{1}{r_{i1} - r_{i2}}.
\]
The convex envelope of $f_1f_2$ over $P$ is given by $\max_{j=1,\ldots,6}\bigl\{ \bigl\langle \alpha_j, (x,t) \bigr\rangle  + c_j \bigr\}$, where $\alpha_j$ denotes the $j^{\text{th}}$ row of the matrix:
\[
\begin{bmatrix}
	0 &  0  & f_2^U & f_1^U \\
	0 &  0 & f_2^L & f_1^L \\
	(f_2^U -f_2^L)w_1 &0 & f_2^L -  r_{12}(f_2^U -f_2^L)w_1 & b_1 \\
	0 & (f_1^U - f_1^L) w_2 & b_2	 &  f_1^L -  r_{22}(f_1^U -f_1^L)w_2   \\
	(f_2^U-b_2)w_1 & (f_1^U-b_1)w_2 & b_2- r_{12}(f_2^U-b_2)w_1  &  b_1 - r_{22} (f_1^U-b_1)w_2 \\
		(b_2-f_2^L)w_1 & (b_1-f_1^L)w_2 & b_2 - r_{11}(b_2-f_2^L)w_1  &b_1 - r_{21} (b_1-f_1^L)w_2 \\
\end{bmatrix}
\]
and the constant term is $c_j = f_1^Uf_2^L - \bigl(\alpha_{j1} + \alpha_{j2} p_{21} + \alpha_{j3}f_1^U + \alpha_{j4}f_2^L\bigr)$. Similarly, the concave envelope of $f_1f_2$ over $P$ is given by $\min_{k=1,\ldots,6} \bigl\{\bigl\langle \beta_k, (x,t) \bigr\rangle + d_k \bigr\}$, where $\beta_k$ denotes the $k^{\text{th}}$ row of the matrix:
\[
\begin{bmatrix}
	0 &  0  & f_2^U & f_1^L \\
	0 &  0 & f_2^L & f_1^U \\
	(f_2^L-f_2^U)w_1 &0 & f_2^L -  r_{11}(f_2^L -f_2^U)w_1 & b_1 \\
	0 & ( f_1^L - f_1^U ) w_2 & b_2	 &  f_1^L -  r_{21}(f_1^L -f_1^U)w_2   \\
	(f_2^L-b_2)w_1 & (b_1-f_1^U)w_2 & b_2 - r_{12}(f_2^L-b_2)w_1  &  b_1 - r_{21} (b_1-f_1^U)w_2 \\
		(b_2-f_2^U)w_1 & (f_1^L-b_1)w_2 & b_2 - r_{11}(b_2-f_2^U)w_1  &  b_1 - r_{22} (f_1^L-b_1)w_2 \\
\end{bmatrix},
\]
with the constant term $d_k = f_1^Uf_2^U - \bigl(\beta_{k1} + \beta_{k2} + \beta_{k3}f_1^U + \beta_{k4}f_2^U\bigr)$.
\end{theorem}
The convex hull description comprises of $12$ non-trivial linear inequalities. Notably, the composite relaxation also yields $12$ inequalities when relaxing the product $f_1(x_1)f_2(x_2)$ using a single underestimator for each $f_i(x_i)$ (see Theorems~1 and 2 in \cite{he2021new}). However, the inequalities derived in Theorem~\ref{thm:pentagon} are strictly tighter and result in improved bounds when included in the relaxation. Building on this insight, we subsequently demonstrate that the substantial gains of composite relaxations over traditional factorable programming methods, as highlighted in \cite{he2024mip}, can be further enhanced using the techniques presented here. Theorem~\ref{thm:pentagon} is therefore of independent interest, providing inequalities that strictly dominate those in Theorems~1 and 2 of \cite{he2021new} and underpin the benefits that our relaxation provides relative to composite relaxation over rectangular domains. Further benefits for non-rectangular domains are realized in Section~\ref{sec:prod-over-nonbox}.

Instead of restricting the outer-approximation of the graph to a pentagon, we now consider a general polytope $P_i$. Even in this generalized setting, the resulting relaxation remains polyhedral.

\begin{proposition}\label{prop:disjbil}
    For $i =1,2$, let $P_i$ be a polytope that outer-approximates the graph of $f_i(\cdot)$. Then,  a polyhedral relaxation of $\G$ is given by the convex hull of the set $\point$, defined as
    \[
    \point := \bigl\{(x,t,\mu)\bigm| \mu=t_1 t_2,\ (x_i,t_i)\in \vertex(P_i) \for i=1,2 \bigr\}.
    \]
\end{proposition}

\begin{proof}
Since $P_i$ is an outer-approximation of the graph of $f_i(\cdot)$, it follows readily that a convex relaxation for  $\G$ is given as follows:
 \[
  \Rlx:= \conv\bigl\{(x,t, \mu) \bigm|   \mu = t_1t_2,\ (x_i,t_i) \in P_i \for i = 1,2   \bigr\}. 
 \]
 By Theorem 8 of~\cite{tawarmalani2002convex}, we conclude that $\Rlx = \conv(\point)$. 
 \end{proof}
 
Proposition~\ref{prop:disjbil} requires computing the convex hull of finitely many points. Since these points reside in a five-dimensional space, we can leverage low-dimensional convex hull algorithms, such as \texttt{QHull}~\cite{barber1996quickhull}, to efficiently obtain the corresponding hyperplane description. Furthermore, when the functions $f_i(\cdot)$ do not appear elsewhere in the formulation, the auxiliary variables $t_i$ can be projected out. In this case, it suffices to compute the convex hull in the $(x,\mu)$ space, namely,
\[
\conv\bigl\{(x,\mu) \bigm| \mu = t_1t_2,\ (x_i,t_i) \in \vertex(P_i) \for i =1,2 \bigr\},
\]
which corresponds to a three-dimensional convex hull computation and is therefore substantially more tractable than computing the convex hull of $\point$ in a lifted space as in Theorem~\ref{thm:pentagon}.  Beyond computational efficiency, an additional advantage of employing computational geometry tools is that they yield an irredundant facet description. Consequently, the resulting relaxation constitutes a minimal formulation, containing no redundant inequalities. For instance, in the context of Example~\ref{ex:comparison-oneestimator}, Theorem~\ref{thm:pentagon} yields six inequalities, while \texttt{QHull} produces only three, resulting in a more compact relaxation. 

By leveraging computational geometry tools, one can construct a sequence of polyhedral relaxations through the progressive refinement of the outer-approximations for each univariate function $f_i(\cdot)$. Indeed, Theorem~\ref{thm:convergence} formally establishes that this sequence converges to the convex hull of $\graph(f_1f_2)$. We illustrate this convergence by revisiting Example~\ref{ex:comparison-oneestimator}. Starting with the outer-approximation defined in~\eqref{eq:examplePi} and depicted in Figure~2(a), the resulting relaxation is shown in Figure~\ref{fig4}(b). As additional gradient inequalities are incorporated to approximate $x_i^2$, Figures~\ref{fig4}(c)–(e) demonstrate that the relaxations become progressively tighter, with the final relaxation closely approximating the convex hull.

\begin{figure}[ht]
\centering

\begin{subfigure}[t]{0.45\linewidth}
\centering
\includegraphics{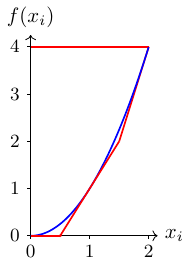}
\caption{}
\end{subfigure}
\hfill
\begin{subfigure}[t]{0.45\linewidth}
\centering
\includegraphics[page=1]{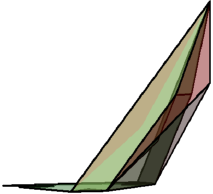}
\caption{}
\end{subfigure}

\begin{subfigure}[t]{0.3\linewidth}
\centering
\includegraphics[page=1]{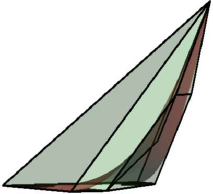}
\caption{Number of gradients: $3$}
\end{subfigure}
\hfill
\begin{subfigure}[t]{0.3\linewidth}
\centering
\includegraphics[page=1]{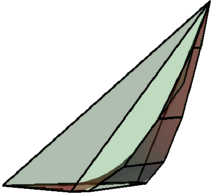}
\caption{Number of gradients: $5$}
\end{subfigure}
\hfill
\begin{subfigure}[t]{0.3\linewidth}
\centering
\includegraphics[page=1]{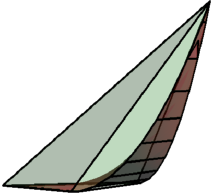}
\caption{Number of gradients: $10$}
\end{subfigure}

\caption{Approximation of $x_1^2x_2^2$ via piecewise linear relaxations}
\label{fig4}
\end{figure}

Finally, we formally demonstrate that our relaxation yields  a tighter relaxation than the composite relaxation (CR) in~\cite{he2021new}. CR applies to a composite function $\phi \mcirc f: [0,1]^d \to \R$  defined as 
\[
(\phi \mcirc f)(x) = \phi \bigl(f_1(x_1), \ldots, f_d(x_d) \bigr) \quad \for x \in [0,1]^d,
\] 
where each inner-function $f_i(\cdot)$ is univariate. CR assumes that, for each $f_i(\cdot)$, there exists a vector of underestimating functions $u_i(\cdot)$. It exploits ordering relations among the estimating functions $u_{i}(\cdot)$ and their upper bounds $a_i$ by embedding them into a polytope $U_i$, and then convexifies the outer-function $\phi(\cdot)$ over the product polytope $U := \prod_{i=1}^d U_i$. Specifically, for each $i \in [d]$, assume a pair $(u_i(x_i),a_i)$ such that for all $x_i \in [0,1]$,
\begin{equation}\label{eq:ordering-OR}
	 a_{i0} \leq a_{i1} \leq \cdots \leq a_{in_i}, \quad u_{ij}(x_i) \leq  \min \bigl\{f_i(x_i), a_{ij}\bigr\} \quad \text{ for } j = 0, \ldots, n_i.
\end{equation}
That is, each $u_{ij}(\cdot)$ underestimates $f_i(\cdot)$ and is bounded above by $a_{ij}$. The composite relaxation introduces auxiliary variables $u_i$ and $f_i$ to represent $u_i(x_i)$ and $f_i(x_i)$ respectively, and encodes the ordering constraints into the polytope $U_i$ defined as follows:
\[
U_i := \bigl\{(u_i,f_i) \bigm| f_i^L \leq f_i \leq f_i^U,\ u_{ij} \leq  a_{ij},\ u_{ij} \leq f_i \text{ for } j  \in \{0, \ldots, n_i\} \bigr\},
\]
where $f_i^L$ (resp. $f_i^U$) is a lower (resp. upper) bound of $f_i(\cdot)$.  Let $U := \prod_{i=1}^dU_i$ and define $\Phi^U := \bigl\{(f,u,\mu) \bigm| \mu = \phi(f),(f,u) \in U\bigr\}$. Given a pair $\bigl(u(x),a\bigr)$, where $u(x) = \bigl(u_1(x_1), \ldots, u_d(x_d)\bigr)$ and $a = (a_1, \ldots, a_d)$, satisfying~\eqref{eq:ordering-OR}, the composite relaxation is defined as 
\[
\mathcal{C}\mathcal{R}\bigl(u(x),a\bigr) := \bigl\{(x,f,\mu,u) \bigm| (f,u,\mu) \in \conv(\Phi^U),\ u \geq u(x) \bigr\}.
\]
In contrast, our approach avoids introducing variables for the estimating functions. Instead, it uses the underestimators $u_{ij}(\cdot)$ to construct a polyhedral outer-approximation of the graph of $f_i(\cdot)$:
\begin{equation}\label{eq:cr-oa}
P_i := \bigl\{(x_i,f_i) \bigm| f_i^L \leq f_i \leq f_i^U ,\ x_i \in [0,1],\ u_{ij}(x_i) \leq f_i \text{ for } j \in \{0, \ldots, n_i\} \bigr\}.
\end{equation}
Let $P := \prod_{i \in [d]}P_i$. Convexification of the outer-function $\phi(\cdot)$ is then performed directly over $P$ in the space of $(x,f)$ variables, whereas the composite relaxation convexifies $\phi(\cdot)$ over $U$ in the higher-dimensional space of $(f,u)$ variables. Next, we show that our relaxation is tighter than the composite relaxation. 
\begin{proposition}~\label{prop:improve_CR}
Let $\bigl(u(x),a\bigr)$ be a pair satisfying~(\ref{eq:ordering-OR}). Define the set 
\[
\Phi^P:=\bigl\{(x,f,\mu) \bigm| \mu = \phi(f), (x_i,f_i) \in P_i\; \for i \in [d]\bigr\},
\]
where $P_i$ defined as in~(\ref{eq:cr-oa}). Then, 
\[\conv\bigl(\Phi^P\bigr) \subseteq \proj_{(x,f,\mu)}\bigl( \mathcal{C}\mathcal{R}\bigl(u(x),a\bigr)\bigr).\]
\end{proposition}
\begin{proof}
It suffices to demonstrate that $\Phi^P \subseteq \proj_{(x,f,\mu)}\bigl( \mathcal{C}\mathcal{R}\bigl(u(x),a\bigr)\bigr)$, given that the latter set is convex. Consider an arbitrary point $(x, f, \mu) \in \Phi^P$ and let $u = u(x)$. By construction, $(f,u) \in U$ and $\mu = \phi(f)$, which implies $(f,u,\mu) \in \Phi^U \subseteq \conv(\Phi^U)$. Since $u = u(x)$ also holds, the tuple $(x,f,\mu,u)$ belongs to $\mathcal{C}\mathcal{R}\bigl(u(x),a\bigr)$. Consequently, the projection $(x,f,\mu)$ lies within $\proj_{(x,f,\mu)}\bigl( \mathcal{C}\mathcal{R}\bigl(u(x),a\bigr)\bigr)$.
 \end{proof}

\subsection{Multiplication of Variables over an Arbitrary Domain}\label{sec:bilinear-over-nonbox}

In this subsection, we construct relaxations for the product of two variables over an arbitrary bounded domain  $D \subset \R^2$. Specifically, we consider the graph:
\begin{equation}\label{eq:bilinear-over-nonbox}
\G : =\bigl\{(x,\mu) \in  \R^3 \bigm| \mu = x_1x_2,\ x \in D  \bigr\}.    
\end{equation}
A prevalent approach relies on the McCormick envelope to relax the bilinear term over a rectangular bounding box of $D$. However, such relaxations are often weak when the geometry of $D$ deviates significantly from a rectangle. Our objective is to demonstrate how \textit{voxelization} can be leveraged to generate a sequence of polyhedral relaxations that are tighter than standard McCormick relaxations and converge to the convex hull of $\G$.  

In computational geometry, voxelization is the process of converting a geometric object into a set of discrete, grid-aligned volumetric elements known as \textit{voxels}. Within our framework, the standard McCormick relaxation of $\G$ can be interpreted as a trivial voxelization of the domain $D$, wherein the entire domain is approximated by a single global voxel. We make this observation explicit as follows:
\begin{remark}\label{rmk:mcc}
    Given a bounded region $D \subset \R^2$, we approximate it with a rectangular bounding box $[x_1^L,x_1^U] \times [x_2^L,x_2^U]$. We then construct the convex hull of the bilinear term over this voxel. This procedure is equivalent to computing the convex hull of the bilinear term evaluated at the four vertices of the box:
    \[
    \conv\bigl\{(x, x_1x_2) \bigm|  x_i \in \{x_i^L,x_i^U\} \for i =1,2 \bigr\}. 
    \]
Clearly, the resulting convex hull is a valid polyhedral relaxation of $\G$, if the voxel outer-approximates $D$. \hfill \qed 
\end{remark}

By systematically excising rectangles from the initial bounding box, we obtain a refined voxelization. This approach yields tighter relaxations by computing the convex hull of the bilinear term over the resulting voxelized domain.  We illustrate this procedure with the following example.

\begin{example}\label{ex:bilinear-nonbox}
Consider a bilinear term $x_1x_2$ defined over a nonconvex domain $D$, given by:
\[
D := \bigl\{x \in \R^2 \bigm| x_2 \leq 1+\log(x_1),\ x_2 \leq 1-\log(x_1),\ x_1 \in [\mathit{e}^{-1},\mathit{e}], x_2\geq 0 \bigr\}.
\]
In FP relaxation, the domain $D$ is approximated by its bounding box $[\mathit{e}^{-1},\mathit{e}]\times [0,1]$, and the bilinear term is relaxed over this box. This single-voxel approach is depicted in the left panel of Figure~\ref{fig:bilinear-nonbox}. As shown, this relaxation is significantly looser than the convex hull of the bilinear term over $D$. Observe that $x_1x_2$ is shown as the red region in Figure~\ref{fig:bilinear-nonbox}. 

The relaxation gap can be systematically reduced by refining the underlying voxelization. For instance, by removing four rectangular regions from the initial bounding box, and computing the convex hull of the bilinear term over the resulting refined domain, we obtain a tighter relaxation. This refinement is depicted in the middle panel of Figure~\ref{fig:bilinear-nonbox}. By continuing this process, the voxelized region increasingly conforms to the geometry of $D$, and the resulting relaxation converges toward the true convex hull, as illustrated in the right panel of  Figure~\ref{fig:bilinear-nonbox}.

Although these three relaxations vary in tightness, they share a crucial structural property: they are all polyhedral. Specifically, the convex hull in each case is generated by a finite set of points. As we will establish later, these generators correspond precisely to the corner points of the voxelized regions, which are highlighted in blue in Figure~\ref{fig:bilinear-nonbox}.  \hfil \qed

\newcommand{\colw}{0.3\textwidth}  %
\newcommand{\imgw}{\linewidth}      %
\newcommand{\vsep}{0.8em}           %

\begin{figure}[t]
\centering

\begin{tabular}{@{}c c c@{}}
\begin{subfigure}[t]{\colw}\centering
\includegraphics{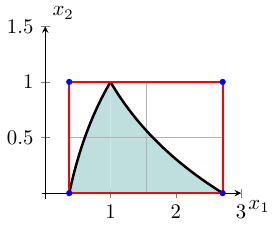}

\end{subfigure}
&
\begin{subfigure}[t]{\colw}\centering
\includegraphics{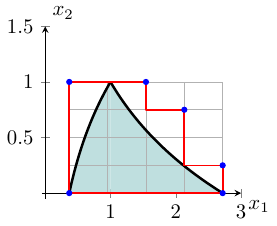}

\end{subfigure}
&
\begin{subfigure}[t]{\colw}\centering
\includegraphics{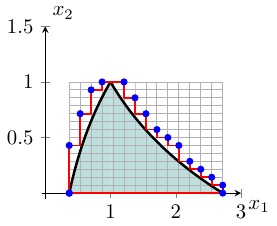}

\end{subfigure}
\\[\vsep]
\begin{subfigure}[t]{\colw}\centering
\includegraphics[page=1]{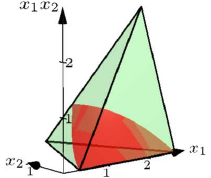}

\end{subfigure}
&
\begin{subfigure}[t]{\colw}\centering
\includegraphics[page=1]{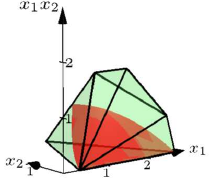}

\end{subfigure}
&
\begin{subfigure}[t]{\colw}\centering
\includegraphics[page=1]{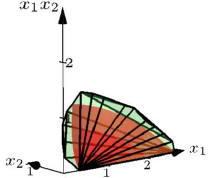}

\end{subfigure}
\end{tabular}

\caption{Refinement of the bilinear relaxation via voxelization. The top row illustrates the domain $D$ (teal) and its axis-aligned outer approximation (red boundary) with increasing levels of accuracy. The bottom row depicts the corresponding three dimensional polyhedral relaxations (green). As the voxelization is refined from left to right, the polyhedral relaxation converges toward the true convex hull of the bilinear term over $D$.}
\label{fig:bilinear-nonbox}
\end{figure}
\end{example}

To formalize the construction presented in Example~\ref{ex:bilinear-nonbox}, we introduce several definitions. First, to ensure that the resulting relaxations remain polyhedral, we require voxelizations to be axis-aligned: this property is shared by the  three voxelizations illustrated in the preceding example. 
\begin{definition}
    A  planar region $\AxisAligned$ is \textit{axis-aligned} if it can be expressed as a finite union of rectangles, each with sides parallel to the coordinate axes.  
\end{definition}
Throughout this paper, we denote by $\rects(\AxisAligned)$ the set of constituent rectangles defining $\AxisAligned$, and by $\disc(\AxisAligned)$ the set of vertices associated with the rectangles in $\rects(\AxisAligned)$. We remark that for a given axis-aligned region, $\disc (\AxisAligned)$ is not unique. However, we do not concern ourselves with this issue since any such representation will suffice for our purpose. We now formally define the specific subset of points that serve as generators for the convex hulls illustrated in Example~\ref{fig:bilinear-nonbox}. Although $\disc(\AxisAligned)$ encompasses all vertices of the underlying rectangles, the relaxation is ultimately determined by a specific subset of these points, which we refer to as \textit{corner points}.

\begin{definition}
    A point $(v_1,v_2)$ is a \textit{corner point} of a planar axis-aligned region $\AxisAligned$ if  $v_i$ is an extreme point of $\conv\{(x_1,x_2)\mid x\in \AxisAligned, x_{i'} = v_{i'}, i'\in \{1,2\}\setminus\{i\}\}$. The set of all corner points of $\AxisAligned$ is denoted by  $\corner(\AxisAligned)$. 
\end{definition}
It is important to note that every corner-point of $\AxisAligned$ must belong to the set of discretization points $\disc(\AxisAligned)$. Indeed, if a point is not a vertex of any constituent rectangles, it lies in the relative interior of some face $F$ of a rectangle. Since $F$ has dimension at least one, there exists a line segment parallel to the one of two coordinate axes that lies within $F$ and contains the point in its relative interior.  Furthermore, our definition of corner points differs from that of extreme points of polytopes, by restricting the allowable directions of such line segments. For instance, applied to the polytope $\bigl\{x\in \R^2_+\mid x_1+x_2\le 1\bigr\}$, our definition would identify all points on the line segment joining $(1,0)$ to $(0,1)$ as corner points.

We now present how the bilinear graph $\G$ can be relaxed over any bounded domain $D$. 
\begin{proposition}\label{prop:bilinear-over-nonbox}
Let $\AxisAligned \subset \R^2$ be an axis-aligned outer-approximation of the domain $D$, and define $\Rlx := \conv\bigl\{(x, x_1x_2) \bigm| x \in \AxisAligned \bigr\}$. Then, $\Rlx$ is a polyhedral relaxation of $\G$ that coincides with the convex hull of $\point$, where
    \[
\point:= \bigl\{(x,\mu) \bigm| \mu =x_1x_2,\ x \in \corner(\AxisAligned)\bigr\}.
    \]
\end{proposition}
Before proceeding to the proof, we discuss its primary technical challenge. Given an arbitrary point $v = (v_1,v_2)$ in an axis-aligned region $\AxisAligned$, we must demonstrate that $(v,v_1v_2)$ can be expressed as a convex combination of points in $\point$. A natural approach is an \textit{iterative coordinate-wise convex decomposition}. Specifically, by fixing the second coordinate, one can express $(v_1,v_2,v_1v_2)$ as a convex combination of $(v^L_1,v_2,v^L_1v_2)$ and $(v^U_1,v_2,v^U_1v_2)$, where $v^L_1$ and $v^U_1$ denote the endpoints of the line segment forming the slice  $\{x\mid (x,v_2)\in \AxisAligned\}$. By iteratively alternating the fixed coordinate, one might expect to eventually reach the corner points of $\AxisAligned$. However, as illustrated in Figure~\ref{fig:loop}, this greedy coordinate-wise approach may enter a cycle. In the example shown, resolving the point $v=(2,1)$ through successive coordinate-wise decompositions can return to the starting point after a finite sequence of steps, failing to terminate at the corner points. Consequently, our proof must employ a more global argument to establish the representation. 

\begin{figure}[ht]
\centering
\includegraphics{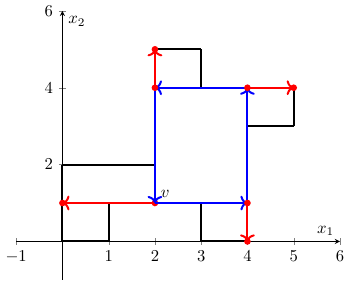}
\caption{Illustration of a cycle (indicated in blue) in the iterative coordinate-wise decomposition. Starting from point $v$, alternating decompositions along the $x_1$ and $x_2$ axes return to $v$ after four iterations, demonstrating that the greedy procedure may fail to yield a finite representation via corner points.}
\label{fig:loop}
\end{figure}

To circumvent the cyclic behavior illustrated in Figure~\ref{fig:loop}, we abandon the greedy deterministic iterative decomposition in favor of one that expresses the initial point $v$ as the expectation of the corner points under a probability distribution. Specifically, our proof constructs a Markov chain to track the evolution of the decomposition process. Within this framework, the probability distribution at each time step is interpreted as a convex decomposition of $v$.  By analyzing the limiting behavior of the Markov chain, we establish that the limiting distribution yields a valid representation of $v$ as a convex combination of corner points in $\AxisAligned$.

For completeness, we briefly review key concepts of Markov chain theory, see~\citep[Chapter 11]{grinstead2012introduction} for further details.  A Markov chain is defined by a finite \textit{state space} $\state$ and a \textit{transition matrix} $\transition$. A sequence of random variables $(X_1, X_2, X_3, \ldots)$ constitutes a Markov chain with state space $\state$ and  transition matrix $\transition$ if the \textit{Markov property} holds. Specifically, for all $x, y$, all $t \geq 2$, and any history event $H_{t-1} = \cap_{s=1}^{t-1}\{X_s = x_s\}$  such that $\prob(H_{t-1} \cap \{X_t = x\}) > 0$, we have 
\[
\prob\bigl\{X_{t+1} = y \bigm| H_{t-1} \cap \{X_t = x\}\bigr\} = \prob\{X_{t+1} =y \mid X_t = x\} = \transition(x,y). 
\]
A state $x$ is termed \textit{absorbing} if the chain cannot leave it; that is, $\transition(x,x) = 1$. A Markov chain is absorbing if it possesses at least one absorbing state and if every state can reach an absorbing state potentially over multiple steps. In such a chain, a non-absorbing state is called \textit{transient}. Let $\transition^t$ denote the $t$-step transition matrix, where the entry $\transition^t(x,y)$ represents the probability of being in state $y$ after $t$ steps, given an initial state $x$.  The transition matrix of an absorbing Markov chain is written in canonical form as:
\[
T = \begin{pmatrix}
    Q & R \\
    0 & I
    \end{pmatrix},
\]
where $Q$ describes transitions among transient states, $R$ describes transitions from transient to absorbing states, and $I$ is the identity matrix corresponding to the absorbing states. The long-term behavior of an absorbing Markov chain is characterized by the limiting transition matrix $\transition^\limiting$:
\begin{equation}\label{eq:limitingT}
    \transition^\limiting:= \lim_{t\rightarrow \infty} \transition^t =  \begin{pmatrix}
    0 & (I-Q)^{-1}R \\
    0 & I
    \end{pmatrix},
\end{equation}
where the matrix $N = (I - Q)^{-1}$ is known as the \textit{fundamental matrix}. The entries of the product $B = NR$ give the probabilities of absorption into each specific absorbing state, conditioned on starting from a transient state.

To leverage Markov chain theory, we must formally specify a state space and a transition matrix within our geometric context. To this end, we propose the iterative procedure detailed in Algorithm~\ref{alg:2d-markov}. Given an axis-aligned region $\AxisAligned$, we define the initial state space $\state$ as the union of the corner points of the constituent rectangles of $\AxisAligned$, denoted by $\disc(\AxisAligned)$. 
For each point $v$ in $\state$, we invoke the procedure $\lineends(v,\AxisAligned)$ to decompose $v$ into a pair of endpoints $(v_l,v_r)$ along one of the coordinate axes. 
Specifically, let $i$ be the smallest index in $\{1,2\}$ and let $j = \{1,2\}\setminus \{i\}$ be such that $v_i$ is not an extremum of the slice $\{x_i  \mid x \in \AxisAligned,\ x_j = v_j\}$. We define the endpoints  $v_l$ and $v_r$ as the extrema of this slice:
\begin{align*}
    v_{l} =\arg\min \{x_i  \mid x \in \AxisAligned,\ x_j = v_j  \} \quad \text{ and } \quad v_{r} = \arg\max \{x_i \mid x \in \AxisAligned,\ x_j = v_j  \}. 
\end{align*}
In this case, $v_l$ and $v_r$ are added to the state space $\state$ (if not already present), and the transition probabilities from $v$ are determined by the convex multipliers specified in line 13 of Algorithm~\ref{alg:2d-markov}. 
If no such index $i$ exists, we set $v_l = v_r = v$ and assign a transition probability of one from $v$ to itself. 
We note that a point is a corner point if and only if it is coordinate-wise extremal; that is, for a corner point $(v_1, v_2)$, $v_1$ is an extreme point of $\{x \mid (x,v_2) \in \AxisAligned\}$ and $v_2$ is an extreme point of $\{y \mid (v_1,y) \in \AxisAligned\}$.
Consequently, when $v_l = v_r = v$, we identify $v$ as a corner-point of $\AxisAligned$. 

\begin{algorithm}[ht]
\caption{A Markov chain representation for coordinate-wise convex decompositions over planar axis-aligned regions}\label{alg:2d-markov}
    \begin{algorithmic}[1]
        \Require An axis-aligned region $\AxisAligned$
\Ensure  A Markov chain with state space $\state$ and transition matrix $\transition$
        \State Initialize state space $\state \gets \disc(\AxisAligned)$
        \State Initialize unprocessed stack $\mathcal{U}\gets \state$
        \While{$\mathcal{U}$ is not empty}
        \State\label{algstep:simplePop} $v \gets \text{pop}(\mathcal{U})$
        \State\label{algstep:findEnds} $(v_l,v_r) \gets \lineends(v,\AxisAligned)$
            \If{$v_l \neq v_r$}
                
        \For{$u\in \{v_l,v_r\}$}
        \If{$u\notin \state \cup \mathcal{U}$}\label{algstep:unotstate}
        \State $\state \gets \state \cup \{u\}$
        \State\label{algstep:addU} Push $u$ onto the stack $\mathcal{U}$
        \EndIf
        \EndFor
        \State Find $\lambda \in [0,1]$ such that $v = \lambda v_l + (1- \lambda) v_r$
        \State Set $\transition(v,v_l) \gets \lambda$ and  $\transition(v,v_r) \gets 1- \lambda$
    \Else
                \State Set $\transition(v,v) \gets 1$
    \EndIf
        \EndWhile
    \end{algorithmic}
\end{algorithm}

\begin{figure}[htbp]
\centering

\begin{subfigure}[t]{0.45\linewidth}
\includegraphics{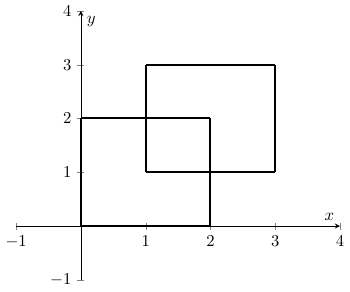}
\caption{Axis-aligned region $\AxisAligned$}\label{fig:a}
\end{subfigure}
\begin{subfigure}[t]{0.45\linewidth}
    \centering
    \includegraphics{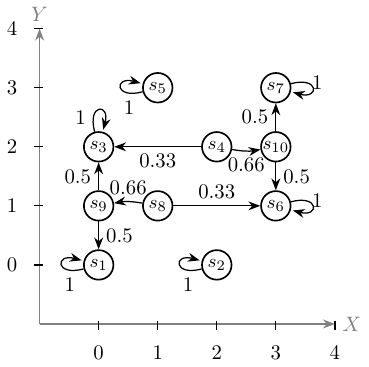}
\caption{Corresponding Markov chain}\label{fig:b}
\end{subfigure}
\caption{Illustration of the convex extension Markov chain: (a) The axis-aligned region $\AxisAligned$ and (b) the resulting stochastic transitions.}
\label{fig:markovchainillustration}
\end{figure}

In the following lemma, we prove that Algorithm~\ref{alg:2d-markov} terminates after finitely many iterations, yielding a finite Markov chain. Subsequently, we analyze its limiting behavior. The core of the proof, detailed in Appendix~\ref{app:bilinear-extension}, relies on the observation that corner points of $\AxisAligned$ constitute absorbing states within the Markov chain. Consequently, the existence of a limiting transition matrix $\transition^\limiting$ ensures that every state can ultimately be expressed as a convex combination of these corner-points. Notably, we establish this result for a general bilinear function over $\AxisAligned$, thereby enabling its application to the more general setting discussed in  Section~\ref{sec:prod-over-nonbox}, where we consider the product of two functions over an arbitrary domain. 
\begin{lemma}\label{lemma:bilinear-extension}
  Given an axis-aligned region $\AxisAligned \subseteq \R^2$, Algorithm~\ref{alg:2d-markov} terminates in finitely many iterations, yielding an absorbing Markov chain $(\state,\transition)$ for which the corner points of $\AxisAligned$ serve as the absorbing states. Moreover, for any bilinear function $\ell: \AxisAligned \to \R$, we have that for  every $x \in \state$
    \[
    \bigl(x,\ell(x)\bigr) = \sum_{v\in\corner(\AxisAligned)}\transition^\limiting(x,v) \cdot \bigl(v, \ell(v) \bigr),
    \]
    where  $\transition^\limiting$ denotes the limiting transition matrix defined as in~\eqref{eq:limitingT}. 
\end{lemma}

With this lemma, we proceed to the proof of Proposition~\ref{prop:bilinear-over-nonbox}.

\begin{proof}[Proof of Proposition~\ref{prop:bilinear-over-nonbox}]
First, observe that $\Rlx$ is a convex relaxation of $\G$, since $D \subseteq \AxisAligned$. We now demonstrate that $\Rlx = \conv(\point)$, thereby establishing the polyhedrality of the relaxation. The inclusion $\conv(\point) \subseteq \Rlx$ follows directly from the facts that $\point \subseteq  \Rlx$ and $\Rlx$ is a convex set. To prove the reverse inclusion, it suffices to show that any point $(x,x_1x_2)$ with $x \in \AxisAligned$ can be expressed as a convex combination of points in  $\point$. Let $H$ be a rectangle in  $\rects(\AxisAligned)$ containing  $x$. By the convex extension property of bilinear term over axis-aligned rectangles~\cite{rikun1997convex,tawarmalani2002convex},  we have \[(x,x_1x_2) = \sum_{v \in \vertex(H)} \lambda_v(v,v_1v_2),\] where $\{\lambda_v\}$ are convex multipliers. Since $\vertex(H) \subseteq \disc(\AxisAligned) \subseteq \state$, we may invoke Lemma~\ref{lemma:bilinear-extension} to write
\[
(x,x_1x_2) = \sum_{v \in \vertex(H)} \lambda_v \sum_{w \in \corner(\AxisAligned)} \transition^\limiting(v,w)(w,w_1w_2). 
\]
Thus, $(x,x_1x_2)$ is a convex combination of points in $\point$.
\end{proof}

\subsection{Multiplication of Functions over an Arbitrary Domain}\label{sec:prod-over-nonbox}
In this subsection, we address both challenges discussed in Sections~\ref{sec:prod-over-box} and~\ref{sec:bilinear-over-nonbox} simultenaeously. Specifically, we consider the graph of the product of two univariate functions:
\[
\G :=\bigl\{(x,\mu) \in  \R^3 \bigm| \mu = f_1(x_1)f_2(x_2),\ x \in D  \bigr\},    
\]
where each $f_i:[x_i^L, x_i^U] \to \R$ is an arbitrary nonlinear function and $D$ is a subset of $[x^L,x^U]$. To obtain a polyhedral relaxation for $\G$, the prevalent FP approach introduces auxiliary variables $t_i$ to represent each univariate function $f_i(x_i)$, subsequently relaxing the graph of $f_i(\cdot)$ and the bilinear term $t_1t_2$ over the domain $D$ independently. In contrast, our approach simultaneously handles the nonlinearities and the nonconvexity of the domain. By integrating the outer-approximation of the univariate functions with the voxelization of the domain, we construct a sequence of polyhedral relaxations that converges to the convex hull of $\G$.

\begin{example}
    Combining the two preceding examples, consider $x_1^2x_2^2$ over the following domain:
    \begin{alignat*}{2}
        &&&x_2 \leq 1+\log(x_1)\\
        &&&x_2 \leq 1-\log(x_1)\\
        &&&x_1 \in [e^{-1},e],\quad x_2\geq 0 .
    \end{alignat*}
    We construct a piecewise linear underestimator $u_i(\cdot)$ for $x_i^2$ using its gradients, while employing linear interpolation as the overestimator $o_i(\cdot)$. The breakpoints of piecewise linear functions divide $[e^{-1},e] \times [0,1]$ into rectangular boxes. After removing boxes outside the domain, we obtain an axis-aligned outer-approximation for the domain. Over this axis-aligned region, we convexify $t_1t_2$ with $u_i(x_i) \leq t_i \leq o_i(x_i)$ for $i = 1,2$. Figure~\ref{fig6} demonstrates that the relaxation is polyhedral and converges to the convex hull of the graph. This convergence is achieved by simultaneously refining the piecewise-linear estimators of $x_i^2$ and the voxelization of the domain. \hfil \qed

\newcommand{\colw}{0.3\textwidth}  %
\newcommand{\imgw}{\linewidth}      %
\newcommand{\vsep}{0.8em}           %

\begin{figure}[t]
\centering

\begin{tabular}{@{}c c c@{}}
\begin{subfigure}[t]{\colw}\centering
\includegraphics[page=1,width = 0.9\linewidth]{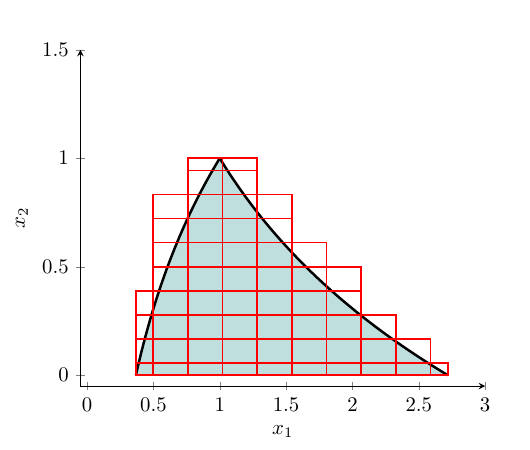}
\end{subfigure}
&
\begin{subfigure}[t]{\colw}\centering
\includegraphics[page = 1,width = 0.9\linewidth]{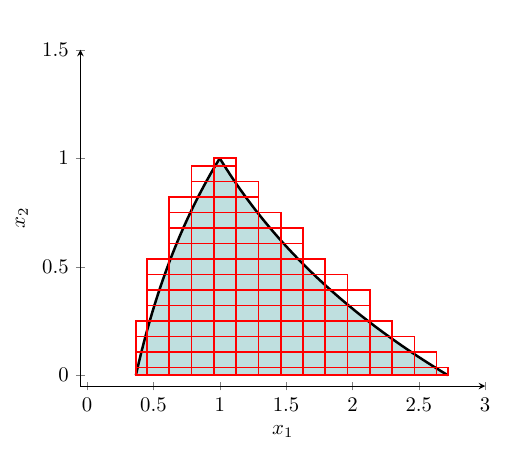}
\end{subfigure}
&
\begin{subfigure}[t]{\colw}\centering
\includegraphics[page = 1,width = 0.9\linewidth]{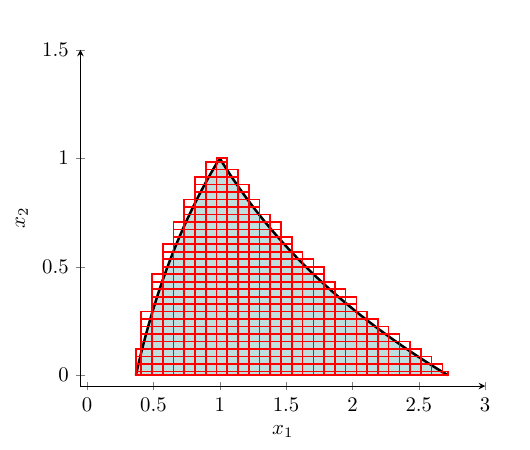}
\end{subfigure}
\\[\vsep]
\begin{subfigure}[t]{\colw}\centering
\includegraphics[page=1]{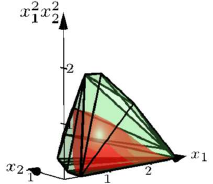}

\end{subfigure}
&
\begin{subfigure}[t]{\colw}\centering
\includegraphics[page=1]{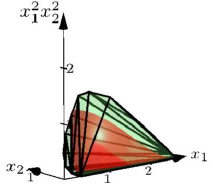}

\end{subfigure}
&
\begin{subfigure}[t]{\colw}\centering
\includegraphics[page=1]{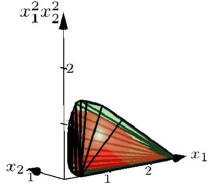}

\end{subfigure}
\end{tabular}

\caption{Multiplication of functions over a nonconvex domain}
\label{fig6}
\end{figure}
\end{example}

We now formalize this construction. Let $u_i(\cdot)$ and $o_i(\cdot)$ denote piecewise affine functions that underestimate and overestimate the univariate function $f_i(\cdot)$ over the interval $[x_i^L,x_i^U]$, respectively.\footnote{While the individual affine segments provide local bounds, we allow them to be discontinuous at breakpoints. Since we typically restrict $x_i$ to a region where $u_i(x_i)$ and $o_i(x_i)$ are both affine in the interior, with a slight abuse of notation, even at the end-points of such intervals, we evaluate $u_i(x_i)$ and $o_i(x_i)$ using the same estimator. When multiple rectangles containing $x$ are in the consideration set, we will assume that the smallest underestimator and the largest overestimator over the rectangles containing $x$ are used to evaluate $u_i(x_i)$ and $o_i(x_i)$ respectively.} 
Let $\AxisAligned$ be  an axis-aligned outer-approximation of the domain $D$. Then, we obtain the following convex relaxation:
\begin{equation}\label{eq:Rlx}
 \Rlx := \conv\bigl\{(x,t_1t_2)\bigm|  x \in \AxisAligned,\ u_i(x_i)\leq t_i\leq o_i(x_i) \, \for  i = 1,2 \bigr\}. 
\end{equation}
As in our previous constructions, we aim to represent the relaxation $\Rlx$ as the convex hull of a finite set of points. To achieve this, we integrate the breakpoints of piecewise-linear estimators with the corner-points of the axis-aligned region $\AxisAligned$. The breakpoints of the estimators divide $[x^L,x^U]$ into rectangular cells that form its cover\footnote{A collection of subsets of a set $S$ that include each point of $S$ in at least one element}. Specifically, let $\D$ be a rectangular cover of $[x^L,x^U]$ such that for every rectangle $H \in \D$, the four piecewise linear functions $u_i(\cdot)$ and $o_i(\cdot)$ ($i = 1,2$) are affine over $H$. Using this cover $\D$, we lift a finite set of points in $\AxisAligned$ to define:
\[
\point := \bigcup_{H \in \D} \Bigl\{(x,t_1t_2) \Bigm|  x \in \corner(\AxisAligned \cap H),\ t_i \in \bigl\{u_i(x_i), o_i(x_i)\bigr\} \for i = 1,2 \Bigr\}. 
\]
Given a set $S \subseteq [x^L,x^U]$ and the cover $\D$ of $[x^L,x^U]$, we construct a cover of $S$ by intersecting each element of $\D$ with $S$. This yields a valid cover because $S = S\cap [x^L,x^U] = S\cap \bigcup_{H\in \D} H = \bigcup_{H\in \D} S\cap H$, where the first equality holds by the definition of $S$, second because $\D$ is a cover of $[x^L,x^U]$, and the third by the distributive law. This motivates our next definition.

\begin{definition}[grid cover]
Let $G$ be a grid in $\mathbb{R}^n$ and let $\mathcal{O}\subseteq\mathbb{R}^n$. The grid cover of $\mathcal{O}$ induced by $G$ is
\[
\mathcal{P}=\{\, v\cap\mathcal{O}\;:\; v\text{ is a cell of }G,\ v\cap\mathcal{O}\neq\emptyset \,\}.
\]
\end{definition}

As an example, consider the shaded axis-aligned region depicted in Figure~\ref{fig:removeRectangle} obtained by removing two dotted rectangles from $[x^L,x^U]$. The grid, drawn in dashed lines, induces a grid cover of the shaded region shown as cells separated by thick red lines.
\begin{figure}[ht]
\centering
\includegraphics{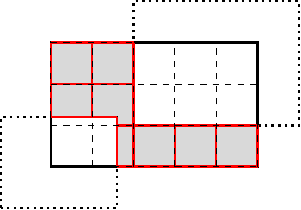}
\caption{Removing rectangles from a grid}
\label{fig:removeRectangle}
\end{figure}
Observe that the projection of $\point$ on $x$-space is a grid cover of $\AxisAligned$. We prove next that $\conv(\point)$ yields $\Rlx$. The proof exploits that any set $S$ and its cover $S_1, \ldots, S_m$ such that $S_i\subseteq S$ satisfy $\conv(S) = \conv \bigl(\cup_{j =1}^m \conv(S_j)\bigr)$. This is used to show that any point in the set $\bigl\{(x,t_1t_2) \bigm| x \in \AxisAligned \cap H,\ u_i(x_i) \leq t_i \leq o_i(x_i)\, \for \, i = 1,2 \bigr\}$ belongs to convex hull of the  union defining $\point$. 
\begin{theorem}\label{thm:functionGraphConvexHull}
Let $u_i(\cdot)$ and $o_i(\cdot)$ denote piecewise linear under- and over-estimators of $f_i(\cdot)$ on $[x^L,x^U]$, respectively, and let $\AxisAligned$ be an axis-aligned outer approximation of $D$. Then, $\Rlx$ constitutes a convex relaxation of $\G$ and coincides with the convex hull of $\point$. 
\end{theorem}
\begin{proof}
Since each point in $\G$ is contained in $\Rlx$ and $\Rlx$ is convex, it follows that $\conv(\G) \subseteq \Rlx$. We now show that $\Rlx = \conv(\point)$.
Since $\{H \cap \AxisAligned\}_{H \in \D}$ forms a grid cover of $\AxisAligned$ induced by $\D$, it follows that $\Rlx$ can be expressed as the convex hull of $\cup_{H \in \D} \Rlx_H$, where
\[
\Rlx_H:= \conv \bigl\{(x,t_1t_2) \bigm| x \in \AxisAligned \cap H,\ u_i(x_i) \leq t_i \leq o_i(x_i)\, \for \, i = 1,2 \bigr\}. 
\]
The proof is complete if we show that for every $H \in \D$, $\Rlx_H =  \conv(\point_H)$, where 
\[
\point_H := \Bigl\{(x,t_1t_2) \Bigm| x \in \corner(\AxisAligned \cap  H), t_i \in \bigl\{u_i(x_i), o_i(x_i)\bigr\}  \, \for \, i = 1,2 \Bigr\}. 
\]
Clearly, $\conv(\point_H) \subseteq \Rlx_H$ since $\point_H \subseteq \Rlx_H$ and $\Rlx_H$ is convex. To establish the reverse inclusion, we consider a point $(\bar{x},\bar{\mu}) \in \Rlx_H$ with $\bar{\mu} = \bar{t}_1\bar{t}_2$, where, for $i\in \{1,2\}$, $\bar{t}_i \in [u_i(\bar{x}_i),o_i(\bar{x}_i)]$. For $i=1,2$, let $\lambda_i$ be such that $\bar{t}_i = \lambda_i o_i(\bar{x}_i) + (1-\lambda_i) u_i(\bar{x}_i)$. Define $\ell(x) = \prod_{i=1,2} \lambda_i o_i(x_i) + (1-\lambda_i) u_i(x_i)$, and observe that $\ell(x)$ is a bilinear expression in $x$ and $\ell(\bar{x}) = \bar{\mu}$. Since $\AxisAligned\cap H$ is axis-aligned and $\ell(x)$ is bilinear over $H$, it follows by Lemma~\ref{lemma:bilinear-extension} that there exists a subset $V$ of $\corner(\AxisAligned\cap H)$ such that $(\bar{x},\bar{\mu}) = \sum_{v\in V} T^*(x,v) \cdot (v,\ell(v))$. Since we have expressed $(\bar{x},\bar{\mu})$ as a convex combination of points in $\point_H$, it follows that $\Rlx_H\subseteq \conv(\point_H)$. 
\end{proof}

As the under- and over-estimators $u_i$ and $o_i$ converge to $f_i$ over $[x_i^L, x_i^U]$ for $i = 1, 2$, and the outer approximation $\AxisAligned$ converges to the domain $D$, $\conv(\Q_H)$ converges to $\conv(\G)$. We will prove this result in Theorem~\ref{thm:convergence}.

\section{Multilinear compositions with Linking Constraints}\label{sec:nD}
In this section, we generalize our constructions to multilinear compositions $\Mul \mcirc f$ over a domain  $D$ defined by linking constraints. Specifically, we consider a multilinear outer-function  $\Mul: (t_1, \ldots, t_k) \in  \R^{m_1}  \times \cdots \times \R^{m_k} \mapsto \R$, where $k \geq 2$ and  $m_k \geq 1$ for each $i \in [k]$. The inner-function is defined as $f(x) = \bigl(f_1(x_1), \ldots, f_k(x_k) \bigr)$, where for $i \in [k]$, $f_i: x_i \in  X_i \subseteq \R^{n_i} \mapsto \R^{m_i} $ is a nonlinear function. In addition, we assume that the domain is defined by general linking constraints, that is, 
\[
D:=\bigl\{x \bigm| x_i \in X_i \for i \in [k],\ g(x) \leq 0\bigr\},
\]
where $g(\cdot)$ is a vector of nonlinear functions. We are interested in constructing polyhedral relaxations for the graph of \textit{multilinear composition} $(\Mul \mcirc f)= \Mul\bigl(f_1(x_1) ,\ldots, f_k(x_k) \bigr)$, that is, 
\[
\graph(\Mul \mcirc f):= \bigl\{(x, \mu) \bigm| \mu = (\Mul \mcirc f)(x),\ x\in D  \bigr\}. 
\]
In Section~\ref{sec:axis-alignedMultilinear}, we extend the notion of axis-aligned region from simple planar sets to the more general domain $D$ defined above. Given an axis-aligned outer-approximation $\AxisAligned$ of $D$, we obtain a polyhedral relaxation of $\graph(\Mul \mcirc f)$ by taking the convex hull of  a finite set of points generated from the corner points of $\AxisAligned$. To achieve this, the primary technical challenge is to generalize the Markov chain based decomposition algorithm proposed in Lemma~\ref{lemma:bilinear-extension}. Then, we further generalize our results to a vector of multilinear compositions.  Finally, in Section~\ref{sec:convergence}, we present two convergence results for two special cases of this framework.

\subsection{Axis-aligned approximations and multilinear compositions}\label{sec:axis-alignedMultilinear}

As in the previous section, we first identify a region $P_i$ that serves as an outer-approximation of the graph of each inner function $f_i(x_i)$. Subsequently, we replace the functions with auxiliary variables $t_i$ and relax the multilinear composition of a multivariate function defined over the Cartesian product of these outer approximations:
\begin{align*}
    &\Mul(t_1,t_2,\cdots,t_k)\\
    &(x_i,t_i)\in P_i,\quad i\in [k]
\end{align*}
Given that computing an outer approximation of the graph of a general multivariate function is computationally challenging, our implementation focuses on functions that are products of univariate components although our constructions apply more generally.

To address the linking constraints, we introduce the concept of polytopal tessellation, which provides a discretization of the feasible region.
\begin{definition}[Polytopal Tessellation]
    A \emph{polytopal tessellation} of a region is a cover of that region with a finite number of polytopes that do not overlap in the interior. The constituent polytopes are referred to as the cells of the tessellation.
\end{definition}

To illustrate how tessellation facilitates the handling of linking constraints, consider the following example:
\begin{align*}
    &\Mul(2x_1+x_2,y) := (2x_1+x_2)y\\
    \text{s.t.}\quad &x_1+x_2\leq 1, y\leq 1\\
    &x_1,x_2,y\geq 0
\end{align*}
The function $\Mul(2x_1+x_2,y)$ is a multilinear composition of linear functions, the domain of $(x_1,x_2)$ is a simplex $S$ and that of $y \in [0,1]$. Let $S$ admit a tessellation $S = S_1\cup S_2\cup S_3$, where
\begin{align*}
    &S_1 = \{(x_1,x_2)\mid x_1+x_2\leq 1,x_1\geq 0,x_2\geq 0.5\}\\
    &S_2 = \{(x_1,x_2)\mid x_2\geq x_1,x_1\geq 0,x_2\leq 0.5\}\\
    &S_3 = \{(x_1,x_2)\mid x_1+x_2\leq 1, x_1\ge x_2, x_2\geq 0\}
\end{align*}
Similarly,  $[0,1]$ admits a tessellation $[0,1] = [0,\frac{1}{3}]\cup [\frac{1}{3},\frac{2}{3}]\cup [\frac{2}{3},1]$. The product of the tessellations is depicted in Figure~\ref{Product of tesselations}. Now, suppose that the problem contains the following constraint linking $x_2$ and $y$:
\[2x_2-3y\leq 0\]
Observe that the cell $S_1\times [0,\frac{1}{3}]$ of the product tessellation is entirely excluded by this constraint. Consequently, $S\times[0,1]\setminus S_1\times[0,\frac{1}{3}]$ constitutes a relaxation of the feasible domain (see Figure~\ref{Product of tesselations}). 
Specifically, $S\times[0,1]\setminus S_1\times\left[0,\frac{1}{3}\right]$ can be viewed as a set that generalizes axis-aligned regions defined in Section~\ref{sec:bilinear-over-nonbox}. By applying a convex extension argument similar to the one used in Section~\ref{sec:bilinear-over-nonbox}, it suffices to consider only a finite set of identifiable points to convexify $\Mul$ over this relaxation. To formally show this result, we expand the definition of axis-aligned regions and their corner points to include more general structures.

\begin{definition}
    A subset $\AxisAligned$ in $\R^{n_1} \times \cdots \times \R^{n_k}$ is \emph{axis-aligned} if it can be expressed as a finite union of products of polytopes as follows:
    \[
    \bigcup_{i=1}^n\prod_{j=1}^k P_{ij},
    \]
    where for each $j\in [k]$, the sets $P_{1j},\dots,P_{nj}$ belong to the same Euclidean space $\R^{n_j}$.
\end{definition}

We define a corner point of such an axis-aligned region as follows:
\begin{definition}\label{defn:generalCorners}
    Let $\AxisAligned$ be an axis-aligned region in $\R^{n_1} \times \cdots \times \R^{n_k}$.
    A point $(v_1,v_2,\dots,v_k)\in\R^{n_1}\times\R^{n_2}\times\cdots\times\R^{n_k}$ is a \emph{corner point} of $\AxisAligned$ if, for each $j'\in [k]$, the vector $v_{j'}$ is an extreme point of the slice \[
    \conv\bigl\{(x_1,\dots,x_k)\in\AxisAligned \bigm| x_j = v_j \text{ for all } j\in [k]\setminus\{j'\}\bigr\}.
    \]
    The set of corner points of $\AxisAligned$ will be denoted as $\corner(\AxisAligned)$.
\end{definition}
We will refer to $\bigcup_{i=1}^n \prod_{j=1}^k \vertex(P_{ij})$ as $\disc(\AxisAligned)$. Once again, $\corner(\AxisAligned) \subseteq \disc(\AxisAligned)$ because if a point $v\not\in \AxisAligned\backslash\disc(\AxisAligned)$, then there exists an $i$ such that $v\in \prod_{j=1}^k P_{ij}$ and a $j$ such that $v_j\not\in \vertex(P_{ij})$. Since the slice of $\AxisAligned$, with $x_{j'} = v_{j'}$ for all $j'\in [k]\backslash \{j\}$ must include $P_{ij}$, $v$ cannot be an extreme point of the convex hull of this slice.

\begin{figure}[htbp]
    \centering
\tdplotsetmaincoords{110}{120} %
\includegraphics{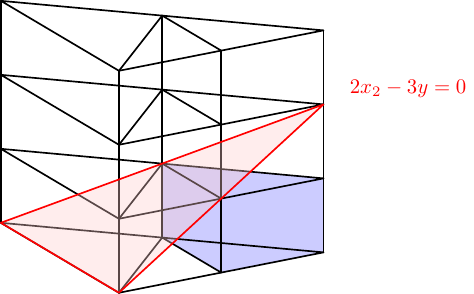}
    \caption{Product of tesselations}
    \label{Product of tesselations}
\end{figure}

We now demonstrate that, to convexify a multilinear function over an axis-aligned region, it suffices to evaluate the function solely at its corner points. We first mention that, to convexify a multilinear function over a product of polytopes $\prod_{i=1}^k P_i$, it suffices to consider only the corner points. This result follows from Theorem~1.2 in \cite{rikun1997convex}. 
\begin{lemma}[\citet{rikun1997convex}]\label{multilinear function}
    Let $\Mul$ be a multilinear function over a product of polytopes $\prod_{i=1}^kP_i$, where $P_i \subset \R^{n_i}$. Then the convex hull of the graph of $\Mul$ is the convex hull of the set
    \[ \left\{(v,\mu)\;\middle|\; \mu = \Mul(v),\ v\in\prod_{i=1}^k\vertex(P_i)\right\}\]
    where $\vertex(P_i)$ denotes the set of vertices of $P_i$.
\end{lemma}

Due to Lemma~\ref{multilinear function}, evaluating a multilinear function over product of polytopes requires solely the vertices. We now demonstrate that, by employing convex extension arguments, this set can be further reduced to the corners of the axis-aligned region itself. To this end, Algorithm~\ref{alg:general markov} constructs a general convex extension Markov chain for $\AxisAligned = \bigcup_{i=1}^{n}\prod_{j=1}^kP_{ij}$.

\begin{algorithm}[ht]
\caption{A Markov chain representation for coordinate-wise convex decompositions over general axis-aligned regions}
\label{alg:general markov}
\begin{algorithmic}[1]
\Require An axis-aligned region $\AxisAligned = \bigcup_{i=1}^{n}\prod_{j=1}^kP_{ij}$
\Ensure  A Markov chain with state space $\state$ and transition matrix $\transition$
        \State\label{algstep:NinitializeState} Initialize state space $\state \gets \disc(\AxisAligned)$ and $\transition(u,v) = 0$ for all $u,v\in \state$
        \State Initialize unprocessed stack $\mathcal{U}\gets \state$
        \While{$\mathcal{U}$ is not empty}
        \State\label{algstep:generalPop} $v \gets \text{pop}(\mathcal{U})$
        \State For each $j\in [k]$, define the slice $S_j = \{x_j \mid x\in\AxisAligned,\, x_{j'} = v_{j'}\text{ for }j'\in [k] \setminus \{j\}\}$
        \If {$v \in \vertex\bigl(\conv(S_j)\bigr)$ for all $j$}
           \State Add the self-loop $v \to v$ with transition probability $\transition(v,v) = 1$
        \Else
           \State Let  $j^*$ be the smallest index in $\bigl\{j \bigm| v\not\in \vertex(\conv(S_j))\bigr\}$
           \State\label{algstep:NfindVerts} Express $v$ as a convex combination: $v = \sum_{u\in \vertex(S_{j^*})} \lambda_{u}u$, where $\lambda_{u} \ge 0$ and $\sum_{u}\lambda_{u} = 1$
           \For{each $u$ in $\vertex(S_{j^*})$ with $\lambda_{u} > 0$}
                \If{$u \not\in\state \cup \mathcal{U}$}\label{algstep:NuNotState}
                    \State Push $u$ to $\state$ 
                    \State Push $u$ to $\mathcal{U}$
                \EndIf
                \State\label{algstep:NaddU}  $\transition(v,u) = \lambda_{u}$
            \EndFor
        \EndIf
        \EndWhile
        \State Return $(\state,\transition)$
\end{algorithmic}
\end{algorithm}

The finiteness of the resulting Markov chain is shown in Theorem~\ref{thm:multilinear-extension} using the finiteness of $\card(\vertex(P_{ij}))$ for all $(i,j)$.

This construction generalizes the convex extension Markov chain introduced in Section~\ref{sec:bilinear-over-nonbox}.
We now demonstrate that, for each vertex of the product-of-polytopes structure defining the axis-aligned region $\AxisAligned$, the chain's limiting distribution expresses the vertex as a convex combination of the corners of $\AxisAligned$. The proof of Theorem~\ref{thm:multilinear-extension} closely parallels the results in Section~\ref{sec:bilinear-over-nonbox}, although it addresses several additional technicalities that arise due to the abstraction of function approximation and the higher-dimensional setting.

\begin{theorem}\label{thm:multilinear-extension}
  Given an axis-aligned region $\AxisAligned$, Algorithm~\ref{alg:general markov} terminates in finitely many iterations, yielding an absorbing Markov chain $(\state,\transition)$ for which the corner points $\corner(\AxisAligned)$ serve as the absorbing states. Moreover, for any multilinear function $\Mul: \AxisAligned \to \R$, we have that for every $x \in \state$:
    \begin{equation}\label{eq:multConvCombination}
    \bigl(x,\Mul(x)\bigr) = \sum_{v\in\corner(\AxisAligned)}\transition^\limiting(x,v) \cdot \bigl(v, \Mul(v) \bigr).
    \end{equation}
    Consider $x\in P_i = \prod_{j=1}^k P_{ij} \subseteq \AxisAligned$ such that $x_j = \sum_{v'_j\in \vertex(P_{ij})}\lambda(v'_j)\cdot v'_j$ is a representation of $x_j$ as a convex combination of vertices in $P_j$. Then,
    \begin{equation}\label{eq:arbitraryXMult}
         \bigl(x,\Mul(x)\bigr) 
          = \sum_{v'\in \vertex(P_i)}\sum_{v\in \corner(\AxisAligned)}  \left(\prod_{j=1}^k \lambda(v'_j)\right)\transition^\limiting(v',v) \cdot \bigl(v,\Mul(v)\bigr).
    \end{equation}
\end{theorem}
\begin{proof}
Since $\AxisAligned$ is a finite union of products of polytopes, we construct vectors consisting of extreme points for each coordinate space $\R^{n_j}$ ($j\in [k]$). %
The $j^{\text{th}}$ vector is defined as the lexicographically sorted list of extreme points in $\bigcup_{i=1}^n \vertex(P_{i j})$, where $n$ is the number of products of polytopes in $\AxisAligned$. We will show that the state space $\state$ of the Markov chain is a subset of the grid formed by the Cartesian product of these vectors. Since the number of polytopes is finite, this grid--and thus $\state$--is finite. Consequently Algorithm~\ref{alg:general markov} terminates finitely, since Step~\ref{algstep:NuNotState} only pushes a point $v'$ to the stack if it is not already in $\state \cup \mathcal{U}$. 

It remains to show that every point generated at Step~\ref{algstep:NfindVerts} lies on this grid. We prove by induction that if the stack $\mathcal{U}$ contains only grid points, then any point added at Step~\ref{algstep:NaddU} is also a grid point. The base case holds because the initial vertices of the constituent polytopes are grid points by construction. Assume that the point $v$ popped from $\mathcal{U}$ at Step~\ref{algstep:generalPop} is a grid point. Then, the vectors $v^i$ generated at Step~\ref{algstep:NfindVerts} are extreme points of the following slice obtained by fixing all coordinates except those in $\R^{n_j}$:
\[S_j = \conv\left\{x_j\,\middle|\, x\in \bigcup_{i=1}^n \prod_{j=1}^k P_{ij},\ x_{j'} = v_{j'} \forall j\ne j'\right\}.\]
Observe that $S_j$ is the convex hull of a union of polytopes, where $P_{ij}$ contributes to this union if and only if $v_{j'}\in P_{ij'}$ for all $j'\ne j$. The extreme points of $S_j$ belong to $\vertex(P_{ij})$ for some contributing $P_{ij}$. This is because any extreme point of a union of sets must be an extreme point of one of the constituent sets; it cannot be expressed as a convex combination of other points in $S_j$ and, thus, also of other points in any constituent set. Since $\vertex(P_{ij})$ are included in the $j^{\text{th}}$ vector, the vertices generated at Step~\ref{algstep:NfindVerts} are included in the constructed grid vectors. Thus, all generated points remain within the finite grid, ensuring that Algorithm~\ref{alg:general markov} terminates with a finite discrete Markov chain.

We now show that $(\state,\transition)$ is an absorbing Markov chain such that the limiting distribution $T^*(s,v)$ equals zero for all $v\not\in \corner(\AxisAligned)$. Following the definition in~\cite{grinstead2012introduction}, a Markov chain is absorbing if it contains at least one absorbing state and if every state can reach an absorbing state (not necessarily in a single step). For such chains, Theorem 11.3 in~\cite{grinstead2012introduction} implies that the limiting distribution $\transition^\limiting$ assigns zero probability to transient states, that is, $\transition^\limiting(s,v) = 0$ for every $s\in\state$ and $v$ is not a corner of $\AxisAligned$. Thus, it suffices to show that $(\state,\transition)$ is absorbing, with corner points as absorbing states and all other points as transient states. To this end, we prove that from every non-corner point $s\in\state$, there exists a finite path to some corner point.

We use lexicographic order. Suppose $s$ is not a corner point of $\AxisAligned$. By construction, $s = \sum_r \lambda_r v_r^*$ for some $v_r^* \in \state$ with $\lambda_r > 0$ and $\sum_r \lambda_r = 1$. Recall that this representation is along a slice of the axis-aligned region with a suitably chosen $j$ such that all $x_{j'}$, $j'\ne j$ are fixed. The representation over this slice is such that, for each $r$, $s$ and $v_r^*$ differ in at least one component of $x_j$; let $j_r$ be the smallest such index. Without loss of generality, let $r=1$ minimize $\{j_r\}_{r}$.

If $v^*_{1j_1} > s_{j_1}$, then $v^*_1$ strictly majorizes $s$ lexicographically. If $v^*_{1j_1} < s_{j_1}$, then since $s_{j_1} = \sum_{i} \lambda_r v^*_{rj_1}$, there must exist some $k$ such that $v^*_{kj_1} > s_{j_1}$. Otherwise, \
\[0 = \sum_{r} \lambda_r (v^*_{rj_1} - s_{j_1}) \le \lambda_1 (v^*_{1j_1} - s_{j_1}),\]
where the first equality represents $s$ as a convex combination of $v^*_r$, the inequality follows because $\lambda_r > 0$ and we cannot locate $k$ only if $v^*_{rj_1} \le s_{j_1}$ for all $i > 1$. However, this contradicts our assumption that $\lambda_1 > 0$ and $v^*_{1j_1} < s_{j_1}$. Thus, $v^*_k$ strictly majorizes $s$ lexicographically. If $v^*_{1j_1} = s_{j_1}$, the minimality of $j_1$ as the first index where $v^*_1$ differs from $s$ is contradicted. Hence, $v^*_{1j_1} \neq v'_{j_1}$, and a strictly lexicographically larger neighbor $v^*$ always exists.

We extend the path by appending $v^*$ that majorizes $v'$ lexicographically and repeat. Since lexicographic order is a strict partial order and each step strictly increases, no cycles occur. The process terminates in finitely many steps at a corner point of $\AxisAligned$, as $\state$ is finite in size. Thus, $(\state,\transition)$ is absorbing, with corner points as absorbing states and all other states transient, thereby showing that 
\begin{equation}\label{eq:transientProbabilityZero}
\transition^\limiting(s,v) = 0, \text{ whenever } v\in \state\backslash \corner(\AxisAligned).
\end{equation}

Now, we show \eqref{eq:multConvCombination}. Consider $v'\in \state$. We prove by induction that for every positive integer $i$,
    \begin{equation}\label{eq:multransition}
    \bigl(v',\Mul(v')\bigr) =\sum_{v\in\state}\transition^i(v',v)\cdot \bigl(v,\Mul(v)\bigr).
    \end{equation}
    For $i = 1$, the statement holds by construction of the Markov chain. Let $V = \{v\mid \transition(v,v') > 0\}$ be the set of vertices with non-zero transition probability from $v'$. Then there exists a $j$ such that for all $v \in V$, $v'_{j'} = v_{j'}$ for $j' \in [k]\backslash\{j\}$. Thus, on the affine hull of $V$, $\Mul(\cdot)$ is affine and \eqref{eq:multransition} follows since $v' = \sum_{v\in V} \transition(v', v) v$. Suppose that the statement holds for $i = k-1$. Then
    \begin{align*}
     (v',\Mul(v')) &= \sum_{v\in\state}\transition^{k-1}(v',v)\cdot \bigl(v,\Mul(v)\bigr) \\
     &= \sum_{v\in\state} \transition^{k-1}(v',v)\sum_{v^*\in\state}\transition(v,v^*)\cdot \bigl(v^*,\Mul(v^*)\bigr)\\
    &=\sum_{v^*\in\state}\left(\sum_{v\in\state}\transition^{k-1}(v',v)\transition(v,v^*)\cdot \bigl(v^*,\Mul(v^*)\bigr)\right)\\
    &=\sum_{v^*\in\state}\transition^k(v',v^*)(v^*,\Mul(v^*)),
    \end{align*}
    where the first equality is by the induction hypothesis, second equality is by \eqref{eq:multransition} with $i=1$, and the third equality is by interchanging the summations. 
    By induction, the identity \eqref{eq:multransition} holds for all positive integers $i$. Taking the limit as $i\to\infty$, we obtain
     \[\bigl(v',\transition(v')\bigr) =\sum_{v\in\state}\lim_{i\to +\infty}\transition^i(v',v)\cdot \bigl(v, \Mul(v)\bigr) = \sum_{v\in\corner(\AxisAligned)}\transition^\limiting(v',v)(v,\Mul(v)),\]
    where the last equality follows from the definition of $\transition^\limiting(v',v)$ and \eqref{eq:transientProbabilityZero}. Thus, $\bigl(v',\Mul(v')\bigr) =\sum_{v\in\corner(\AxisAligned)}\transition^\limiting(v',v)\cdot \bigl(v,\Mul(v)\bigr)$ as required.

    Finally, we show \eqref{eq:arbitraryXMult}. Lemma~\ref{multilinear function} shows that $\bigl(x,\Mul(x)\bigr)$ can be written as a convex combination of points $\bigl(v',\Mul(v')\bigr)$ for $v'\in \vertex(P_i)$. We first provide this representation explicitly (see \eqref{eq:replaceAllCoord} below). For any $\bar{x}\in P_i$ and $v'_j\in \vertex(P_{ij})$, define $R_{j,v'_j}(\bar{x})\in \R^{n_1+\cdots+n_k}$ as the vector obtained by replacing $\bar{x}_j$ with $v'_j$. Formally,
    \begin{equation*}
        R_{j,v'_j}(\bar{x})_{j'} = \left\{\begin{aligned}
            &\bar{x}_{j'}  && j'\ne j\\
            &v'_j && \text{otherwise.}
        \end{aligned}\right.
    \end{equation*}
    Since $\Mul$ is affine over $P_{ij}$ when $x_{ij'}$ for $j'\ne j$ are fixed, it follows that:
    \begin{equation}\label{eq:replaceOneCoord}
        \bigl(\bar{x},\Mul(\bar{x})\bigr) = \sum_{v'_j\in \vertex(P_{ij})} \lambda(v'_j)\cdot \bigl(R_{j,v'_j}(\bar{x}), \Mul(R_{j,v'_j}(\bar{x})\bigr).
    \end{equation}
    Then, using \eqref{eq:replaceOneCoord} iteratively with $j$ from $k$ through $1$ and observing that $v' = R_{1,v'_1} \circ \cdots \circ R_{k,v'_k} (x)$, we have:
    \begin{equation}\label{eq:replaceAllCoord}
    \begin{split}
        \bigl(x,\Mul(x)\bigr) = \sum_{v'\in \vertex(P_{i})}\left(\prod_{j=1}^k\lambda(v'_j)\right)\cdot\bigl(v',\Mul(v')\bigr)
    \end{split}
    \end{equation}
    Using this representation of $\bigl(x,M(x)\bigr)$, we now derive \eqref{eq:arbitraryXMult} as follows:
    \begin{equation}
    \begin{split}
         \bigl(x,\Mul(x)\bigr) &= \sum_{v' \in \vertex(P_i)} \left(\prod_{j=1}^k \lambda(v'_j)\right)\cdot \bigl(v', \Mul(v')\bigr)\\
          &= \sum_{v' \in \vertex(P_i)} \left(\prod_{j=1}^k \lambda(v'_j)\right)\cdot \sum_{v\in\corner(\AxisAligned)}\transition^\limiting(v',v) \cdot \bigl(v, \Mul(v) \bigr)\\
          &= \sum_{v'\in \vertex(P_i)}\sum_{v\in \corner(\AxisAligned)}  \left(\prod_{j=1}^k \lambda(v'_j)\right)\transition^\limiting(v',v) \cdot \bigl(v,\Mul(v)\bigr),
    \end{split}
    \end{equation}
    where the first equality is because of \eqref{eq:replaceAllCoord}, the second equality is because of \eqref{eq:multConvCombination}, and the last moves $\prod_{j=1}^k \lambda(v'_j)$ inside the summation.
\end{proof}

\begin{corollary}\label{Hull}
    Let $\AxisAligned$ be a bounded axis-aligned region defined as in Definition~\ref{defn:generalCorners},  and let $\Mul: \AxisAligned \to \R$ be a multilinear function. The convex hull of $\graph(\Mul)$ is the convex hull of the following finite number of points:
    \[
    \bigl\{\bigl(v,\Mul(v)\bigr) \bigm| v\in\corner(\AxisAligned)\bigr\}. 
    \]
\end{corollary}

Consider multiple multilinear compositions sharing the same variables and domain. We are interested in their \emph{simultaneous hull}.
\begin{definition}
    Let $M_i:D\rightarrow\R$ for $i=1,2,\cdots,m$. The simultaneous hull of the set $\{M_i\}_{i\in [m]}$ is defined as:
    \begin{equation*}
    \conv(\{(x,M_1(x),M_2(x),\cdots,M_m(x))\mid x\in D\}).
    \end{equation*}
\end{definition}

Extending our analysis from single multilinear compositions, we relax the simultaneous hull to a collection of multilinear functions defined over an axis-aligned region $\AxisAligned$.  We formalize this observation in the following theorem.
\begin{theorem}\label{thm:simhull}
Let $\AxisAligned$ be a bounded axis-aligned region, and for each $i \in [m]$, let $M_i: \AxisAligned \to \R$ be a multilinear function. The simultaneous hull of $\{M_i\}_{i\in [m]}$ is the convex hull of the set
\begin{equation}\label{eq:cornerpointsSimultaneous}
    \bigl\{\bigl(v,M_1(v),\ldots,M_m(v)\bigr)\bigm| v\in \corner(\AxisAligned)\bigr\}.
\end{equation}
\end{theorem}
\begin{proof}
    The functions $\bigl(M_1(x),\ldots, M_m(x)\bigr)$ share the same convex extension Markov chain $(\state, \transition)$ and, therefore, also share $\transition^\limiting$. Moreover, from \eqref{eq:arbitraryXMult}, it follows that:
\begin{equation*}
         \bigl(x,\Mul_i(x)\bigr) 
          = \sum_{v'\in \vertex(P_i)}\sum_{v\in \corner(\AxisAligned)}  \left(\prod_{j=1}^k \lambda(v'_j)\right)\transition^\limiting(v',v) \cdot \bigl(v,\Mul_i(v)\bigr).
\end{equation*}
Since $\left(\prod_{j=1}^k \lambda(v'_j)\right)\transition^\limiting(v',v)$ is independent of $\Mul_i$, it follows that:
\begin{equation}\label{eq:arbitraryXMMult}
\begin{aligned}
         &\bigl(x,\Mul_1(x),\ldots,\Mul_m(x)\bigr) \\
         &\quad = \sum_{v'\in \vertex(P_i)}\sum_{v\in \corner(\AxisAligned)}  \left(\prod_{j=1}^k \lambda(v'_j)\right)\transition^\limiting(v',v) \cdot \bigl(v,\Mul_1(v),\ldots,\Mul_m(v)\bigr).
\end{aligned}
\end{equation}
The identity implies that any point in the graph of the functions over $\AxisAligned$ lies within the convex hull of their evaluations at the corners of $\AxisAligned$. Consequently, it suffices to consider the points given in \eqref{eq:cornerpointsSimultaneous} to characterize the simultaneous hull.
\end{proof}

A special case arises when each $x_i$ is a scalar, \textit{i.e.}, $n_j=1$ for all $j$:
\begin{align*}
    &\Mul(x_1,x_2,\cdots,x_k)\\
    &x_i\in \AxisAligned\quad i\in [k]
\end{align*}
Here, each $P_{ij}$ is an interval, making $\AxisAligned$ a union of axis-aligned hypercubes. This scenario generalizes the treatment in Section~\ref{sec:bilinear-over-nonbox}. If linking constraints exist, cells exterior to the region are removed as described previously. We specialize Theorem~\ref{thm:simhull} to this setting as follows.

\begin{corollary}\label{simplest convex extension}
  Let $\AxisAligned$ be an axis-aligned region in $\R^n$, that is, $\AxisAligned$ is a union of axis-aligned hypercubes, and for $i\in[m]$, let $M_i:\AxisAligned  \to \R$  be a multilinear function. Then, the simultaneous hull of $\{M_i\}_{i\in[m]}$ over $\AxisAligned$ is the convex hull of:
     \begin{equation}\label{eq:simplestn-d-points}
         \bigl\{(v,M_1(v),M_2(v),\cdots,M_m(v)) \bigm| v\in\corner(\AxisAligned)\bigr\}
     \end{equation}
\end{corollary}

Theorem~\ref{thm:simhull} unifies multilinear compositions with products of more general functions, extending beyond the specialized treatment for binary products in Section~\ref{sec:prod-over-nonbox}. This is achieved by approximating each function graph as a union of polytopes using an auxiliary variable to represent the function. This approach permits more flexible finite point selections for creating convex hulls. For instance, while Figure~\ref{fig:tessByVox} depicts the tessellation yielding the point set in Theorem~\ref{thm:functionGraphConvexHull}, based on the pentagonal outer-approximation of $x_i^2$ from Example~\ref{ex:comparison-oneestimator}, Theorem~\ref{thm:simhull} admits more general tessellations such as the one depicted in Figure~\ref{fig:tessGeneral}.

\begin{figure}[ht]
\centering
\begin{subfigure}[t]{0.45\linewidth}
\centering
\includegraphics{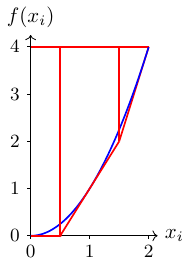}
\caption{Tessellation used by voxelization}\label{fig:tessByVox}
\end{subfigure}
\hfill
\begin{subfigure}[t]{0.45\linewidth}
\centering
\includegraphics{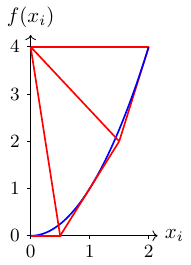}
\caption{An arbitrary tessellation}\label{fig:tessGeneral}
\end{subfigure}
\caption{Voxelization and Tessellation}
\label{fig:voxelization and tessellation}
\end{figure}

\subsection{Convergence Results}\label{sec:convergence}
In this subsection, we establish the convergence of our polyhedral relaxations. Specifically, we focus on the graph of a special class of multilinear compositions 
\begin{equation}\label{eq:multCompositionSet}
\bigl\{(x, \mu ) \bigm| x \in D ,\   \mu_i = M_i\bigl(f_1(x),\ldots, f_k(x)\bigr) \text{ for } i \in [m]\bigr\},
\end{equation}
where we assume that $D$ is a bounded domain, the outer-function $M_i: \R^k \to \R$ is a multilinear function, and each inner-function $f_j: \R^n \to \R$ is bounded. With a sequence of axis-aligned outer-approximations $\{\AxisAligned^\nu\}_{\nu \in N}$ of $\graph(f)$, we associate a sequence of sets $\{\point^\nu\}_{\nu \in N}$, each consisting of finitely many points defined as follows:
\begin{equation}\label{eq:sequenceQ}
\point^\nu = \bigl\{(x,\mu) \bigm| (x,t) \in \corner(\AxisAligned^\nu),\ \mu_i = M_i(t) \for i \in [m] \bigr\}.
\end{equation}
Our convergence analysis replies on the Hausdorff distance. Recall that for two subsets $C$ and $D$ of $\R^n$, the \textit{Hausdorff distance} between $C$ and $D$ is the quantity 
$d_\infty (C,D) := \inf \bigl\{\eta \geq 0 \bigm| C \subseteq D +\eta B_n,\ D\subseteq C + \eta B_n  \bigr\}$,
where $B_n$ is the unit Euclidean ball in $\R^n$. A sequence $C^{\nu}$ of sets is said to \textit{converge with respect to Hausdorff distance} to $C$, denoted as $C^\nu \to C$, if  $d_\infty(C^{\nu},C) \to 0$ as $\nu$ goes to infinity. 

We first  present a procedure for generating axis-aligned outer-approximations of a bounded set described by a system of nonlinear inequalities. 

\begin{algorithm}[ht]
\caption{An axis-aligned outer-approximation procedure}\label{alg:axis-aligned}
\begin{algorithmic}[1]
\Require A bounded region $S$  given by a system of inequalities $g_j(x) \leq 0$ for $j \in [m]$, finite bounds $x^L$ and $x^U$ on $x$, and a tolerance $\epsilon$ 
\State Initialize unprocessed stack $\mathcal{U} \gets [x^L, x^U] $
\While{$\mathcal{U} \neq \emptyset$}
  \State  $B \gets \text{pop}(\mathcal{U})$
  \State $[g^L,g^U] \gets \texttt{range}(g,B)$
  \If{$g^U_j \leq 0$ for all $j \in [m]$ }
  \State Push the box $B$ to $\rects$ 
  \State  \textbf{continue}
  \ElsIf{$\exists j \in [m]$ such that $g^L_j > 0 $}\label{algstep:outsideBox}
  \State \textbf{continue}
  \EndIf
  \If{ $\diam(B) \geq \epsilon$}
      \State Push a set of boxes $\texttt{split}(B)$ to $\mathcal{U}$
      \Else 
      \State Push the box $B$ to $\rects$.
   \EndIf
\EndWhile
\State \Return An axis-aligned region $\AxisAligned: = \cup_{B \subseteq \rects} B$.
\end{algorithmic}
\end{algorithm}
For a box $B$, the procedure $\texttt{range}(g,B)$ returns an $m$-dimensional box $[g^L,g^U]$, where $g^L, g^U \in \R^m$, such that the image $g(B)$ is contained within it. For a Lipschitz-continuous function $g$ with constant $K$ over $B$,  a valid enclosure is obtained by considering a point $c$ of $B$  and defining: 
\[
g^L_j:=g_j(c) -K \max_{x\in B} \| x - c\| \quad \text{ and} \quad
g^U_j:= g_j(c) +K \max_{x\in B} \| x - c\|.
\]
This range estimation is consistent, that is, as the diameter of $B$ approaches zero, the diameter of $[g^L,g^U]$ approaches zero as well. Here, the diameter of a set $C$ is defined as $\diam(C): = \sup_{x,y \in D} \|x-y \|$. Although this Lipschitz-based bound is consistent, it can be loose. To get a tighter estimation of the range, one can use interval arithmetic \citep[Section 6]{moore2009introduction}. Finally, for a given box $B$, the procedure $\texttt{split}$ uniformly partitions it into $2^n$ sub-boxes by bisecting each dimension. 
\begin{lemma}\label{lemma:axisalign}
Let $C:=\bigl\{x \bigm| x^L \leq  x \leq x^U,\ g(x) \leq 0\bigr\}$, where $x^L, x^U \in \R^n$, and $g: [x^L, x^U] \to \R^m$ is Lipschitz continuous. For $\nu \in \N$, let $\AxisAligned^\nu$ be the axis-aligned region generated by Algorithm~\ref{alg:axis-aligned} with tolerance $\epsilon := \frac{1}{\nu}$. Then,  $\AxisAligned^\nu$ outer-approximates the domain $C$, and $\AxisAligned^\nu \to C$. 
\end{lemma}
\begin{proof}
For any $\nu \in \N$,  Algorithm~\ref{alg:axis-aligned} terminates in a finite number of steps. The uniform \texttt{split} procedure partitions a box $B$ into $2^n$ sub-boxes by bisecting each dimension, which reduces the diameter of the resulting boxes by a factor of $2$ at each depth of the search tree. Specifically, at depth $k$, the diameter is $\diam(B^{(k)}) = 2^{-k} \diam([x^L, x^U])$. Since $\epsilon = \frac{1}{\nu}$, there exists a finite $k$ such that $\diam(B^{(k)}) < \epsilon$, at which point the algorithm either accepts or discards the box without further splitting. 

To show $C \subseteq \AxisAligned^\nu$, observe that a box $B$ is discarded if and  only if the condition at Line~\ref{algstep:outsideBox} is satisfied. For a discarded box $B$, there exists $j \in [m]$ such that $g^L_j>0$ and, therefore, for any $x \in B$, $g_j(x)\geq g_j^L>0$, implying $x \notin C$. Thus, the boxes discarded have empty intersection with $C$. Since all other boxes are collected in $\rects$, their union $\AxisAligned^\nu$  contains $C$.

We show that $d_\infty(\AxisAligned^\nu, C) \to 0$ as $\nu \to \infty$. Since $C \subseteq \AxisAligned^\nu$, the Hausdorff distance is given by $d_\infty(\mathcal{R}^\nu, C) = \sup_{x \in \AxisAligned^\nu} \text{dist}(x, C)$. Let $x \in \AxisAligned^\nu$. Then, $x$ belongs to some box $B \in \rects$. Let $[g^L,g^U]$ be $\texttt{range}(g,B)$. If $g^U_j\leq 0$ for $j \in [m]$ then $B \subseteq C$ and thus $\text{dist}(x, C) = 0$. Otherwise, there exists $k \in [m]$ such that $g_k^U>0$. Since such a box $B$ has been pushed to $\rects$, $\diam(B) \leq \epsilon := \frac{1}{\nu}$ and $g^L_j\leq 0$ for every $j \in [m]$. Consequently,  for each $ j \in [m]$, we have
\[
g_j(x) \leq g_j^U = (g_j^U - g_j^L) +g_j^L \leq g_j^U - g_j^L \leq \diam\bigl(\texttt{range}(g,B)\bigr). 
\]
As $\nu \to \infty$, the consistency of the \texttt{range} procedure ensures that $\diam\bigl(\texttt{range}(g,B)\bigr) \to 0$, which implies that for every $j \in [m]$, $g_j(x) \leq \delta(\nu) $ where $\delta(\nu) \to 0^+$. Since $g$ is continuous and the domain is compact, the sublevel sets $C(\delta) := \{x \in [x^L, x^U] \mid g(x) \leq \delta\}$ converge to $C$ in the Hausdorff sense as $\delta \to 0^+$. Thus, $\text{dist}(x, C) \to 0$ as $\nu \to \infty$, completing the proof. 
\end{proof}

Next, we present two technical results; Lemma~\ref{lemma:graphconvergence} follows from Theorems~4.10 and 4.26 in \cite{rockafellar1998variational}, and Lemma~\ref{lemma:convergence} follows from Theorem~4.10 and Proposition 4.30 in \cite{rockafellar1998variational}. However, we provide direct simple proofs needed for our setting.

\begin{lemma}\label{lemma:graphconvergence}
    Let $g: \R^n \to \R^m$ be a continuous function, and let $\{C^\nu\}_{\nu \in \N}$ be a sequence of bounded sets in $\R^n$ such that $C^\nu \to C$. Then, $g(C^\nu) \to g(C)$. 
\end{lemma}
\begin{proof}
We aim to show that $d_\infty(g(C^\nu),g(C))\rightarrow 0$ as $\nu\rightarrow \infty$. Since $C$ is a bounded subset of $\R^n$, its closure $\bar{C}$ is compact. The function $g$ is continuous on $\R^n$, and, therefore, by Heine-Cantor Theorem, uniformly continuous on the compact set $\bar{C}$. By the definition of uniform continuity, for every $\epsilon > 0$, there exists a $\delta > 0$ such that for all $x,y \in \bar{C}$:
\begin{equation*}
\|x-y\| < \delta \Longrightarrow \|g(x)-g(y)\| < \epsilon.
\end{equation*}
Since $C^\nu \rightarrow C$, there exists an integer $N$ such that for all $\nu > N$, $d_\infty(C^\nu,C) < \delta$. This implies that for every $x\in C^\nu$, there exists $y\in C$ such that $\|x-y\| < \delta$. By uniform continuity, $\|g(x) - g(y)\| < \epsilon$. Thus, $g(C^\nu) \subseteq g(C) + \epsilon B_m$, where $B_m$ is a Euclidean unit ball in $\R^m$. A symmetric argument shows that $g(C) \subseteq g(C^\nu) + \epsilon B_m$. Combining these inclusions, we have $d_\infty(g(C^\nu),g(C)) \le \epsilon$ for all $\nu > N$. Since $\epsilon$ was arbitrary, $d_\infty(g(C^\nu),g(C)) \to 0$, proving that $g(C^\nu) \to g(C)$.
\end{proof}
\begin{lemma}\label{lemma:convergence}
If a sequence  $\{C^\nu\}_{\nu \in \N}$ of bounded sets in $\R^n$ converges to $C$ then $\conv(C^\nu) \to \conv(C)$. In particular $d_\infty\bigl(\conv(C^\nu),\conv(C)\bigr) \le d_\infty(C^\nu,C)$.
\end{lemma}
\begin{proof}
    We show that $d_\infty\bigl(\conv(C^\nu),\conv(C)\bigr) \le d_\infty(C^\nu,C)$. Assume $d_\infty(C^\nu,C) = \epsilon$. Then,
    \begin{equation*}
    C^\nu \subseteq C + \epsilon B_n
    \Longrightarrow \conv(C^\nu) \subseteq \conv(C + \epsilon B_n) = \conv(C) + \epsilon B_n
    \end{equation*}
    where the implication follows by convexifying both sides and the equality because, for sets $X,Y$ $\conv(X+Y) = \conv(X) + \conv(Y)$ and $\epsilon B_n$ is convex. A symmetric argument shows that $C\subseteq C^\nu + \epsilon B_n$ implies $\conv(C)\subseteq \conv(C^\nu) + \epsilon B_n$. Therefore, $d_\infty\bigl(\conv(C^\nu),\conv(C)\bigr) \le \epsilon$, or  $d_\infty\bigl(\conv(C^\nu),\conv(C)\bigr) \le d_\infty(C^\nu,C)$. Then, $C^\nu\rightarrow C$ implies $\conv(C^\nu)\rightarrow \conv(C)$.
\end{proof}

\begin{theorem}\label{thm:convergence}
    Let  $\G$ be the graph defined as in \eqref{eq:multCompositionSet}. Assume that the domain $D$ is bounded and is described by a system of nonlinear inequalities $g(x) \leq 0$ and $x^L \leq x \leq x^U$. Moreover, assume that each $g(\cdot)$ and $f(\cdot)$ is Lipschitz continuous. For $\nu \in \N$, let  $\AxisAligned^\nu$ be the region returned by Algorithm~\ref{alg:axis-aligned} when it is invoked with $S$ as $\graph(f)$ and a tolerance $\epsilon$ of $\frac{1}{\nu}$. Then, for all $\nu$, $\AxisAligned^\nu$ is an axis-aligned outer-approximation of $\graph(f)$. For each element in the sequence $H^\nu$, construct the point set $\point^\nu$ as in Equation~\eqref{eq:sequenceQ}.  Then, for every $\nu \in \N$, $\conv(\point^v)$ is a polyhedral relaxation of $\G$. Moreover, $\conv(\point^\nu) \to \conv(\G)$. 
\end{theorem}

\begin{proof}
    By Lemma~\ref{lemma:axisalign}, it follows directly that $H^\nu$ is an axis-aligned outer-approximation of $\graph(f)$. This justifies constructing $\point^\nu$ using Equation~\eqref{eq:sequenceQ}, as described in the statement of the result. Now, we show that $\conv(\point^\nu)$ is a polyhedral relaxation of $\G$. By Theorem~\ref{thm:simhull}, $\conv(\point^\nu)$ is the convex hull of:
    \begin{equation*}
        \Rlx^\nu := \bigl\{(x, \mu) \bigm| \bigm (x,t) \in \AxisAligned^\nu,\ \mu_i = M_i(t) \for i \in [m]   \bigr\}.
    \end{equation*}
    Since $\AxisAligned^\nu$ outer-approximates $\graph(f)$, $\Rlx^\nu$ is a relaxation of $\G$ and, therefore, $\conv(\point^\nu) = \conv(\Rlx^\nu) \supseteq \G$. By Lemma~\ref{lemma:axisalign}, Algorithm~\ref{alg:axis-aligned} terminates finitely. Therefore, $\AxisAligned^\nu$ has finitely many boxes and $\corner(\AxisAligned^\nu)$ is finitely sized. This shows that $\conv(\point^\nu)$ is a polyhedral outer-approximation of $\G$.
    
     We now show that $\conv(\point^\nu)$ converges to $\conv(\G)$. By Lemma~\ref{lemma:axisalign}, we have $\AxisAligned^\nu \rightarrow \graph(f)$. Since a point in $\graph(f) = \{(x,t)\mid x\in D, t_i=f_i(x)\}$ is mapped to $\G$ by computing $\mu_i=\Mul_i(t)$ for $i\in [m]$ and projecting out $t$, it is a continuous map. The same map also transforms $\AxisAligned^\nu$ to $\Rlx^\nu$. Therefore, by Lemma~\ref{lemma:graphconvergence}, it follows that $\Rlx^\nu\rightarrow \G$. Then, by Lemma~\ref{lemma:convergence}, it follows that $\conv(\Rlx^\nu)\rightarrow \conv(\G)$. Since we have shown that $\conv(\point^\nu) = \conv(\Rlx^\nu)$, the result follows.
\end{proof}

Next, we consider a special case of the graph defined in~\eqref{eq:multCompositionSet} where the number of inner-functions is $n$ and for $i \in [n]$ $f_i(x) = x_i$. Given an axis-aligned approximation $\AxisAligned$ of the domain $D$, Corollary~\ref{simplest convex extension}  implies that  a polyhedral relaxation is given as the convex hull of 
\begin{equation}\label{eq:mul-rlx-points}
    \point(\AxisAligned) := \bigl\{(x, \mu) \bigm| x \in \corner( \AxisAligned), \mu_i = \Mul_i(x) \for i \in [m]  \bigr\}. 
\end{equation}
For this specific structure, we establish a refined convergence result that provides an upper bound on  the cardinality of $\point(\AxisAligned)$ required to achieve an $\epsilon$-approximation of the  graph. A consequence of this result is that the convergence rate is linear for bivariate functions. 
 For a set of finite number of points $S$, let $\card(S)$ denote its cardinality. 
\begin{proposition}\label{points counting}
Let $D$ be  a full dimensional bounded set in $\R^n$, and consider a vector of  functions $\Mul:D \to \R^m$ defined as $\Mul(x) = (\Mul_1(x),\Mul_2(x),\cdots, \Mul_m(x))$, where $\Mul_i(x)$ is a multilinear function and has a Lipschitz constant of $L$. Let $w$ be the diameter of $D$.  Then, for any given $\epsilon >0$, there exists an axis-aligned outer-approximation $\AxisAligned$ of $D$ such that 
$d_\infty\bigl(\conv(\point(\AxisAligned)) , \conv(\graph(\Mul)) \bigr) \le \epsilon$, where $\point(\AxisAligned)$ is the set defined in \eqref{eq:mul-rlx-points} with $\card(\point(\AxisAligned)) \le \bigl(\frac{w(Lm+1)\sqrt{n-1}}{\epsilon} + 1\bigr)^{n-1}$. 
\end{proposition}
\begin{proof}
 The proof constructs the desired axis-aligned approximation $\AxisAligned$. Towards this end, we define $G$ to be a regular grid over $\R^n$ with mesh size $d = \frac{\epsilon}{(Lm+1)\sqrt{n-1}}$. Let $D' = \proj_{(x_1,\ldots,x_{n-1})} (D)$ and $G' = \proj_{(x_1,\ldots,x_{n-1})} (G)$ be the projections on the first $n-1$ coordinates. Observe that $G'$ is also a regular grid in $\R^{n-1}$. Let $\AxisAligned'$ be a minimal axis-aligned approximation of $D'$ using the grid cells induced by $G'$, \textit{i.e.}, if any grid cell is removed from $\AxisAligned'$ then the resulting region does not cover $D'$. For each $H' \in \AxisAligned'$, we derive two endpoints for the $n^{\text{th}}$ coordinate:
 \[
 \ell^L(H'):=\min\{x_n \mid (x_1, \ldots, x_{n-1} ) \in H',\ x \in D \} \text{ and } \ell^U(H'):=\max\{x_n \mid (x_1, \ldots, x_{n-1} ) \in H',\ x \in D \},
 \]
 and lift $H'$ into $\R^n$ as follows $ L(H') := \bigl\{ x \bigm| (x_1, \ldots, x_{n-1} ) \in H', \ell^L(H') \leq x_n \leq \ell^U(H') \bigr\}$. Now, we obtain an axis-aligned region by taking a union of lifted hypercubes, that is, 
 \[
 \AxisAligned: = \bigcup_{H' \in \AxisAligned'} L(H'). 
 \]
To see that $\AxisAligned$ is an outer-approximation of $D$, we consider a point $x\in D$. Since $\AxisAligned'$ is an axis-aligned outer-approximation of $D'$, there exists $H' \in \AxisAligned'$ such that $(x_1, \ldots, x_{n-1}) \in H'$. Therefore, $\ell^L(H') \leq x_n \leq \ell^U(H')$. This shows that $x$ is contained in $L(H')$ and hence in $\AxisAligned$. In other words, $\AxisAligned$ is indeed an axis-aligned outer-approximation of $D$. 

Next, we construct a superset of $\corner(\AxisAligned)$.  Let $E':=\cup_{H' \in \AxisAligned'}\vertex(H')$. Then, $E'$ contains the projection of $\corner(\AxisAligned)$ onto the first $n-1$ coordinates since
\[
\proj_{(x_1, \ldots,x_{n-1})}\corner(\AxisAligned) \subseteq \proj_{(x_1, \ldots,x_{n-1})} \cup_{H' \in \AxisAligned'} \vertex(L(H'))  = \cup_{H'\in \AxisAligned'} \vertex(H') = E'. 
\]
Now, we lift $E'$ into $\R^n$ to obtain a superset $E$ of $\corner(\AxisAligned)$. For each point $v' \in E'$, we consider a set of candidate values for the $n^{\text{th}}$ coordinate: $C(v'):= \bigcup_{H' \in \AxisAligned \mid v' \in H'} \bigl\{\ell^L(H'), \ell^U(H')\bigr\}$.
Clearly, $\bigl\{(v',x_n) \bigm| v' \in E',\ x_n \in C(v') \bigr\} = \cup_{H \in \AxisAligned}\vertex(H)$, which contains corner points of $\AxisAligned$. Now, some of the non-extreme elements in $C(v')$ can be removed to obtain a refined  superset as follows: 
\[
E:=\Bigl\{(v',x_n) \Bigm| v' \in E',\ x_n \in \vertex\bigl(\conv(C(v'))\bigr)\Bigr\}.
\]
We will now show that $\card(E) \leq \bigl(\frac{w(Lm+1)\sqrt{n-1}}{\epsilon} + 1\bigr)^{n-1}$ which implies the same bound for $\card(\point(\AxisAligned))$ since $\corner(\AxisAligned) \subseteq E$. Since the diameter of $D$ is $w$, the diameter of $D'$ is upper bounded by $w$. We start the grid $G$ at the coordinate-wise minimum on all coordinates in $D$. Then,
\[
\card(\point(\AxisAligned)) = \card\bigl(\corner(\AxisAligned)\bigr) \leq \card(E) = 2\card(E') \le 2\Bigl(\frac{w}{d}+1\Bigr)^{n-1} = \Bigl(\frac{w(Lm+1)\sqrt{n-1}}{\epsilon} + 1\Bigr)^{n-1},
\]
 where the first inequality is because $E\supseteq \corner(\AxisAligned)$, the second inequality is because each point in $E'$ was lifted to two points in $E$, the first inequality is because $\lceil\frac{w}{d}\rceil$ suffices to cover the width $w$ at mesh-size $d$, and the last equality is by substituting $d$ based on our choice.

Finally, we show that the approximation error is bounded by $\epsilon$, that is, $d_\infty\bigl(\conv(\point(\AxisAligned)), \conv(\graph(\Mul)) \bigr) \le \epsilon$. Since $\AxisAligned \supseteq D$, Corollary~\ref{simplest convex extension} implies that  $\graph(\Mul) \subseteq \conv(\point(\AxisAligned))$, and thus $\conv(\graph(\Mul)) \subseteq  \conv(\point(\AxisAligned))$. Therefore, it suffices to show that $\conv(\point(\AxisAligned)) \subseteq \conv(\graph(\Mul)) + \epsilon B_{n+m}$, where $B_{n+m}$ is a unit ball in $\R^{n+m}$. Moreover, since $\conv(\graph(M)) + \epsilon B_{n+m}$ is convex, it only remains to show that $\point(\AxisAligned)\subseteq \conv(\graph(M)) + \epsilon B_{n+m}$. Pick an arbitrary point $(v',x_n,\mu)\in \point(\AxisAligned)$, which, by definition, is such that $v'\in E'$, $x_n\in \vertex(\conv(C(v')))$ and $\mu = M(v',x_n)$. Observe that, by the definition of $C(v')$, there exists a $H' \in \AxisAligned$ and $(x_1,\ldots,x_{n-1})\in H'$ such that $(x_1, \ldots,x_n) \in D$. Then, $(x,M(x))\in \graph(\Mul)$. The following shows that $(x,M(x))$ is within an $\epsilon$ distance of $(v',x_n,\mu)$:
   \begin{equation*}
       \begin{split}
       &\|(x,M(x)) - (v',x_n,M(v',x_n)\| \le \|x-(v',x_n)\| + \sum_{i=1}^m L\|x - (v',x_n)\| =\\
       &(Lm+1)\|x-(v',x_n)\|  \le (Lm+1) \sqrt{\max\nolimits_{i=1}^{n-1} (v'_i - x_i)^2} \le (Lm+1)d\sqrt{n-1} = \epsilon,
       \end{split}
   \end{equation*}
   where the first inequality is by triangle inequality and the Lipschitz constant of $M_i$ for $i\in [m]$, the second inequality simply overestimates each term, and the third inequality is because the mesh size of $\AxisAligned'$ is $d$ implying that $|v_i-x_i|\le d$, and the last equality is by substituting $d=\frac{\epsilon}{(Lm+1)\sqrt{n-1}}$.
\end{proof}

\section{Implementation}\label{sec:implementation}
This section details the implementation of our proposed relaxation schemes for constructing a polyhedral relaxation for mixed-integer nonlinear programs (MINLP) of the form:
\begin{equation}
\label{prob:general_optimization}
\begin{aligned}
    \min_{x \in \R^n} \quad &f(x)\\
    &g(x)\leq 0 \\
    & x \in \mathcal{X} \text{ and } x_i \in \mathbb{Z} \, \for  i\in I,
\end{aligned}
\end{equation}
where $f: \R^n \to \R$ and $g:\R^n \to \R^m$ are given functions, $\mathcal{X} \subseteq 
\R^n$ is a bounded polyhedral set, and $I \subseteq [n]$ denotes the index-set of the integer variables. We restrict our attention to  factorable functions $f$ and $g$, \textit{i.e.}, functions that can be recursively expressed as finite compositions of a list of binary operators $\{+, *, /\}$ and  univariate functions~\cite{mccormick1976computability}. 

\subsection{Base Relaxation}\label{BASE}
To benchmark our relaxation scheme against standard techniques, we first construct a baseline (Base) that is derived from factorable programming (FP), augmented with specific enhancements. In FP, each nonlinear function is represented as an expression tree. Then, an auxiliary variable is introduced for each internal node (representing an operator), and its domain is constrained by relaxing the operator over bounds derived independently for its operands. For more details, we refer to Algorithm~\ref{BASE} in Appendix~\ref{Factorable relaxation}  and the standard literature~\cite{mccormick1976computability,tawarmalani2004global}.

However, our baseline (Base) incorporates several enhancements to strengthen FP relaxations, aligning it with state-of-the-art techniques. Specifically, the implementation employs expression simplification, expansion, and matching. First, the constructor parses each expression by flattening nested additions/subtractions and multiplications/divisions, then recursively distributes multiplication over addition. Complex nonlinear expressions are decomposed into simpler bivariate forms via auxiliary variables. Second, redundant auxiliary variables for identical sub-expressions are eliminated in favor of a single representative variable. Third, repeated products (e.g., $x \times x \times x$) are consolidated into power terms ($x^3$), with auxiliary variables introduced for univariate functions and products. Finally, additional auxiliary variables simplify expressions into affine combinations of univariate functions, followed by a final round of expression matching.

The baseline also deduces tighter bounds on all variables in the transformed model via forward and reverse bound propagation. The process infers bounds for each variable using the current bounds of the others, employing interval arithmetic~\cite{moore2009introduction} as detailed in Appendix~\ref{variablebound}. We further apply range reduction using a local  solution (OBTT)~\cite{bestuzheva2023global} and, after solving, perform duality-based range reduction~\cite{tawarmalani2013convexification}. Finally, we construct convex hulls for the graphs of univariate functions (monomials, exponentials, and logarithms) over their interval bounds by exploiting their specific properties~\cite{tawarmalani2013convexification}.

To avoid excessive complexity in the baseline, we have chosen to omit several prevalent techniques. Specifically, our relaxations remain purely polyhedral excluding inequalities involving Lorentz cones~\cite{ben2021lectures}, positive-semidefinite matrices~\cite{lasserre2001explicit,ben2021lectures}, or completely positive matrices~\cite{burer2009copositive}. Integer variables are relaxed to continuous domains without incorporating integer-programming cuts~\cite{conforti2014integer}. Additionally, we do not implement convexity detection~\cite{bonami2008algorithmic,kronqvist2019review, coey2020outer}, perspective reformulations~\cite{gunluk2010perspective}, and symmetry reduction~\cite{hojny2024detecting}, or embed the relaxations within a branch-and-bound framework~\cite{land1960automatic,ryoo1996branch,tawarmalani2005polyhedral,belotti2009branching}. Consequently, while our baseline is robust for general factorable structures, it may be less competitive on problems where these specialized techniques are critical. This limitation of our baseline motivates our comparison against the commercial solver Gurobi, which integrates many of these advanced features.

Since our Voxelization Relaxation (VR) differs from the baseline primarily in its treatment of multiplication nodes, this section details VR's specific approach. We first describe how VR handles multiplication nodes within expression trees, and then present two distinct methods for generating voxelizations. To avoid excessive implementation complexity while allowing assessment of the potential benefits and promise of our proposed framework, we selectively implement a subset of the proposed features.

\subsection{Treatment of Multiplication Nodes}
A key distinction between VR and Base lies in their treatment of multiplication operators. While Base relaxes the multiplication operator by finding bounds for each of the operands, VR performs a deeper structural analysis, to identify if operands correspond to univariate functions, treating such children as functions rather than scalars. Figure~\ref{fig:expression different treatment} illustrates this difference. Base relaxes the product $t_0 = t_1\times t_2$ using bounds for $t_1$ and $t_2$, whereas VR recognizes the structure $t_0 = t_1 \times \exp(t_5)$ and relaxes it the graph in $(t_0,t_1,t_5)$ space exploiting the function structure of $\exp(t_5)$.

\begin{figure}
    \scriptsize
    \begin{center}
    \begin{forest}
      for tree={
        align=center, grow'=south,
        child anchor=north, parent anchor=south,
        l sep=6pt, s sep=6pt,
        edge={thick}, %
        draw, rounded corners, inner sep=2pt, outer sep=0pt
      }
      [%
        {$t_0:\times$}
          [%
            {$t_2:\exp$}
              [{$t_5$}]
          ]
          [%
            {$t_1:\times$}
              [{$t_3$}]
              [{$t_4$}]
          ]
      ]
    \end{forest}
    \end{center}
    \caption{Partial expression tree}
    \label{fig:expression different treatment}
\end{figure}
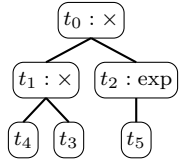

Furthermore, unlike Base, which ignores the linking constraints and relaxes products over rectangular domains, VR constructs an axis-aligned outer approximation of the feasible region. It then employs computational geometry tools (Section~\ref{sec:prod-over-nonbox}) to relax products over this domain. Although our methodology extends to high dimensions (see Section~\ref{sec:nD}), our implementation restricts itself to low-dimensional spaces, utilizing voxelization to construct the outer approximation~\cite{sramek2002alias} (see Figure~\ref{fig:banana_voxel_supervoxel}).

\begin{figure}[ht]
\centering

\begin{subfigure}[t]{0.3\linewidth}
\centering
\includegraphics{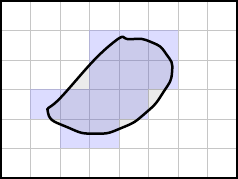}
\caption{Approximation via Voxelization}
\end{subfigure}
\hfill
\begin{subfigure}[t]{0.3\linewidth}
\centering
\includegraphics{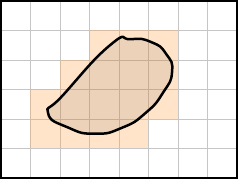}
\caption{Outerapproximation via Voxelization}
\end{subfigure}
\hfill
\begin{subfigure}[t]{0.3\linewidth}
\centering
\includegraphics{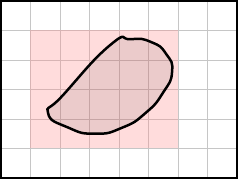}
\caption{Bounding Box}
\end{subfigure}
\caption{A comparison among approximation and outer approximation produced by voxelization and bounding box}
\label{fig:banana_voxel_supervoxel}
\end{figure}

This procedure is formally described in Algorithm~\ref{alg:multiplication node}:
\begin{algorithm}[ht]
\caption{Handling Multiplication Nodes}
\label{alg:multiplication node}
\begin{algorithmic}[1]
\Require Expression tree $T$, multiplication node $t_0 \in T$, inequality system $I$, number of breakpoints $N_b$
\For{each child node $t_i$ of $t_0$}
    \If{$t_i$ represents a nonlinear univariate function}
        \State Let $t_i^c$ be the input variable such that $t_i = f_i(t_i^c)$.
        \State Construct piecewise linear under- and over-estimators, $u_i$ and $w_i$, for $f_i$ on its domain using $N_b$ uniformly spaced breakpoints $X_i$.
        \State Set $v_i \gets t_i^c$.
    \Else
        \State Set $u_i(t_i) = w_i(t_i) = t_i$. Let $X_i$ be the set of bounds for $t_i$.
        \State Set $v_i \gets t_i$.
    \EndIf
\EndFor
\State Construct an axis-aligned outer approximation $A$ of the domain of $(v_1, v_2)$ using voxelization and constraints $I$.
\State Define a grid cover $P$ of $A$ induced by the Cartesian product $X_1 \times X_2$.
\State Compute the convex hull of the set $R$ using \texttt{QHull}, where:
\begin{align*}
    R = \Bigl\{(v_1, v_2, t_0) \;\Big|\; &t_0 \in \{u_1(v_1)u_2(v_2), u_1(v_1)w_2(v_2), w_1(v_1)u_2(v_2), w_1(v_1)w_2(v_2)\}, \\
    &(v_1, v_2) \in \bigcup_{p \in P} p \Bigr\}.
\end{align*}
\State Add the linear constraints defining $\conv(R)$ to $I$.
\end{algorithmic}
\end{algorithm}

In addition, fractions such as $t_0 = \frac{t_1}{t_2}$ are rewritten as $t_1 = t_0\cdot t_2$ and processed by Algorithm~\ref{alg:multiplication node}.

\subsection{Voxelization via Approximated Projection}\label{sec:voxelProject}
We describe our first voxelization algorithm, which computes an approximate projection of the feasible set onto the subspace of the target variables before performing voxelization. This approach avoids the computational expense of exact high-dimensional projections. 

The approximation follows the method described in~\cite{kamenev1996algorithm,lotov2008modified}. Alternative ways for constructing projections were described in \cite{muller2020using}. These were used to identify inequalities for $x_1x_2$ over a polytope as given in \cite{locatelli2018convex}. Given a high-dimensional polytope $P$, we first establish a bounding box for the target variables $(x,y)$ by optimizing $P$ along the $x$ and $y$ directions identifying the optimal solutions. The convex hull of these projected supporting points forms an inner approximation of $\proj_{(x,y)} P$. We then iteratively refine this approximation: computing facet normals of the current inner hull and identifying the facet with maximal distance to the outer boundary of the projection to find an additional supporting point. This process repeats until the maximum distance falls below a tolerance $\epsilon$ or an iteration limit $N_{\text{max}}$ is reached. The resulting polytope, denoted $\AP(P,N_{\max},\epsilon,(x,y))$, serves as the outer-approximation of the projection.

This two-dimensional polyhedral set is efficiently voxelized by discretizing the boundary. Let $V_i$, $i=1,\ldots,n$ be the vertices of $\AP(P, N_{\max},\epsilon, (x,y))$ sorted clockwise with $V_{n+1} = V_1$. The voxelization is the union of rectangles generated along each segment $(V_i, V_{i+1})$. The precision of the approximation is controlled by the number of voxels per segment, $N_V$. We summarize this method in Algorithm~\ref{alg:approx_projection}.

\begin{algorithm}[ht]
\caption{Voxelization via Approximated Projection}
\label{alg:approx_projection}
\begin{algorithmic}[1]
\Require tolerance \(\varepsilon > 0\); maximum number of iterations $N_{\max}$; number of voxels each segment $N_V$; system of inequalities $I$; target variables $(x,y)$
\State Let $I$ be the linear inequalities defining polytope $P$
\State $\Voxels\gets \emptyset$
\State $P_{out}\gets \AP(P,N_{max},\epsilon,(x,y))$
\State $V\gets \text{vertices of } P_{out}$
\State Sort $V$ clockwise; append $V_1$ to $V$ to close the loop
\For{$i = 1$ to $\lvert V\rvert - 1$}
    \For{$j = 0$ to $N_V-1$}
        \State Define rectangle with diagonal corners:
        \[ \left(V_i + \frac{j}{N_V}(V_{i+1}-V_i), \;\; V_i + \frac{j+1}{N_V}(V_{i+1}-V_i)\right) \]
        \State Add this rectangle to $\Voxels$
    \EndFor
\EndFor
\State Return $\Voxels$
\end{algorithmic}
\end{algorithm}

\begin{remark}
    The choice of a 2D projection is motivated by computational efficiency. For an $n$-dimensional target subspace, achieving an outer approximation within Hausdorff distance $\epsilon$ typically requires $\mathcal{O}(1/\epsilon^{n-1})$ linear programs (LPs). In the two-dimensional case, this bound improves to $\mathcal{O}(1/\sqrt{\epsilon})$~\cite{kamenev1996algorithm,lotov2008modified}. We exploit this efficiency to maintain high precision while avoiding the exponential cost associated with higher dimensions.
\end{remark}

\subsection{Voxelization via Quadtree}\label{sec:voxelQuadtree}
The second algorithm voxelizes a region by eliminating voxels exterior to the feasible set.
Given a bounding box $[x^L, x^U] \times [y^L, y^U]$ for the target variables $(x,y)$, we generate an initial grid of $p \times q$ voxels and iteratively prune those entirely outside the domain. The process leverages interval arithmetic and branch-and-bound techniques adapted from \cite{stolte1997robust,stolte2001novel}.

We employ a quadtree approach starting with the full bounding box. For every subgrid $S$ and constraint function $h_i(z)$, we compute the pair of interval bounds $[h_i]_S$ over $S$ using interval arithmetic techniques detailed in Appendix~\ref{variablebound}, incorporating global variable bounds for $z$ and local restrictions that $(x,y)\in S$. If all constraints are certified feasible, $\max([h_i]_S) \le 0$, the subgrid is included in the list of voxels. If any constraint is certified infeasible, $\max_i \min([h_i]_S) < 0$, the subgrid is pruned. Otherwise, if the grid resolution permits splitting, the subgrid is subdivided into four quadrants by choosing a centrally located grid point; if not, it is retained as an unresolved leaf. The procedure is summarized in Algorithm~\ref{VQ}.

\begin{algorithm}[t]
\caption{Quadtree Voxelization}\label{VQ}
\begin{algorithmic}[1]
\Require target components $(x,y)$ of variable vector $z$; bounding box of $z$; grid over $(x,y)$: $\{x_1,\dots,x_p\}\times\{y_1,\dots,y_q\}$; Inequalities system $I$
\State $\{h_i(z)\leq 0\}_{i\in [m]}\gets \text{inequalities in }I$
\State $Q \gets \{(1,p,\,1,q)\}$ \Comment{Queue of subgrids indexed by $(i_L,i_U,j_L,j_U)$}
\State $\text{Voxels} \gets \emptyset$
\While{$Q \neq \emptyset$}
  \State Pop $(i_L,i_U,j_L,j_U)$ from $Q$
  \State Define subgrid $S \gets [x_{i_L},x_{i_U}] \times [y_{j_L},y_{j_U}]$
  \For{$i=1,\dots,m$}
    \State Compute interval bounds $[h_i]_S$ of the constraint function $h_i(z)$ over $S$ using interval arithmetic (Appendix~\ref{variablebound}) with global variable bounds on $z$ and $(x,y)\in S$
  \EndFor
  \If{$\max_{1\le i\le m}\ \max([h_i]_S)\le 0$} \Comment{Certified inside}
     \State $\text{Voxels} \gets \text{Voxels}\cup\{\text{voxels in $S$}\}$
  \ElsIf{$\exists\, i \text{ with }\min([h_i]_S)>0$} \Comment{Certified outside}
     \State \textbf{skip} $S$
  \ElsIf{$(i_U-i_L\le 1)\ \lor\ (j_U-j_L\le 1)$} \Comment{Grid cannot be further split}
     \State $\text{Voxels} \gets \text{Voxels}\cup\{S\}$ \Comment{Unresolved leaf}
  \Else \Comment{grid-aligned split into four}
     \State $i_M \gets \big\lfloor\tfrac{i_L+i_U}{2}\big\rfloor$, \quad $j_M \gets \big\lfloor\tfrac{j_L+j_U}{2}\big\rfloor$
     \State Push $(i_L,i_M,\,j_L,j_M)$, $(i_M,i_U,\,j_L,j_M)$, $(i_L,i_M,\,j_M,j_U)$, and $(i_M,i_U,\,j_M,j_U)$ to $Q$
  \EndIf
\EndWhile
\State \Return $\text{Voxels}$
\end{algorithmic}
\end{algorithm}
Practically, $p$ and $q$ are set to $3$ with $(x_2,y_2)$ being the middle point of the bounding box of $(x,y)$. Though this method can handle nonlinearity and nonconvexity, the complexity increases substantially with $p$ and $q$.

\section{Numerical Results}\label{sec:numerical}
In this section, we present numerical results for the proposed voxelization-based relaxation (VR) on two problem classes. First, we evaluate VR against composite relaxations on polynomial optimization instances from~\cite{he2024mip}. Second, we compare VR with our baseline (Base) and Gurobi's root-node results, reporting the gap closed by each relaxation. We also investigate several algorithmic choices in VR's construction: the relative performance of the projection-based versus quadtree voxelization methods (Sections~\ref{sec:voxelProject} and \ref{sec:voxelQuadtree}), and the trade-off between relaxation quality and computational cost as voxelization is done at finer resolution. Experiments were conducted on a MacBook Pro (Apple M4 Pro, 12-core CPU, 24 GB RAM). Code was written in \texttt{Julia} v1.11~\cite{Julia-2017}, with relaxations modeled via \texttt{JuMP} v1.27.0~\cite{lubin2023jump}. We used \texttt{Ipopt} v1.10.6~\cite{wachter2006implementation} for upper bounds and \texttt{Gurobi} v11.7.5~\cite{gurobi} as the primary solver.

\subsection{Polynomial Optimization Problem}\label{sec:polynomial}
Consider polynomial optimization problems:
\begin{align*}
        \min\left\{\langle c,x\rangle+\langle d,y\rangle\mid Ax+By\leq b, x^L\leq x\leq x^U, y = (x^{\alpha_1},x^{\alpha_2},\cdots,x^{\alpha_m})\right\},\numberthis\label{poly opt}
\end{align*}
where $x\in\R^n$, $x^L,x^U\in \R^n$, $x^{\alpha_j} = \prod^n_{i = 1}x_i^{\alpha_{ji}}$, $\alpha_j = (\alpha_{j1},\cdots,\alpha_{jn})$, and parameters $c\in\R^n$, $d\in\R^m_+$, $A\in\R^{r\times n}$, $B\in\R_+^{r\times m}$, $b\in\R^q$. Problem sizes are denoted by $(n,m,r)$. For each tuple in \[\{(15,30,20),(25,50,20),(50,100,20),(100,200,20)\},\] we generate $100$ instances of Equation~\eqref{poly opt} using the procedure detailed in Appendix~\ref{app:generation}.

\begin{figure}[htbp]
\centering

\newcommand{\colw}{0.45\textwidth}
\newcommand{\vsep}{1ex}

\begin{tabular}{@{}c c@{}}

\begin{subfigure}[t]{\colw}\centering
\includegraphics[width=\linewidth]{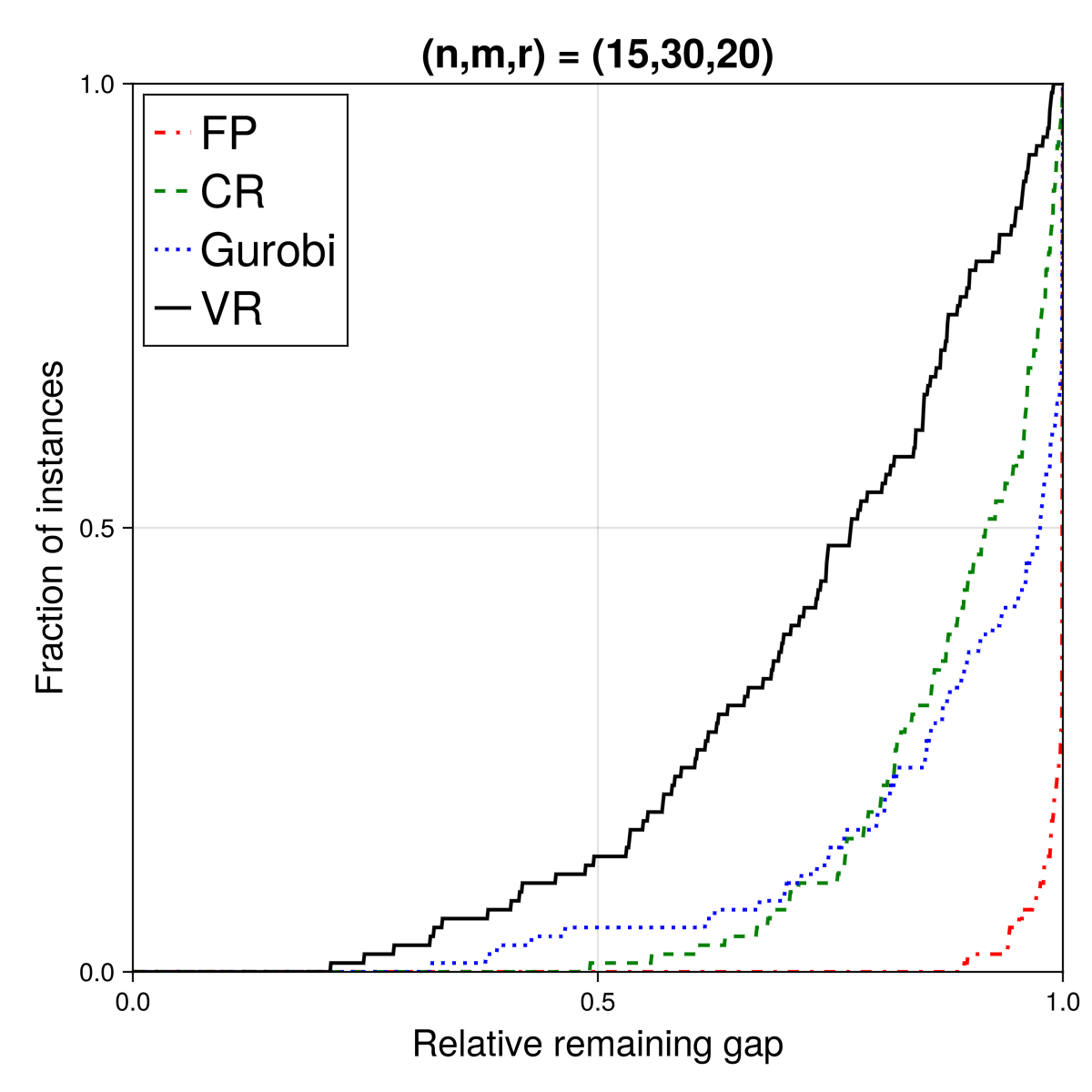}
\end{subfigure} &
\begin{subfigure}[t]{\colw}\centering
\includegraphics[width=\linewidth]{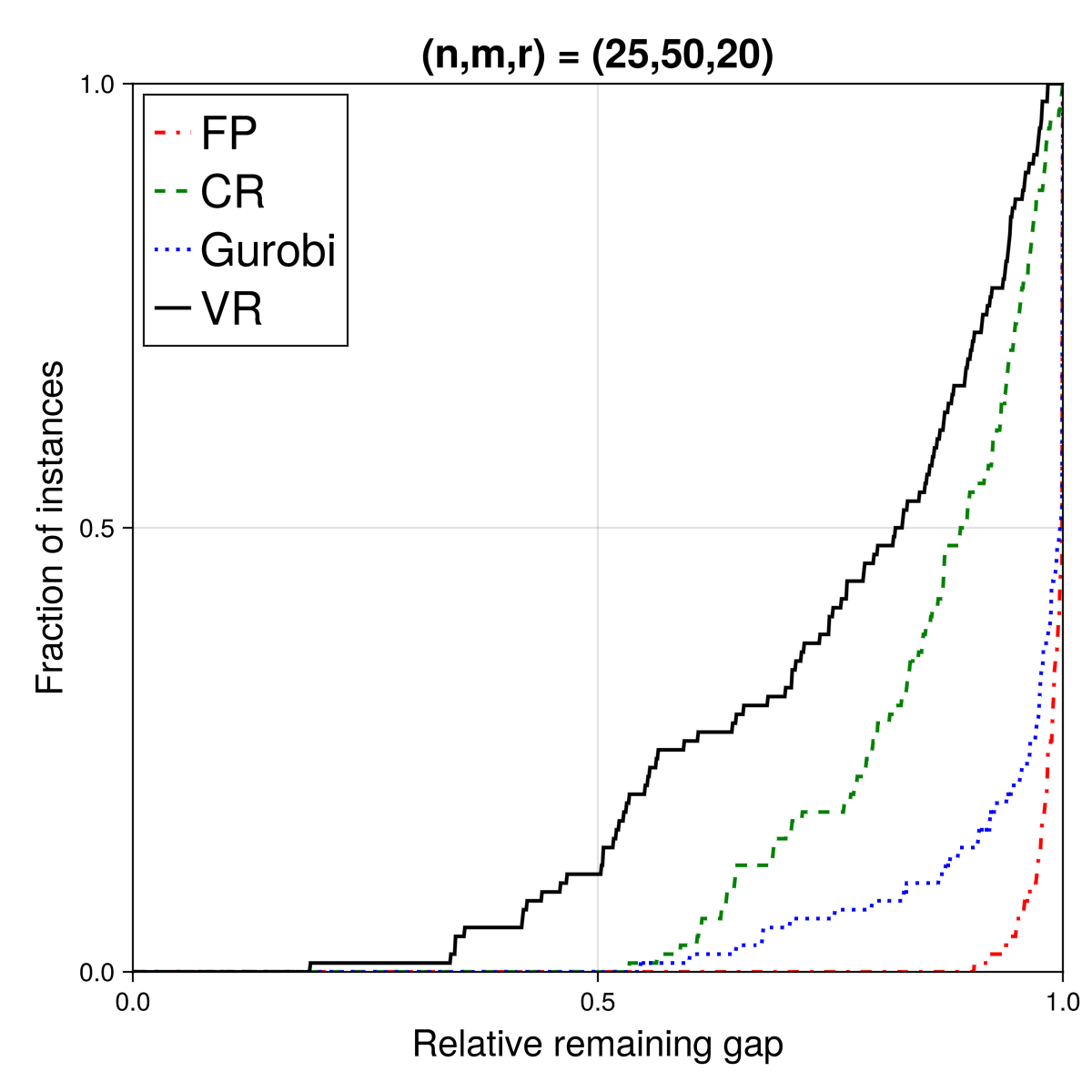}
\end{subfigure} \\[\vsep]

\begin{subfigure}[t]{\colw}\centering
\includegraphics[width=\linewidth]{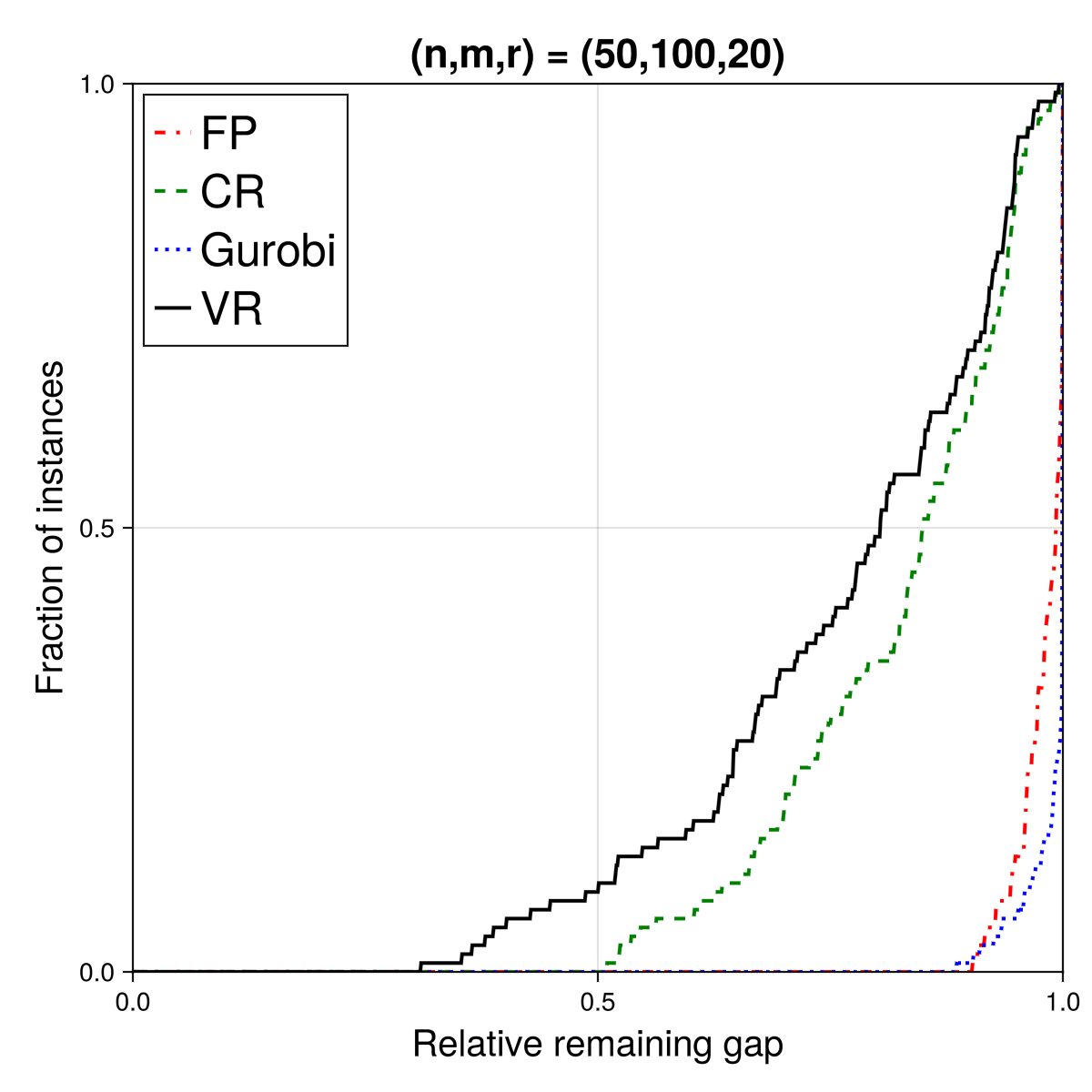}
\end{subfigure} &
\begin{subfigure}[t]{\colw}\centering
\includegraphics[width=\linewidth]{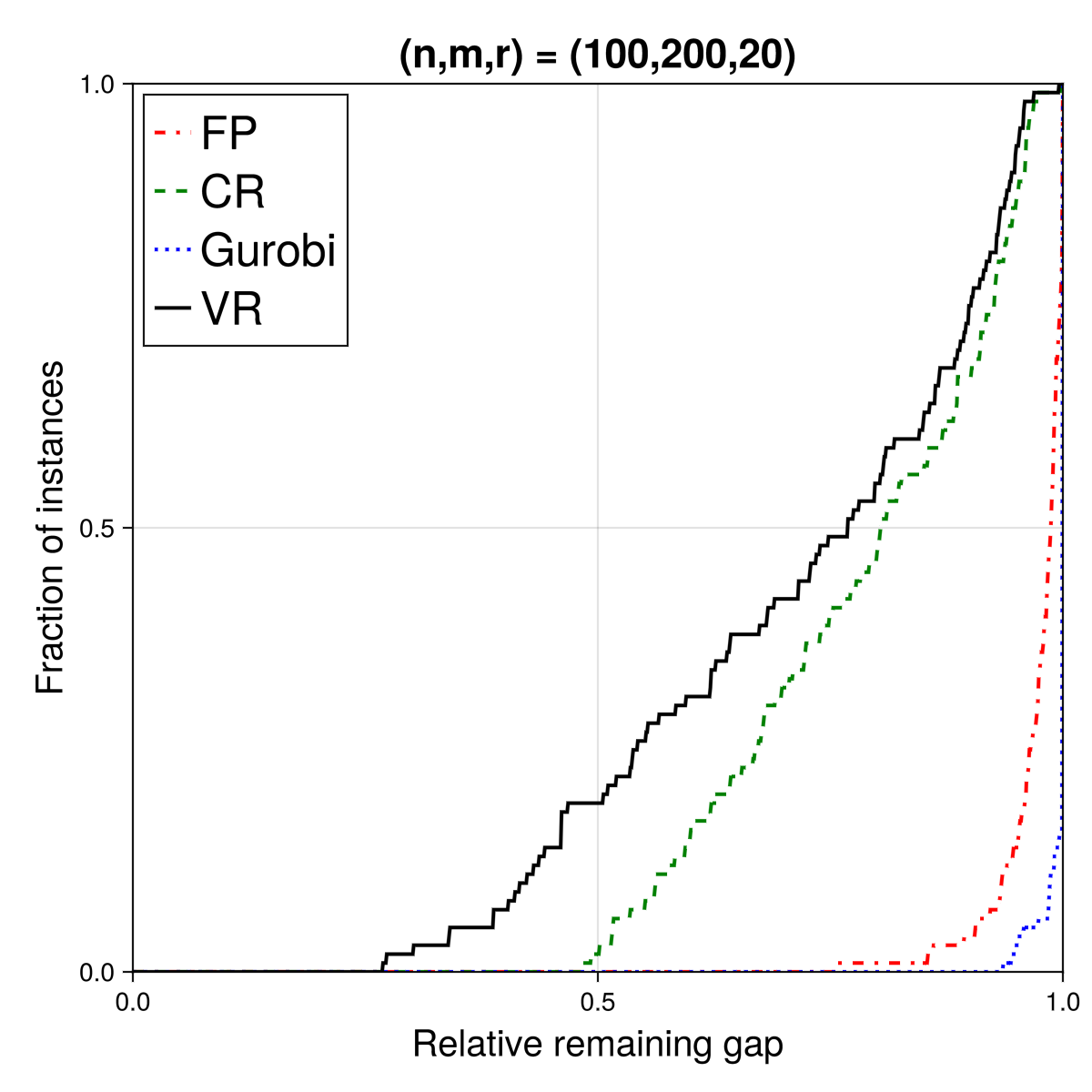}
\end{subfigure} 
\end{tabular}

\caption{Comparison with Relative Remaining Gap}
\label{fig:relaxation-cdfs-cr}
\end{figure}

We evaluate four relaxation settings:
\begin{itemize}
    \item \texttt{Gurobi}: Root-node relaxation ($\texttt{Nodelimit} = 0$) using default settings (cuts and range reduction enabled). Bounds are reported before the solver branches. %
    \item \texttt{FP}: Products are relaxed via standard McCormick envelopes~\cite{mccormick1976computability}.
    \item \texttt{CR}: Composite relaxation from~\cite{he2024mip} ($3$ underestimators and a maximum of $5$ separation calls.
    \item \texttt{VR}: Base relaxation scheme described in Section~\ref{alg:factorable relaxation} with multiplication nodes handled by Algorithm~\ref{alg:multiplication node} with $N_b = 5$, relaxing over cubes without voxelization.
\end{itemize}
\begin{table}[ht]
\begin{center}
\begin{tabular}{ | c |c  c  c  c | }
   \hline
 $(n,m,r)$  &  $\texttt{FP}$ & $\texttt{CR}$ & $\texttt{Gurobi}$ & $\texttt{VR}$\\ \hline
$(15,30,20)$&0.003& 0.168	&	0.021&0.118  \\ 
$(25,50,20)$&0.007& 0.181&	0.030&0.207 	\\ 
$(50,100,20)$&0.013 & 0.354 &0.072 &0.492 \\ 
$(100,200,20)$&0.033 & 0.679 &0.160 &1.113 \\ \hline
\end{tabular}
\end{center}
\caption{ Average construction and solution times of relaxations on $100$ instances (seconds).}~\label{table:computationtime}
\end{table}

\begin{table}[ht]
\begin{center}
\begin{tabular}{ | c |c  c  c | }
   \hline
 $(n,m,r)$  &  $\texttt{FP}$ & $\texttt{CR}$ &  $\texttt{VR}$\\ \hline
$(15,30,20)$&0.0019&0.0086&0.0086\\ 
$(25,50,20)$&0.0036&0.0188&0.0194\\ 
$(50,100,20)$&0.0097&0.0615&0.0645\\ 
$(100,200,20)$&0.0262&0.1746&0.1198\\ \hline
\end{tabular}
\end{center}
\caption{ Average solving times of relaxations on $100$ instances (seconds).}~\label{table:solvingtime}
\end{table}
We compare the strength using the relative remaining gap. For instance $p$, let $v_{p,i}$ be the lower bound from relaxation $i$ and $u_p$ the upper bound from \texttt{Ipopt}. Using $\min_i\{v_{p,i}\}$ as a reference, the relative gap for relaxation $i$ is:
\begin{equation}\label{relative remaining gap}r_{p,i} = \frac{u_p - v_{p,i}}{u_p - \min_{i'}\{v_{p,i'}\}}.\end{equation}
We define $\mu_i(\alpha)$, as the fraction of instances where $r_{p,i} \le \alpha$:
\begin{equation}\label{fraction rrg}
    \mu_i(\alpha) = \frac{1}{\lvert P\rvert}\lvert \{p \in P \mid r_{p,i} \leq \alpha\} \rvert.
\end{equation}
    
In Figure~\ref{fig:relaxation-cdfs-cr}, curves further left and higher indicate stronger relaxations; VR significantly outperforms the others. Table~\ref{table:computationtime} shows average solution times. For small instances, VR and CR costs are comparable. For large instances, VR improves relaxation quality at roughly double the computational cost. Note that these include the time to construct the relaxation. However, once constructed, solving time for VR is comparable (see Table~\ref{table:solvingtime}), which shows that VR relaxation takes 5-7 times the time to solve FP relaxation. Given that VR reduces the gap significantly and the time to solve the relaxation is reasonable encourages its use in branch-and-bound solvers, especially for hard instances.
\subsection{MINLPLib Instances}\label{sec:minlplib}

We evaluate VR on publicly available MINLPLib instances~\cite{bussieck2003minlplib}. Our default VR configuration is constructed as follows: 
\begin{enumerate} 
\item It builds upon the baseline relaxation (Algorithm~\ref{alg:factorable relaxation}), where multiplication nodes processed via Algorithm~\ref{alg:multiplication node} with enhancements described in Section~\ref{BASE}. Duality-based reduction is performed on relaxation solution~\cite{tawarmalani2004global}. 
\item \texttt{Ipopt} is used to find a local solution to invoke OBTT for range reduction~\cite{bestuzheva2023global}.
\item Univariate functions are approximated by piecewise linear functions with $N_b=9$ breakpoints. 
\item  Voxelization uses approximate projection method described in Algorithm~\ref{alg:approx_projection}, with $N_{\text{max}} = 5$ LPs per projection. 
\item Each segment is voxelized into $N_V = 5$ voxels. 
\item The algorithm executes a single iteration.
\end{enumerate}

Under this setting, VR demonstrates exceptional performance on several challenging MINLPLib instances, outperforming the best bounds from mainstream solvers reported in the library even after branch-and-bound. We first showcase these results, then report aggregate statistics comparing VR against the baseline relaxation and Gurobi's root-node bound. Finally, we conduct a sensitivity analysis on VR's hyperparameters to assess their impact on tightness and computational efficiency.

\begin{table}[ht]
\centering
\begin{tabular}{|c|ccccc|c|}
\hline
Instance & VR  &BARON &Gurobi & SCIP &Primal &VR Time(s)\\
\hline
cesam2log &-254.60  & missing &-600.32&-389.02&0.50&7.72\\
beuster &95237.22 & 103117.3 &105174.35&55563.78&116329.67&0.93\\
camshape200 &-4.89 & -4.97 &-4.79&-4.83&-4.27&32.44\\
camshape400 &-5.07 & -5.17 &-4.99&-5.04&-4.27&124.79\\
camshape800 &-5.22 & -5.27 &-5.12&-5.19&-4.27&528.97\\
transswitch0039r &25383.99 & 565.14 &77.98&27081.57&41866.12&77.53\\
\hline
\end{tabular}%
\caption{Showcase instances (based on MINLPLib.org)}\label{table: showcase instances}
\end{table}

Table~\ref{table: showcase instances} compares the optimal value of the VR relaxation against the best dual bounds reported in \texttt{minlplib} after branch-and-bound by BARON, Gurobi, and SCIP. These instances are nonconvex, arise from practical applications, and remain unsolved in MINLPLib.

We experiment with VR across $619$ MINLPLib instances, translated into \texttt{JuMP} models by \texttt{MINLPLib.jl}. Instances are selected if they satisfy the following restrictions:
\begin{enumerate}
    \item Have fewer than $5000$ variables and $500$ nonlinear constraints,
    \item Are not archived on \texttt{MINLPLib.org},
    \item Operators are limited to $\{+,-,\times,\div,\power,\log(\cdot),\exp(\cdot)\}$,
    \item \texttt{Julia} translation is correct\footnote{For some instances, bounds imposed on variables differed between translated and original models.},
    \item Nonlinear multiplications and/or divisions are present.
\end{enumerate}

For comparison, we convert all problems to minimization by negative the objective if necessary. To compare two relaxations $R_1$ and $R_2$ we use the relative gap closed (RCG). Let $p_1$ and $p_2$ be the optimal values obtained by $R_1$ and $R_2$, respectively, and $p^{\text{primal}}$ be the best known primal value. If $p_1\ge p_2$ (\textit{i.e.}, $R_1$ provides a better dual bound), we define:

\[\RCG(R_1, R_2) = \frac{p_1 - p_2}{p^{\text{primal}} - p_2}. \]

Table~\ref{minlp percentage} reports the percentage of instances where Base and VR—achieve:
\begin{enumerate}
    \item $\geq 1\%$ improvement over Gurobi's root-node ($\RCG(\cdot,\texttt{Gurobi})$),
    \item $\geq 1\%$ inferior relative to Gurobi root-node ($\RCG(\texttt{Gurobi},\cdot)$),
    \item Negligible performance difference.
\end{enumerate}
\begin{table}[ht]
\begin{center}
\begin{tabular}{ | c |c  c c | }
   \hline
 Relaxation  &  $\geq 1\%$ improvement & $\geq 1\%$ inferior & negligible difference \\ \hline
\texttt{Base} &10.96& 26.77	&62.27\\ 
\texttt{VR}&29.19 & 19.35&51.46 \\ \hline
\end{tabular}
\end{center}
\caption{Relative Closed Gap w.r.t Gurobi Root-Node}~\label{minlp percentage}
\end{table}

As detailed in Section~\ref{BASE}, the Base relaxation incorporates several enhancements into factorable programming like range reduction, expression simplification, and expression match. We use it a a benchmark to isolate VR's algorithmic contributions.
While the Base relaxation outperforms Gurobi on over $10\%$ of the instances, it underperforms on 26.77\% of instances, making them incomparable benchmarks. Notably, VR improves upon Gurobi by at least 1\% on 29.19\% of instances. Many cases where Gurobi excels involve latent structures where convexity identification, coefficient reduction, and integer-programming cuts offer opportunities, which our current implementation does not detect.

We compare VR and Gurobi using the proportion of instances $\mu(\alpha)$ (Equation~\eqref{fraction rrg}) where the relative remaining gap~(Equation~\eqref{relative remaining gap}) is at least $\alpha$ (Figure~\ref{Relative Remaining Gap Distribution MINLPLib}). In cumulative distribution plots, curves higher and to the left indicate stronger relaxations. VR dominates Gurobi when $\alpha\ge 31\%$, indicating superior performance on problems where significant gaps remain. Gurobi's superiority is largely confined to cases with extremely small gaps (near-solved instances) or problems where VR fails to derive reasonable bounds due to missing variable bounds in preprocessing. Both categories often include special structures that we conjecture are detected by Gurobi.

To demonstrate that VR's improvements stem from unique techniques, we consider the best bound obtained by either solver (green dashed line). Since this line exceeds Gurobi's performance by a significant margin, we conclude that VR's techniques significantly enhance the standalone root-node relaxation in the solver.

\begin{figure}[htbp]
\centering
\includegraphics[scale=0.3]{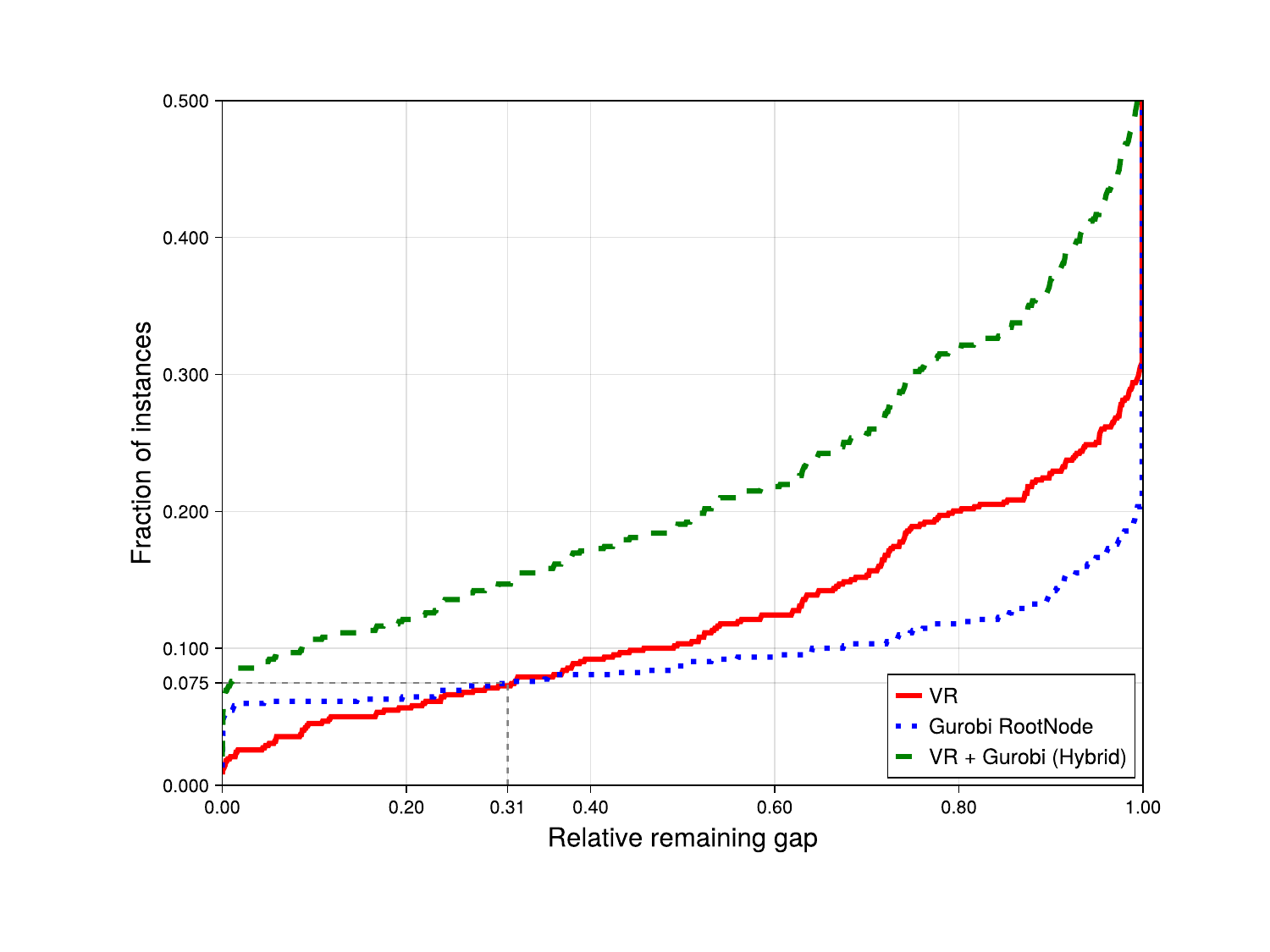}
\caption{Relative Closed Gap Distribution: VR v.s Gurobi Root-node Across MINLPLib Instances}
\label{Relative Remaining Gap Distribution MINLPLib}
\end{figure}

We further analyze the performance of VR considering the problems in which its performance improves on the benchmark (although the percentage is expressed over the full set of $619$ instances). As we mentioned before, VR improves over Gurobi on 29.9\% of the instances which corresponds to the $y$-intercept in Figure~\ref{fig:compareVRandGurobi}.

\begin{figure}[htbp]
\centering

\newcommand{\colw}{0.5\textwidth}
\newcommand{\vsep}{1ex}

\begin{tabular}{@{}c c@{}}

\begin{subfigure}[t]{\colw}\centering
\includegraphics[width=\linewidth]{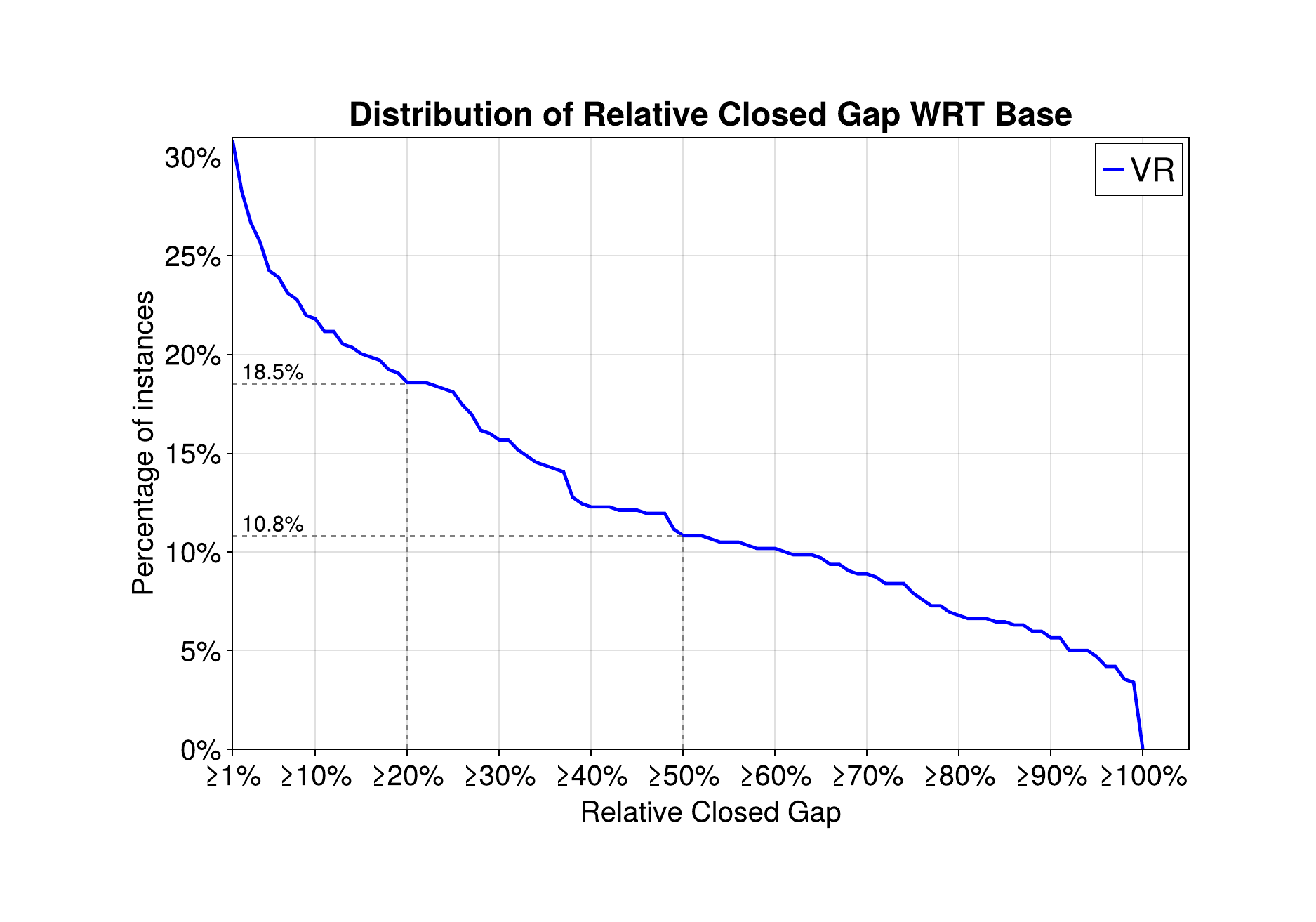}
\caption{VR relative to Base}\label{fig:compareVRandBase}
\end{subfigure} &
\begin{subfigure}[t]{\colw}\centering
\includegraphics[width=\linewidth]{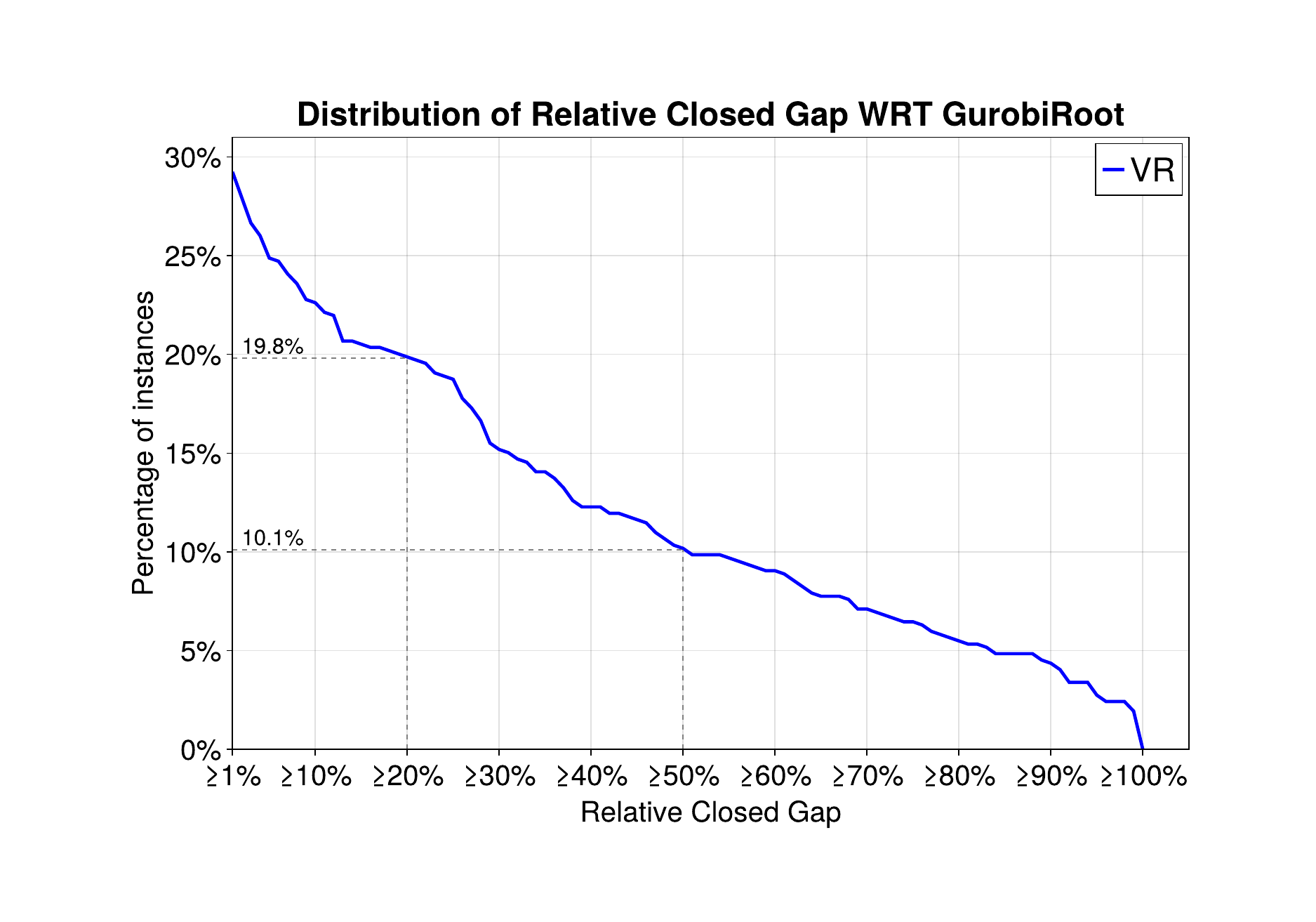}
\caption{VR relative to Gurobi}\label{fig:compareVRandGurobi}
\end{subfigure} 
\end{tabular}
\caption{Relative Closed Gap Distribution}
\label{Relative Closed Gap Distribution, Con}
\end{figure}

Key observations include:
\begin{itemize}
    \item Compared to Base, VR reduces the remaining gap by $> 20\%$ for $18.5\%$ of instances and by $> 50\%$ for $10.8\%$;
    \item Compared with Gurobi, VR reduces the gap by $> 20\%$ for $19.8\%$ instances and by $> 50\%$ for $10.1\%$.
\end{itemize}
The similarity in these distributions suggests VR's gains are broad. Specifically, VR's ability to improve Gurobi correlates strongly with its ability to improve Base, indicating that it captures structures missed by standard factorable programming. Importantly, these improvements do not introduce exceesive constraints or variables; most inequalities are sparse, and auxiliary variables match those of the Base relaxation, highlighting VR's potential as a valuable root-node relaxation.

Table~\ref{solving construction time} reports average solving and construction times. While VR takes longer to construct, solving times remain comparable. VR is designed for strong root-node relaxations; it need not be recomputed at every branch-and-bound node. Opportunities exist to memoize finite point sets or selectively refine formulations based on relaxation solutions, which we leave for future work.
\begin{table}[ht]
    \centering
    \begin{tabular}{|c|c c c|}
        \hline
         &Base  & VR& VRQ \\
        \hline
        Average Solving Time(s)&0.004 & 0.008 &0.008 \\
        Average Construction Time(s)&0.21& 5.01 &11.12 \\
        \hline
    \end{tabular}
    \caption{Average Solving and Construction Time of Base Relaxation, VR, and VRQ}\label{solving construction time}
\end{table}

\newcommand{\colw}{0.5\textwidth}  %
\newcommand{\imgw}{\linewidth}      %
\newcommand{\vsep}{0.8em}           %
\begin{figure}[t]
\centering

\begin{tabular}{@{}c c @{}}
\begin{subfigure}[t]{\colw}\centering
\includegraphics[scale = 0.3]{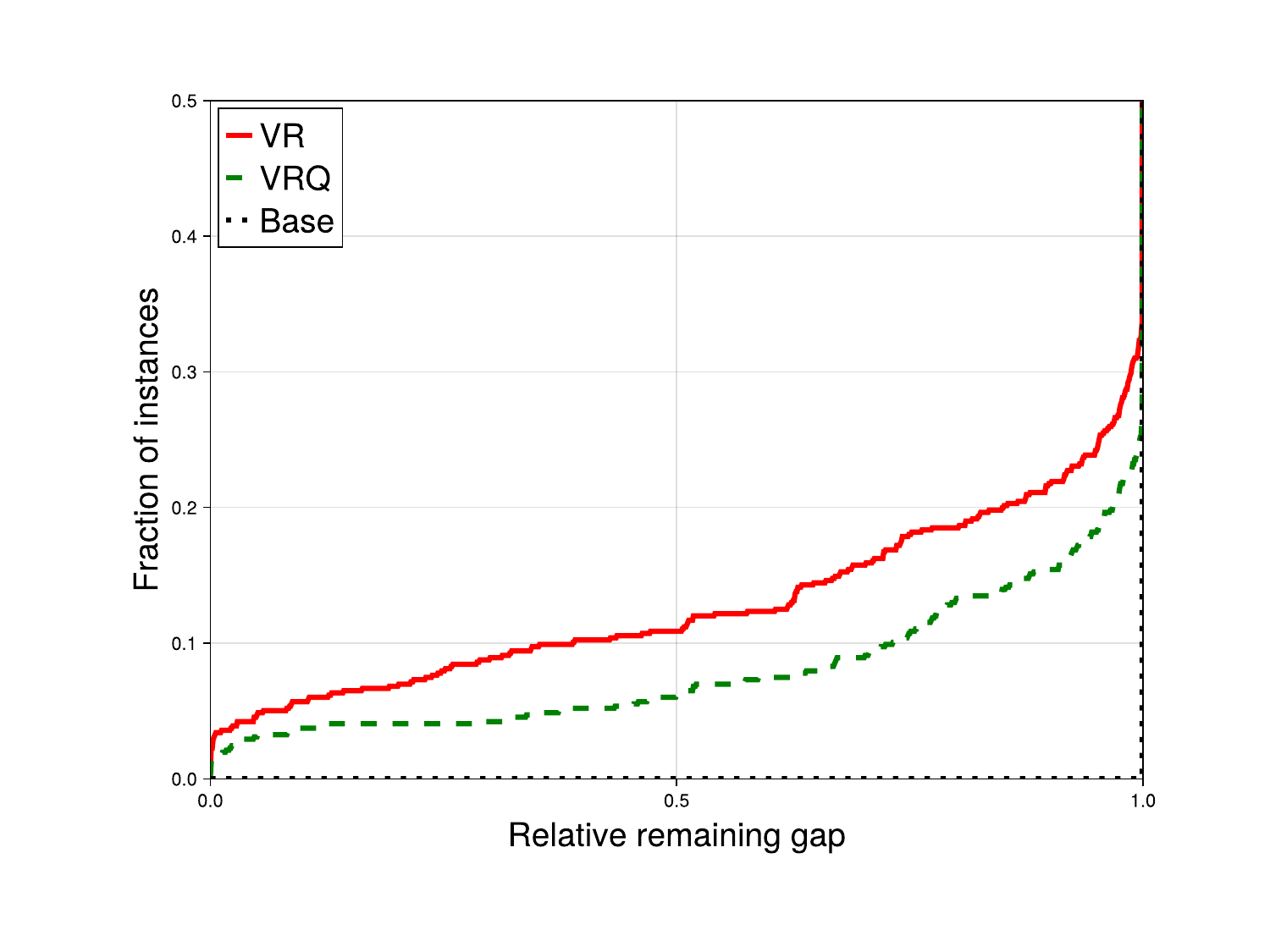}
\caption{VR v.s VRQ}
\label{Relative Remaining Gap Distribution MINLPLib VRQ}
\end{subfigure}
&
\begin{subfigure}[t]{\colw}\centering
\includegraphics[scale = 0.3]{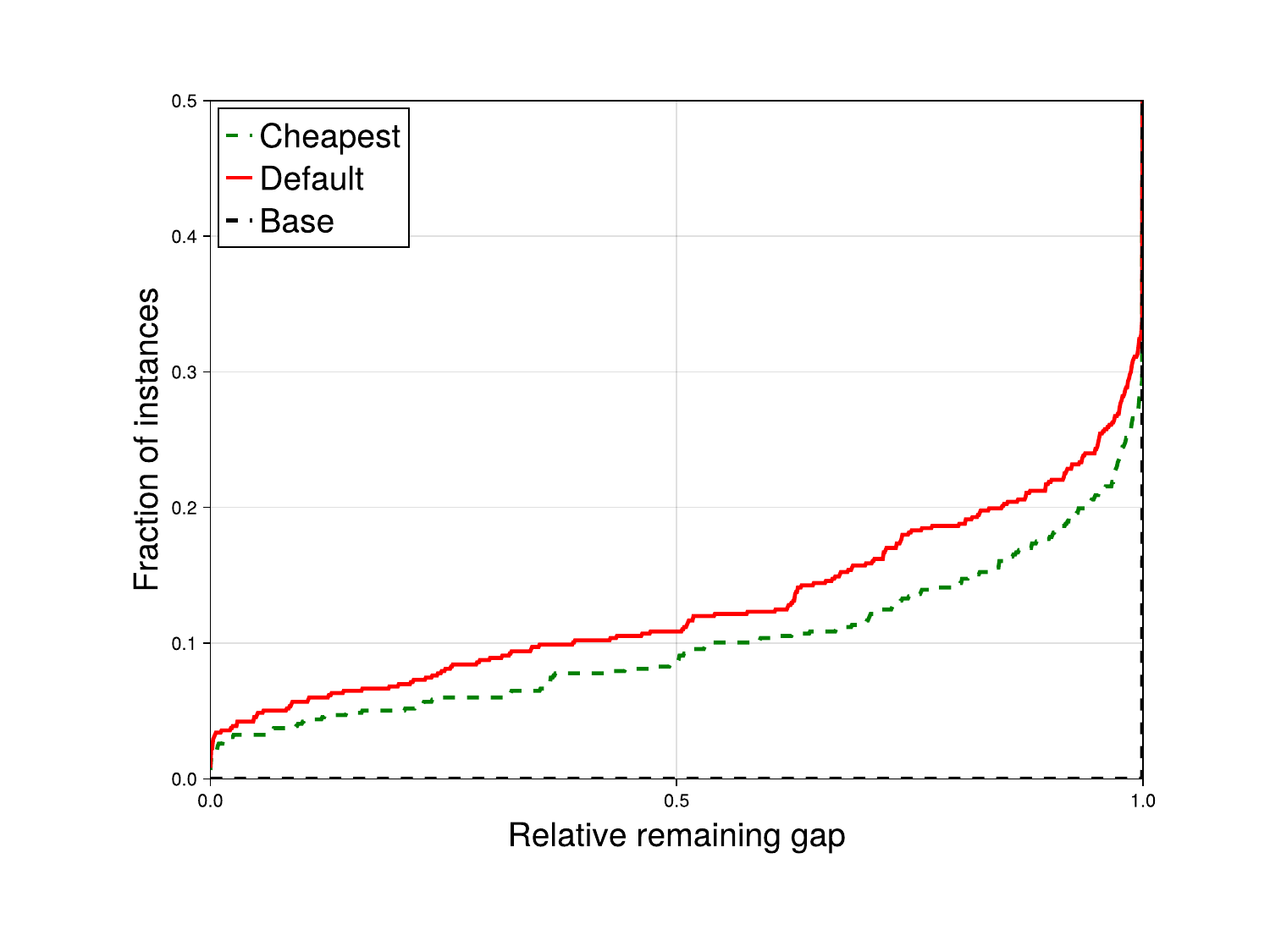}
\caption{Default v.s Cheapest}
\label{Relative Remaining Gap Distribution MINLPLib DC}
\end{subfigure}
\\[\vsep]

\begin{subfigure}[t]{\colw}\centering
\includegraphics[scale = 0.3]{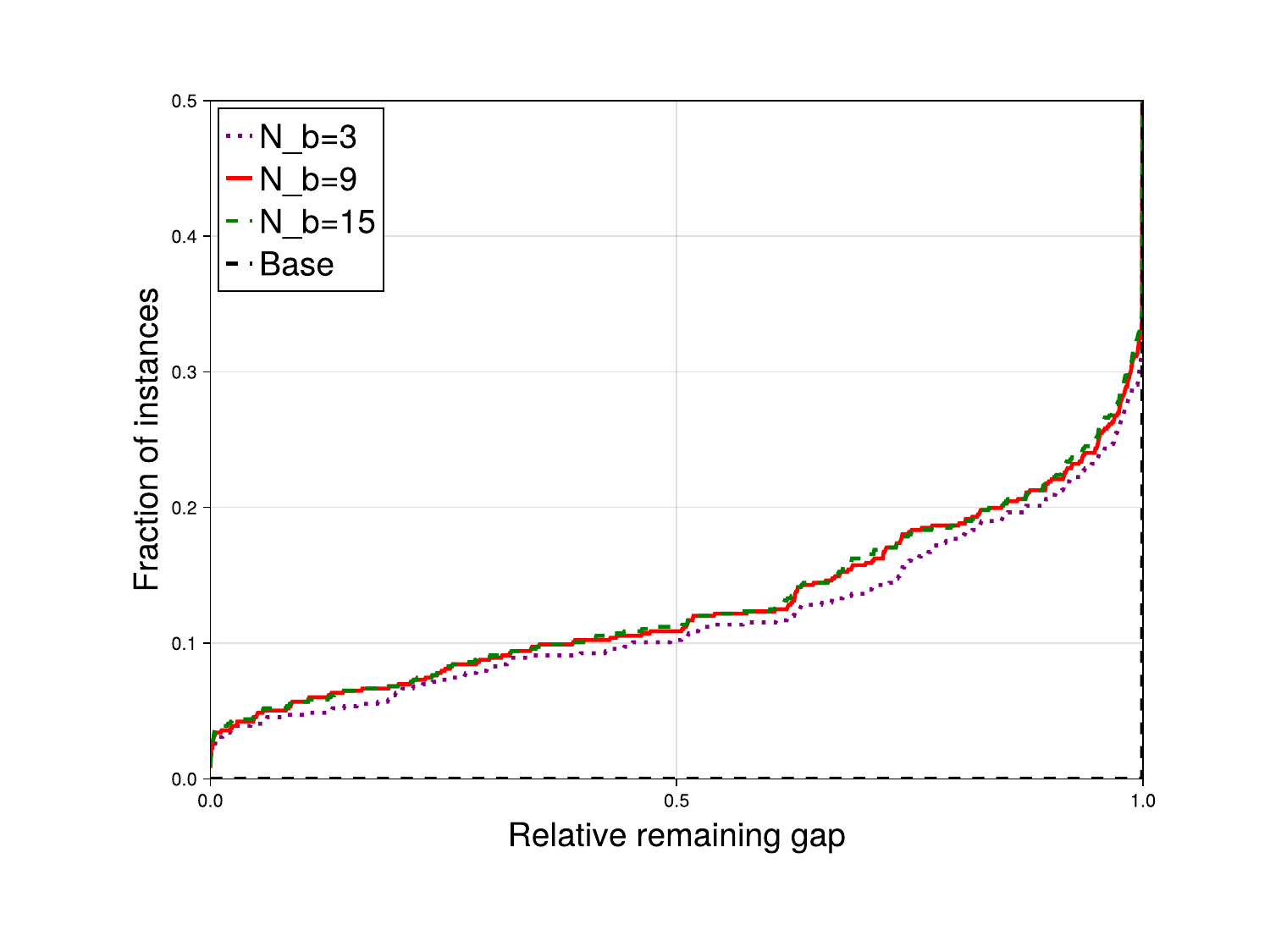}
\caption{Parameter: $N_b$}
\label{Relative Remaining Gap Distribution MINLPLib N_b}
\end{subfigure}

&
\begin{subfigure}[t]{\colw}\centering
\includegraphics[scale = 0.3]{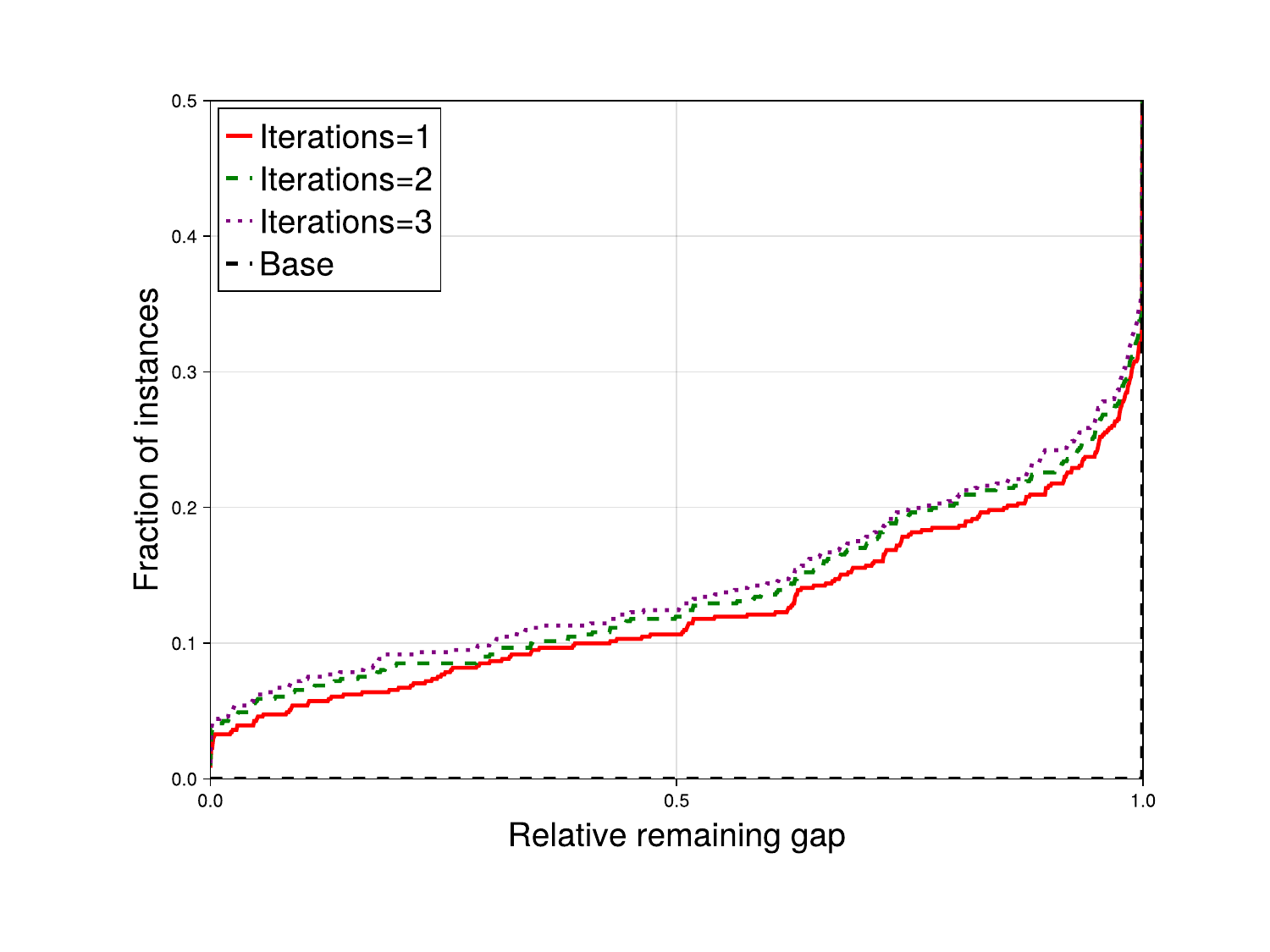}
\caption{Parameter: Iterations}
\label{Relative Remaining Gap Distribution MINLPLib ite}

\end{subfigure}
\\[\vsep]
\begin{subfigure}[t]{\colw}\centering
\includegraphics[scale = 0.3]{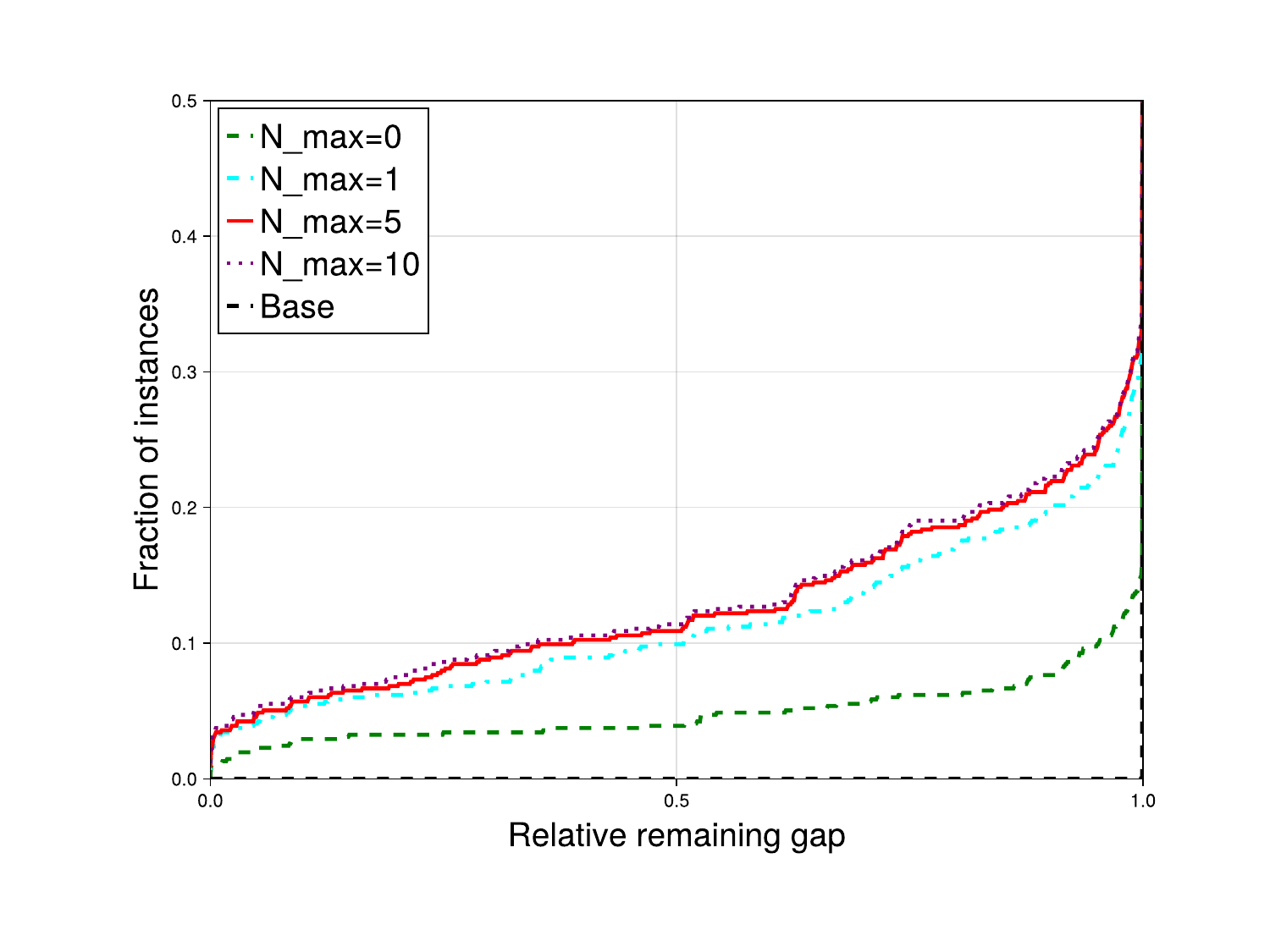}
\caption{Parameter: $N_{max}$}
\label{Relative Remaining Gap Distribution MINLPLib N_{max}}

\end{subfigure}
&
\begin{subfigure}[t]{\colw}\centering
\includegraphics[scale = 0.3]{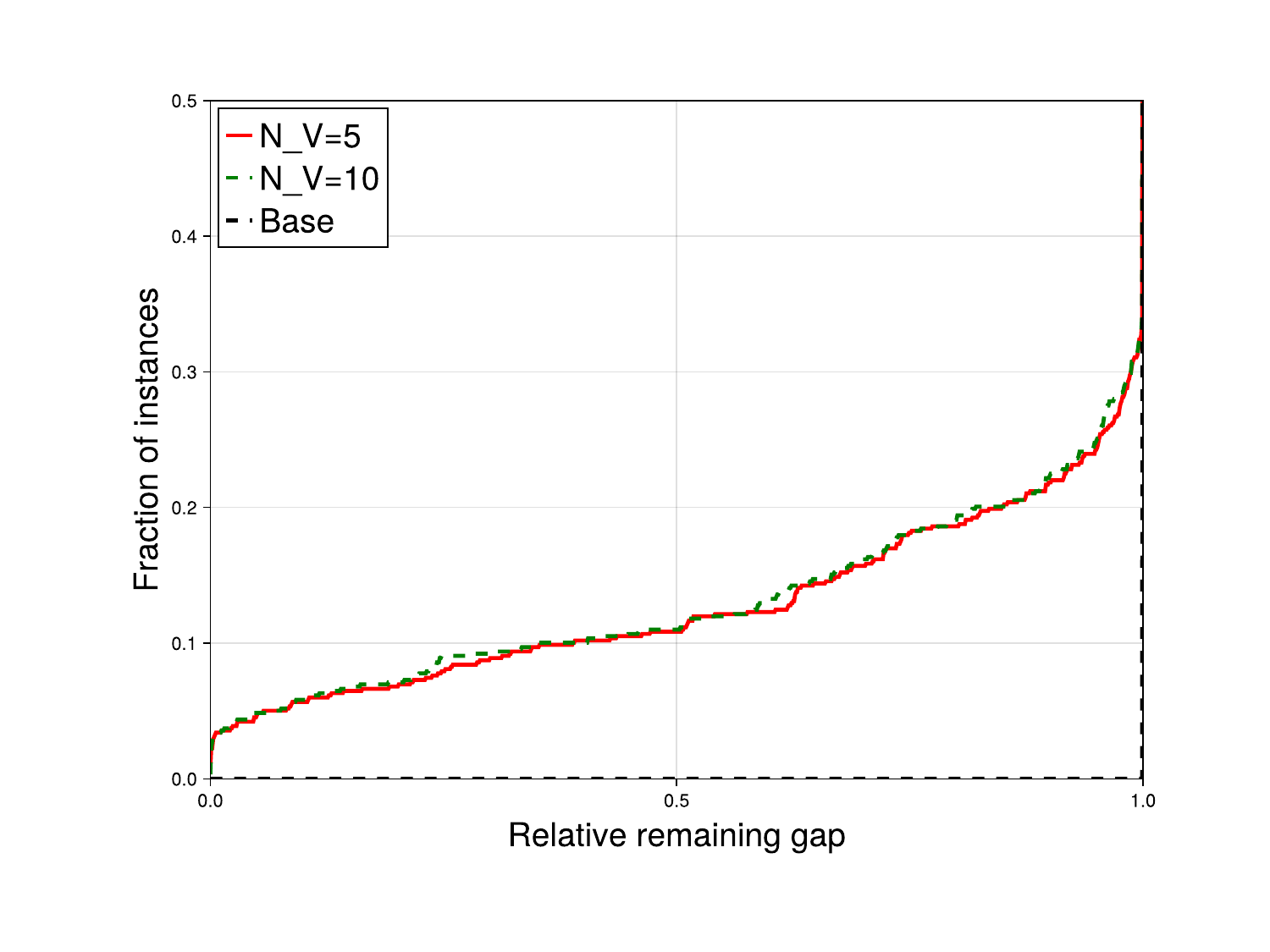}
\caption{Parameter: $N_V$}
\label{Relative Remaining Gap Distribution MINLPLib N_V}

\end{subfigure}
\end{tabular}
\caption{Parameter Study of VR with Relative Closed Gap Distribution}
\end{figure}
We default to approximate projection algorithm for voxelization (Section~\ref{sec:voxelProject}). We compare this against the quadtree-based scheme (Section~\ref{sec:voxelQuadtree}, abbreviated as VRQ. Figure~\ref{Relative Remaining Gap Distribution MINLPLib VRQ} shows the relative remaining gap and Table~\ref{solving construction time} reports total times. These results justify our default choice: VRQ is inferior in both relaxation quality and computational efficiency.

Finally, we conduct sensitivity analysis on VR's hyperparameters:
\begin{itemize}
    \item Number of breakpoints $N_b$ for univariate functions.
    \item Number of LPs $N_{\max}$ per approximate projection
    \item Number of VR iterations. 
    \item Number of voxels $N_V$ per segment.
\end{itemize}

We use $\mu(\alpha)$ (Equation~\ref{fraction rrg}) with Base as the reference.
First, we compare the default setting with a computationally cheaper setting ($N_b = 2$, $N_{\max} = 1$, Iterations = $1$, $N_V = 5$) in Figure~\ref{Relative Remaining Gap Distribution MINLPLib DC} and Table~\ref{solving construction time DC} showing that the construction time is significantly larger than the solving time for the relaxations and that both drop by a factor of two for the cheaper setting.

\begin{table}[ht]
    \centering
    \begin{tabular}{|c|c c|}
        \hline
         &Cheapest&Default\\
        \hline
        Average Solving Time(s)&0.004 & 0.008\\
        Average Construction Time(s)&2.61& 5.01\\
        \hline
    \end{tabular}
    \caption{Average Solving and Construction Time: Default v.s Cheapest}\label{solving construction time DC}
\end{table}

Varying $N_b$ (Figure~\ref{Relative Remaining Gap Distribution MINLPLib N_b}, Table~\ref{solving construction time N_b}) shows moderate improvement from $3$ to $9$, but negligible gains from $9$ to $15$. We select $N_b=9$ as default, though $N_b=3$ offers a computationally efficient alternative.

\begin{table}[ht]
    \centering
    \begin{tabular}{|c|c c c|}
        \hline
         &$N_b = 3$  & $N_b = 9\;(\text{Default})$&$N_b = 15$ \\
        \hline
        Average Solving Time(s)&0.007 & 0.008&0.008 \\
        Average Construction Time(s)&4.82& 5.01&5.15 \\
        \hline
    \end{tabular}
    \caption{Average Solving and Construction Time: VR with different $N_{b}$}\label{solving construction time N_b}
\end{table}

Varying $N_{\max}$ (Figure~\ref{Relative Remaining Gap Distribution MINLPLib N_{max}} and Table~\ref{solving construction time N_{max}}) reveals a significant jump from $0$ to $1$ LP. Increasing from $1$ to $10$ yields slight improvements at triple the cost. We default to $N_{\max} = 5$, noting that $N_{\max} = 1$ offers low computational overhead. This suggests diminishing returns for finer approximations. 

\begin{table}[ht]
    \centering
    \begin{tabular}{|c|c c c c|}
        \hline
         &$N_{\max} = 0$&$N_{\max} = 1$  & $N_{\max} = 5\;(\text{Default})$&$N_{\max} = 10$ \\
        \hline
        Average Solving Time(s)&0.005&0.005 & 0.008&0.008 \\
        Average Construction Time(s)&0.87&2.63& 5.01&7.40 \\
        \hline
    \end{tabular}
    \caption{Average Solving and Construction Time: VR with different $N_{max}$}\label{solving construction time N_{max}}
\end{table}

Varying number of iterations of VR (Figure~\ref{Relative Remaining Gap Distribution MINLPLib ite} and Table~\ref{solving construction time ite} shows non-negligible benefits but at high computational cost. Our current implementation does not reuse information across rounds, leading to significant overhead. Future work could address this and explore whether multiple rounds may offer computational benefits.

\begin{table}[ht]
    \centering
    \begin{tabular}{|c|c c c|}
        \hline
         &$\text{Iterations} = 1(\text{Default})$  & $\text{Iterations}= 2$&$\text{Iterations} = 3$ \\
        \hline
        Average Solving Time(s)&0.008 & 0.013&0.015 \\
        Average Construction Time(s)&5.01& 20.62&56.61 \\
        \hline
    \end{tabular}
    \caption{Average Solving and Construction Time: VR with different $N_{max}$}\label{solving construction time ite}
\end{table}

Finally, varying $N_V$ (Figure~\ref{Relative Remaining Gap Distribution MINLPLib N_V} and Table~\ref{solving construction time N_V}) shows negligible improvement from $5$ to $10$, while doubling the cost. Thus, $N_V=5$ seems to be a reasonable default choice.

\begin{table}[ht]
    \centering
    \begin{tabular}{|c|c c|}
        \hline
         &$N_V = 5\;(\text{Default})$  & $N_V= 10$ \\
        \hline
        Average Solving Time(s)&0.008 & 0.010 \\
        Average Construction Time(s)&5.01& 7.66 \\
        \hline
    \end{tabular}
    \caption{Average Solving and Construction Time: VR with different $N_V$}\label{solving construction time N_V}
\end{table}

\begin{appendices}

\section{Missing Proofs}\label{secA1}

\subsection{Proof of Theorem~\ref{thm:pentagon}}

In this proof, we need to characterize the convex hull of 
\[
S^P:=\bigl\{(x_1,f_1,x_2,f_2, f_1f_2) \bigm| (x_i,f_i) \in  P_i\; \for i = 1, 2 \bigr\}. 
\]
  We start by showing that  it is equal to the convex hull of
\[
S^Q:=\bigl\{(x_1,f_1,x_2,f_2, f_1f_2) \bigm| (x_i,f_i) \in  Q_i\; \for i = 1, 2 \bigr\}, 
\]  
where $Q_i$ is the triangle defined as the convex hull of $v_{i1}$, $v_{i2}$ and $v_{i3}$. Let $R:= \bigl\{(x,f,\mu)\bigm| 0 \leq x \leq x',\  (x',f,\mu) \in \conv(S^Q) \bigr\}$. Then, $\conv(S^P) \subseteq R$ holds since $R$ is convex, and, by the definition of $P$, $P \subseteq \bigl\{(x,f) \bigm| 0 \leq x \leq x',\ (x'_i,f_i) \in Q_i \for i =1,2 \bigr\}$ which shows $S^P\subseteq R$. Now, to show $R \subseteq \conv(S^P)$, we consider a point $(\bar{x}, \bar{f},\bar{\mu})\in R$. Then, there exists an $\bar{x}'$  such that $\bar{x}' \geq \bar{x}$ and $(\bar{x}',\bar{f},\bar{\mu}) \in \conv(S^Q)$. Clearly, $(\bar{x},\bar{f},\bar{\mu}) \in H:=\{(x,f,\mu) \mid (x',f,\mu) =(\bar{x}',\bar{f},\bar{\mu}),\ 0\leq x \leq \bar{x}'\}$. So, it suffices to show that the vertex set of $H$ belongs to $\conv(S^P)$. 
 The vertex of the form $(0, \bar{f},\bar{\mu})$ lies in $\conv(S^P)$. For $I \subseteq \{1,2\}$, let $S_I:=\{(x,f,\mu) \mid (x',f,\mu ) \in S^Q, x_i = x'_i \for i \in I,\ x_i = 0 \for i \notin I  \}$, and observe that it can be expressed as the image of $S^Q$ under the affine transformation $A_I:(x',f,\mu) \mapsto (x'_I,0, f,\mu)$. Therefore, $\conv(S_I) =\conv\bigl(A_I(S^Q)\bigr) = A_I\bigl(\conv(S^Q)\bigr)$. Hence, $(\bar{x}'_I,0,\bar{f},\bar{\mu}) \in A_I\bigl(\conv(S^Q)\bigr) = \conv(S_I) \subseteq \conv(S^P)$, where the containment holds since $(\bar{x}',\bar{f},\bar{\mu}) \in \conv(S^Q)$ and the inclusion follows from $S_I \subseteq S^P$.

Next, we characterize the convex hull of $S^Q$. It suffices to derive the concave and convex envelope of $f_1f_2$ over $Q$.  First, we show that $\min\bigl\{\langle \beta_k, (x,f) \rangle + d_k \bigm| k=1, \ldots, 6 \bigr\}$ describes the concave envelope of $f_1f_2$ over $Q$. Since $v_{i1} \leq v_{i2} \leq v_{i3}$, we can treat these vertices as a chain, and $\{v_{11}, v_{12}, v_{13}\} \times \{v_{21}, v_{22}, v_{23}\}$ as a lattice consisting of a product of chains. Moreover, the bilinear term $f_1f_2$ is supermodular over this lattice and concave-extendable from the vertex set of $Q$~\citep[Theorem 8]{tawarmalani2002convex}. It follows from Lemma 2 in~\cite{he2022tractable} that the concave envelope can be obtained by affinely interpolating $f_1f_2$ over the vertices of 6 simplices, $\Upsilon^k$, $k \in \{1, \ldots, 6 \}$ defined as follows. Let $(\pi^1,\ldots, \pi^6) = \bigl( (2,2,1,1), (1,1,2,2), (1,2,2,1), (2,1,1,2), (1,2,1,2), (2,1,2,1)\bigr)$, and for $k \in \{1, \ldots, 6\}$, let 
\[
\Upsilon^k :=\conv\Bigl(\Bigl\{(v_{1j_1}, v_{2j_2}) \Bigm| (j_1, j_2) = (1,1) +\sum_{p=1}^t e_{\pi^i_p} ,\ t=0,\ldots,4 \Bigr\}\Bigr) \quad \text{ for } k=1, \ldots, 6,
\]
where $e_j$ is the $j^{\text{th}}$ principal vector in $\R^2$.  The affine functions $\langle \beta_k, (x,f)\rangle + d_k$ in the statement of the theorem are obtained by interpolating $f_1f_2$ over the vertices of $\Upsilon^k$.

A similar argument, leveraging the submodularity of $f_1f_2$ over the lattice $\{v_{11},v_{12},v_{13}\}\times\{v_{23},v_{22},v_{21}\}$, shows that $\max \{ \langle \alpha_i, (x,f) \rangle + b_i \mid i=1, \ldots, 6 \}$ describes the convex envelope of $f_1f_2$ over $\P$. The affine functions $\langle \alpha_i, (x,f) \rangle + b_i$ are obtained by affinely interpolating $f_1f_2$ over 6 simplices defined as
\[
\Upsilon'_{\pi^i} :=\conv\Bigl(\Bigl\{(v_{1j_1}, v_{2(4-j_2)}) \Bigm| (j_1, j_2) = (1,1) +\sum_{p=1}^k e_{\pi^i_p} ,\ k=0,\ldots,4 \Bigr\}\Bigr) \quad \text{ for } i=1, \ldots, 6.
\]
This completes the proof.
\subsection{Proof of Lemma~\ref{lemma:bilinear-extension}}\label{app:bilinear-extension}
\begin{proof}

Since the set $\AxisAligned$ is a finite union of rectangles, we construct two vectors consisting of extremal values along each coordinate axis. For any rectangle $H$ in the union, let $x^H_{i\min} = \min\{x_i\mid x\in H\}$ and $x^H_{i\max} = \max\{x_i\mid x\in H\}$. We form a sorted vector for axis $x_i$, $i=1,2$ containing all such bounds from every rectangle. We will show that the state space $\state$ of the Markov chain derived by Algorithm~\ref{alg:2d-markov} is a subset of the grid obtained as the Cartesian product of the two sorted vectors. Since the number of rectangles is finite, this grid--and thus $\state$--is finite. Consequently Algorithm~\ref{alg:2d-markov} terminates finitely, as Step~\ref{algstep:unotstate} only adds a point $u$ to the stack if it is newly explored. 

It remains to show that every point generated at Step~\ref{algstep:findEnds} lies on the grid. We prove by induction that if the stack $\mathcal{U}$ contains only grid points, then any point $u$ added at Step~\ref{algstep:addU} is also a grid point. The base case holds because the vertices of the constituent rectangles are grid points by construction. Assume the point $v$ popped from $\mathcal{U}$ on Line~\ref{algstep:simplePop} is a grid point. The endpoints $v_l$ and $v_r$ generated at Step~\ref{algstep:findEnds} lie on the boundaries of some rectangles $H_l, H_r\subseteq \AxisAligned$. Specifically, $v_l$ is obtained by fixing one coordinate to a grid value (matching $v$) and minimizing the other. This minimized coordinate corresponds to the lower bound of $H_l$ and was explicitly included in the constructed vector. Thus, $v_l$ is a grid point. A symmetric argument confirms that $v_r$ is also a grid point. It follows that Algorithm~\ref{alg:2d-markov} terminates finitely, returning a finite discrete Markov chain.

    Then, we prove that $(\state, \transition)$ is absorbing with $\corner(\AxisAligned)$ as absorbing states. By construction, only $\corner(\AxisAligned)$ are absorbing states. Therefore, we only need to prove that for any non-corner state $v\in\state$, there exists a path in $(\state,\transition)$ from it to a corner point.
    Note that every point $v\in\state$ that is not in $\corner(\AxisAligned)$ is a convex combination of grid-points along either the $x_1$- or $x_2$-direction, and the transition probabilities direct the chain towards these neighbors. Now, choose the neighbor that majorizes the current point, that is, its $x_1$- or $x_2$-coordinate is strictly larger than the current point. Since $(\state,\transition)$ is finite, and each iteration yields a strictly majorizing point, this process does not enter a loop, and terminates in finite steps. Since termination can occur only at a corner point, there exists a path from it to a corner point.

    Finally, we prove by induction that for every state $x\in \state$, and every integer $i\geq 1$, we have:
    \begin{equation}\label{eq:singlesteptransition}
     (x, x_1x_2) = \sum_{v\in\state}\transition^i(x,v) \cdot (v, v_1v_2).
    \end{equation}
    For $i = 1$, the statement holds by construction of the the Markov chain.
    Suppose the statement holds for $i = k-1$. Then
     \begin{align*}
      (x,x_1x_2) = \sum_{v\in\state}\transition^{k-1}(x,v) \cdot (v, v_1v_2) &= 
     \sum_{v\in\state} \transition^{k-1}(x,v)\sum_{v^*\in\state}\transition(v,v^*) \cdot (v^*,v^*_1v^*_2)\\
     &=\sum_{v^*\in\state}(\sum_{v\in\state}\transition^{k-1}(x,v)\transition(v,v^*) \cdot (v^*,v^*_1v^*_2))\\
     &=\sum_{v^*\in\state}\transition^k(x,v^*)\cdot (v^*,v^*_1v^*_2),
     \end{align*}
    where the first two equalities are by the induction hypothesis, and the third equality is by interchanging the summations. 
    By induction, the identity holds for all positive integers $i$. Taking the limit as $i\to\infty$, we obtain
    \[(x, x_1x_2) =\sum_{v\in\state}\lim_{i\to +\infty}\transition^i(x,v)(v, v_1v_2) = \sum_{v\in\corner(\AxisAligned)}\transition^\limiting(x,v)(v,v_1v_2),\]
     where the last equality follows from the definition of $\transition^\limiting(x,v)$ and the previous statement, which implies that $\transition^\limiting(x, v) = 0$ for all non-corner states $v$. Thus, $(x,x_1x_2) =\sum_{v\in\corner(\AxisAligned)}\transition^\limiting(x,v)\cdot (v,v_1v_2)$ as required. Since $(x,\ell(x))$ is an affine transform of $(x,x_1x_2)$ and convexification commutes with affine transformation, the result follows.

\end{proof}

\section{Missing Procedures}

\subsection{Factorable relaxation}\label{Factorable relaxation}
We use Algorithm~\ref{alg:factorable relaxation} to create a factorable relaxation which is enhanced as described in Section~\ref{BASE}.
\begin{algorithm}[ht]
\caption{Factorable Relaxation}
\label{alg:factorable relaxation}
\begin{algorithmic}[1]
\Require MINLP Problem $P$ as defined in Eq. \ref{prob:general_optimization}
\Ensure A polyhedral relaxation of $P$
\State Initialize linear system $I$ with the linear constraints from $P$
\State Construct expression trees $\{T_i\}_{i=1}^m$ for all constraint functions $g_i(x)$, and expression tree $T_0$ for objective function $f(x)$
\State Introduce auxiliary variables for all operator nodes in $\{T_i\}_{i=0}^m$ to the system $I$. 
\State Identify the bounds of every variable in the system $I$

\For{each tree $T_i$}
    \For{each operator node $v \in T_i$}
        \State Identify its children $(u_1, u_2,\cdots,u_\ell)$
        \If{$v$ is a linear/affine operation}
            \State Add exact linear equality $v = \text{affine}(u_1, u_2,\cdots,u_\ell)$ to $I$
        \ElsIf{$v = u_1 u_2$ or $v = u_1 / u_2$ (as $v u_2 = u_1$)}
            \State Add McCormick inequalities for the bilinear term to $I$.
        \ElsIf{$v = f(u_1)$ is a univariate nonlinear function}
            \State Add constraints defined by linear under- and over-estimators of $f$ to $I$
        \EndIf
    \EndFor
    \State For $v$ representing the root node of $T_i$, add constraint $v\leq 0$ to $I$.
\EndFor
\end{algorithmic}
\end{algorithm}

\subsection{Computation of Variable Bounds}\label{variablebound}
\subsubsection{Bound Propagation}
Bound propagation~\cite{moore2009introduction,mccormick1976computability} is a computationally efficient technique for estimating lower and upper bounds of a factorable function using interval arithmetic. Given interval bounds for the input variables, bounds for composite expressions are obtained by recursively applying interval operations according to the function's structure. Specifically, for two intervals \( [a^L,a^U] \) and \( [b^L,b^U] \), the basic arithmetic operations are bounded as follows:
\[
\begin{aligned}[b]
[a^L,a^U] + [b^L,b^U] &= [a^L + b^L,\; a^U + b^U], \\
[a^L,a^U] - [b^L,b^U] &= [a^L - b^U,\; a^U - b^L], \\
[a^L,a^U] \times [b^L,b^U] &= 
\big[ \min\{a^L b^L, a^L b^U, a^U b^L, a^U b^U\},\;
       \max\{a^L b^L, a^L b^U, a^U b^L, a^U b^U\} \big], \\
[a^L,a^U] \div [b^L,b^U] &= [a^L,a^U] \times 
\Big[\frac{1}{b^U},\frac{1}{b^L}\Big],
\quad \text{provided } 0 \notin [b^L,b^U].
\end{aligned}
\]
Bounds for elementary univariate functions are obtained by evaluating their minimum and maximum values over the given interval. By propagating these bounds through the computational graph of the factorable function, valid lower and upper bounds on the function value are obtained. Although bound propagation may produce weak bounds, it is very efficient. We implement it on the following function over $[0,1]^3$ to illustrate how to obtain bounds of a factorable function with bound propagation (Figure~\ref{fig:expr_interval_tree}):
\[f(x) = x_1^2 \exp(x_1) x_2 - x_2^2x_3^3x_1^3\]
\begin{figure}[H]
\centering

\centering

\begin{forest}
for tree={
  draw,
  rounded corners,
  align=center,
  inner sep=2pt,
  s sep=8pt,
  l sep=12pt,
  font=\small
}
[$-:$ \text{[-1,e]} 
    [$\times: {[0,e]}$
      [$\times: {[0,e]}$
        [$\power:{[0,1]}$  [$x_1:{[0,1]}$ ] [$2$]]
        [$\exp:{[1,e]}$   [$x_1:{[0,1]}$]]
      ]
      [$x_2:{[0,1]}$ ]
    ]
  [$\times:{[0,1]}$ 
    [$\times:{[0,1]}$ 
      [$\power:{[0,1]}$  [$x_2:{[0,1]}$ ] [$2$]]
      [$\power:{[0,1]}$  [$x_3:{[0,1]}$ ] [$3$]]
    ]
    [$\power:{[0,1]}$  [$x_1:{[0,1]}$ ] [$3$]]
  ]
]
\end{forest}

\caption{Expression tree and corresponding interval tree for
\(f(x)=x_1^2\exp(x_1)x_2 - x_2^2x_3^3x_1^3\) over \([0,1]^3\).}
\label{fig:expr_interval_tree}
\end{figure}
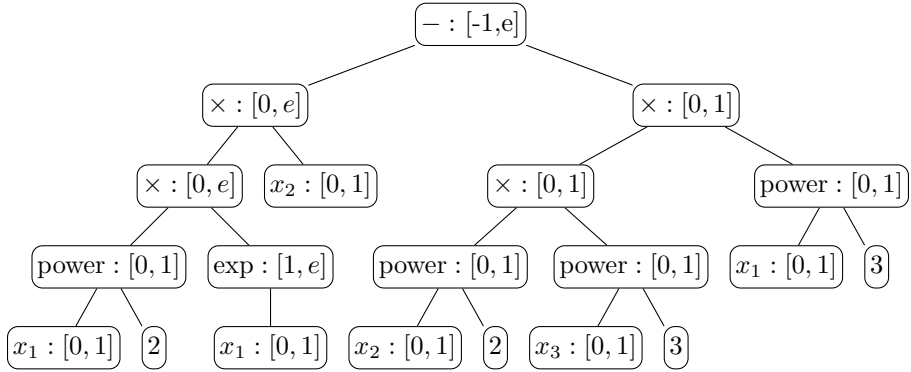

Bound propagation provides the bounds $[-1, e]$ whereas the best possible bounds are $[0, e]$. 
Bound propagation is certified to provide valid bounds in this case~\cite{moore2009introduction}.
Furthermore, when the factorable function only involves continuous operators and basic functions, bound propagation converges to a degenerate interval containing the function value as the domain shrinks, making it compatible with branch-and-bound techniques~\cite{moore2009introduction}.

\subsubsection{Inverse Bound Propagation}
Inverse bound propagation operates in the reverse direction of standard bound propagation. This idea is also used in SCIP and is incorporated in a range reduction technique called OBTT~\cite{bestuzheva2023global}. Whereas bound propagation evaluates an expression tree bottom--up, propagating bounds from the leaves to the root, inverse bound propagation propagates bounds top--down, from the root to the leaves, to tighten variable domains using known constraints.

For example, the constraint \(x^2 \le 1\) implies the variable bound \(x \in [-1,1]\). 
The constraint \(x+y \le 1\) together with initial bounds \((x,y)\in[0,2]^2\), imply the tighter bounds \((x,y)\in[0,1]^2\).
As a formalization, for basic arithmetic operations, the tightening rules are as follows.
\paragraph{Addition.}
If
\[
z = x + y, \quad z \in [z^L,z^U],
\]
then
\[
\begin{aligned}
x &\in [x^L,x^U] \cap [\,z^L - y^U,\; z^U - y^L\,], \\
y &\in [y^L,y^U] \cap [\,z^L - x^U,\; z^U - x^L\,].
\end{aligned}
\]

\paragraph{Subtraction.}
If
\[
z = x - y, \quad z \in [z^L,z^U],
\]
then
\[
\begin{aligned}
x &\in [x^L,x^U] \cap [\,z^L + y^L,\; z^U + y^U\,], \\
y &\in [y^L,y^U] \cap [\,x^L - z^U,\; x^U - z^L\,].
\end{aligned}
\]

\paragraph{Multiplication.}
If
\[
z = x\,y, \quad z \in [z^L,z^U],
\]
and the current bounds are \(x \in [x^L,x^U]\) and \(y \in [y^L,y^U]\), then inverse bound propagation tightens the variable bounds as
\[
\begin{aligned}
x &\in [x^L,x^U] \cap \Bigl(\,[z^L,z^U] \div [y^L,y^U]\,\Bigr), \\
y &\in [y^L,y^U] \cap \Bigl(\,[z^L,z^U] \div [x^L,x^U]\,\Bigr),
\end{aligned}
\]

\paragraph{Division.}
If
\[
z = \frac{x}{y}, \quad z \in [z^L,z^U],
\]
with \(y \in [y^L,y^U]\) and \(0 \notin [y^L,y^U]\), then inverse bound propagation yields
\[
\begin{aligned}
x &\in [x^L,x^U] \cap \Bigl(\,[z^L,z^U] \times [y^L,y^U]\,\Bigr), \\
y &\in [y^L,y^U] \cap \Bigl(\,[x^L,x^U] \div [z^L,z^U]\,\Bigr),
\end{aligned}
\]

\paragraph{Monotone univariate functions.}
If
\[
z = \phi(x),
\]
where \(\phi\) is monotone on \([x^L,x^U]\), then
\[
x \in [x^L,x^U] \cap \phi^{-1}([z^L,z^U]).
\]
In particular, piecewise monotone univariate functions, such as $x^2$ can be handled by applying the above rule separately on each monotone segment.

These rules are applied recursively along the expression tree until no further domain reductions are possible or the procedure is terminated by a prescribed iteration limit.

\subsection{Procedure to generate random polynomial optimization instances}\label{app:generation}  
The instances are generated as follows:
\begin{itemize}
	\item each monomial $x^{\alpha_j}$ is generated by first randomly selecting nonzero entries of $\alpha_j$ such that  the number of nonzero entries is $2$ or $3$ with equal probability, then assign each nonzero entry $2$ or $3$ uniformly. 
	\item each entry of $d$ and $B$ is zero with probability $0.3$ and uniformly generated from $[0,1]$ with probability $0.7$.
	\item each entry of $A$ is uniformly distributed over $[-10,10]$, and $x^L_i$ and $x_i^U$ is uniformly selected from $\{0,1,2\}$ and $\{3,4\}$, respectively.
	\item $c = \sum_{j = 1}^m \nabla m_j(\tilde{x}) d_j$, and $b = A\tilde{x} + B\tilde{y}$, where $\tilde{x}$ is randomly generated from $[x^L, x^U]$, $\nabla m_j(\tilde{x})$ is the gradient of $x^{\alpha_j}$ at $\tilde{x}$, and $\tilde{y} = (\tilde{x}^{\alpha_1}, \ldots, \tilde{x}^{\alpha_m} )$.
\end{itemize}
This procedure for generating random instances was first suggested in \cite{he2024mip}.

\section{Missing figures}
\subsection{Convex extension Markov chain of the axis-aligned region in Figure~\ref{markov chain polytope}}\label{secA2}
In Figure~\ref{markov chain polytope}, we illustrate the Markov chain corresponding to the region in Figure~\ref{fig:loop}, where $\AxisAligned = [0,1]\times [0,2]\cup [0,2]\times[1,2]\cup[2,4]\times[1,4]\cup[2,3]\times[4,5]\cup[4,5]\times[3,4]\cup[3,4]\times[0,1]$. Though there is a loop $s_7-s_9-s_{10}-s_{11}-s_7$, the stationary joint distribution can give a convex combination of $s_7 = (2,1)$ by only the corner points of $\AxisAligned$. Let $\transition^\limiting$ be the limiting distribution of the Markov chain represented by a $21\times 21$ matrix, where $\transition^\limiting_{ij}$ represents the probability of ending up with $s_j$ when starting from $s_i$.
Then, for $s_7$, the relevant transition probabilities are as follows:
\[\transition^\limiting_{7,1} = \frac{24}{95},\transition^\limiting_{7,4} = \frac{24}{95},\transition^\limiting_{7,12} = \frac{8}{95},\transition^\limiting_{7,15} = \frac{3}{95},\transition^\limiting_{7,17} = \frac{36}{95},\]
while $\transition^\limiting_{7,i} = 0$ for other states $i$. It is easy to check that  $s_7 = \sum_{i=1}^{21}\transition^\limiting_{7,i} s_i$, and $s_7[1]s_7[2] = \sum_{i=1}^{21}\transition^\limiting_{7,i} s_i[1]s_i[2]$. 
\begin{figure}
    \centering
    \includegraphics{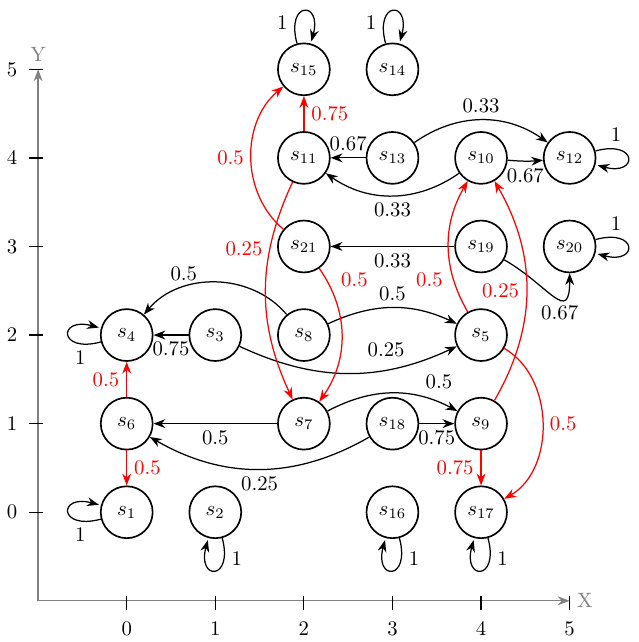}
    \caption{Markov Chain}
    \label{markov chain polytope}
\end{figure}

\end{appendices}

\newpage
\bibliographystyle{plainnat} %
\bibliography{reference}%

@article{lubin2023jump,
  title={\text{JuMP} 1.0: Recent improvements to a modeling language for mathematical optimization},
  author={Lubin, Miles and Dowson, Oscar and Garcia, Joaquim Dias and Huchette, Joey and Legat, Beno{\^\i}t and Vielma, Juan Pablo},
  journal={Mathematical Programming Computation},
  volume={15},
  number={3},
  pages={581--589},
  year={2023},
  publisher={Springer}
}

@article{nagarajan2019adaptive,
  title={An adaptive, multivariate partitioning algorithm for global optimization of nonconvex programs},
  author={Nagarajan, Harsha and Lu, Mowen and Wang, Site and Bent, Russell and Sundar, Kaarthik},
  journal={Journal of Global Optimization},
  volume={74},
  number={4},
  pages={639--675},
  year={2019},
  publisher={Springer}
}

@article{atamturk2023supermodularity,
  title={Supermodularity and valid inequalities for quadratic optimization with indicators},
  author={Atamt{\"u}rk, Alper and G{\'o}mez, Andr{\'e}s},
  journal={Mathematical Programming},
  volume={201},
  number={1},
  pages={295--338},
  year={2023},
  publisher={Springer}
}

@article{dey2025second,
  title={A second-order cone representable class of nonconvex quadratic programs},
  author={Dey, Santanu S and Khajavirad, Aida},
  journal={arXiv preprint arXiv:2508.18435},
  year={2025}
}

@article{he2025convexification,
  title={Convexification techniques for fractional programs},
  author={He, Taotao and Liu, Siyue and Tawarmalani, Mohit},
  journal={Mathematical Programming},
  volume={213},
  number={1},
  pages={107--149},
  year={2025},
  publisher={Springer}
}

@article{del2021running,
  title={The running intersection relaxation of the multilinear polytope},
  author={Del Pia, Alberto and Khajavirad, Aida},
  journal={Mathematics of Operations Research},
  volume={46},
  number={3},
  pages={1008--1037},
  year={2021},
  publisher={INFORMS}
}

@article{bao2015global,
  title={Global optimization of nonconvex problems with multilinear intermediates},
  author={Bao, Xiaowei and Khajavirad, Aida and Sahinidis, Nikolaos V and Tawarmalani, Mohit},
  journal={Mathematical Programming Computation},
  volume={7},
  number={1},
  pages={1--37},
  year={2015},
  publisher={Springer}
}

@article{he2022tractable,
  title={Tractable relaxations of composite functions},
  author={He, Taotao and Tawarmalani, Mohit},
  journal={Mathematics of Operations Research},
  volume={47},
  number={2},
  pages={1110--1140},
  year={2022},
  publisher={INFORMS}
}

@article{he2021new,
  title={A new framework to relax composite functions in nonlinear programs},
  author={He, Taotao and Tawarmalani, Mohit},
  journal={Mathematical Programming},
  volume={190},
  number={1},
  pages={427--466},
  year={2021},
  publisher={Springer}
}

@article{locatelli2018convex,
  title={Convex envelopes of bivariate functions through the solution of KKT systems},
  author={Locatelli, Marco},
  journal={Journal of global optimization},
  volume={72},
  number={2},
  pages={277--303},
  year={2018},
  publisher={Springer}
}

@article{muller2020using,
  title={Using two-dimensional projections for stronger separation and propagation of bilinear terms},
  author={Muller, BENJAMIN and Serrano, Felipe and Gleixner, Ambros},
  journal={SIAM Journal on Optimization},
  volume={30},
  number={2},
  pages={1339--1365},
  year={2020},
  publisher={SIAM}
}

@article{mccormick1976computability,
  title={Computability of global solutions to factorable nonconvex programs: Part I—Convex underestimating problems},
  author={McCormick, Garth P},
  journal={Mathematical programming},
  volume={10},
  number={1},
  pages={147--175},
  year={1976},
  publisher={Springer}
}

@article{tawarmalani2002convex,
  title={Convex extensions and envelopes of lower semi-continuous functions},
  author={Tawarmalani, Mohit and Sahinidis, Nikolaos V},
  journal={Mathematical programming},
  volume={93},
  number={2},
  pages={247--263},
  year={2002},
  publisher={Springer}
}

@article{barber1996quickhull,
  title={The quickhull algorithm for convex hulls},
  author={Barber, C Bradford and Dobkin, David P and Huhdanpaa, Hannu},
  journal={ACM Transactions on Mathematical Software (TOMS)},
  volume={22},
  number={4},
  pages={469--483},
  year={1996},
  publisher={Acm New York, NY, USA}
}

@article{kamenev1996algorithm,
  title={An algorithm for approximating polyhedra},
  author={Kamenev, George K},
  journal={Computational mathematics and mathematical physics},
  volume={36},
  number={4},
  pages={533--544},
  year={1996},
  publisher={Pergamon Press, Inc. Elmsford, NY, USA}
}

@article{lotov2008modified,
  title={The modified method of refined bounds for polyhedral approximation of convex polytopes},
  author={Lotov, Alexander Vladimirovich and Pospelov, Alexis I},
  journal={Computational Mathematics and Mathematical Physics},
  volume={48},
  number={6},
  pages={933--941},
  year={2008},
  publisher={Springer}
}

@article{tawarmalani2004global,
  title={Global optimization of mixed-integer nonlinear programs: A theoretical and computational study},
  author={Tawarmalani, Mohit and Sahinidis, Nikolaos V},
  journal={Mathematical programming},
  volume={99},
  number={3},
  pages={563--591},
  year={2004},
  publisher={Springer}
}

@article{misener2014antigone,
  title={\text{ANTIGONE}: algorithms for continuous/integer global optimization of nonlinear equations},
  author={Misener, Ruth and Floudas, Christodoulos A},
  journal={Journal of Global Optimization},
  volume={59},
  number={2},
  pages={503--526},
  year={2014},
  publisher={Springer}
}

@article{mahajan2021minotaur,
  title={Minotaur: A mixed-integer nonlinear optimization toolkit},
  author={Mahajan, Ashutosh and Leyffer, Sven and Linderoth, Jeff and Luedtke, James and Munson, Todd},
  journal={Mathematical Programming Computation},
  volume={13},
  number={2},
  pages={301--338},
  year={2021},
  publisher={Springer}
}

@article{bestuzheva2023global,
  title={Global optimization of mixed-integer nonlinear programs with {SCIP} 8},
  author={Bestuzheva, Ksenia and Chmiela, Antonia and M{\"u}ller, Benjamin and Serrano, Felipe and Vigerske, Stefan and Wegscheider, Fabian},
  journal={Journal of Global Optimization},
  pages={1--24},
  year={2023},
  publisher={Springer}
}

@article{cohen20023d,
  title={3D line voxelization and connectivity control},
  author={Cohen-Or, Daniel and Kaufman, Arie},
  journal={IEEE Computer Graphics and Applications},
  volume={17},
  number={6},
  pages={80--87},
  year={2002},
  publisher={IEEE}
}

@article{sramek2002alias,
  title={Alias-free voxelization of geometric objects},
  author={Sramek, Milos and Kaufman, Arie E},
  journal={IEEE transactions on visualization and computer graphics},
  volume={5},
  number={3},
  pages={251--267},
  year={2002},
  publisher={IEEE}
}

@article{stolte1997robust,
  title={Robust voxelization of surfaces},
  author={Stolte, Nilo},
  journal={Technial Report of Center for Visual Computing and Computer Science Department, State University of New York at Stony Brook},
  year={1997}
}

@article{stolte2001novel,
  title={Novel techniques for robust voxelization and visualization of implicit surfaces},
  author={Stolte, Nilo and Kaufman, Arie},
  journal={Graphical Models},
  volume={63},
  number={6},
  pages={387--412},
  year={2001},
  publisher={Elsevier}
}

@article{gorte2016rasterization,
  title={Rasterization and voxelization of two-and three-dimensional space partitionings},
  author={Gorte, BGH and Zlatanova, Sisi},
  year={2016}
}

@article{tawarmalani2013explicit,
  title={Explicit convex and concave envelopes through polyhedral subdivisions},
  author={Tawarmalani, Mohit and Richard, Jean-Philippe P and Xiong, Chuanhui},
  journal={Mathematical Programming},
  volume={138},
  number={1},
  pages={531--577},
  year={2013},
  publisher={Springer}
}

@article{he2024mip,
  title={\text{MIP} relaxations in factorable programming},
  author={He, Taotao and Tawarmalani, Mohit},
  journal={SIAM Journal on Optimization},
  volume={34},
  number={3},
  pages={2856--2882},
  year={2024},
  publisher={SIAM}
}

@book{grinstead2012introduction,
  title={Introduction to probability},
  author={Grinstead, Charles Miller and Snell, James Laurie},
  year={2012},
  publisher={American Mathematical Soc.},
  address={Providence, RI}
}

@book{moore2009introduction,
  title={Introduction to interval analysis},
  author={Moore, Ramon E and Kearfott, R Baker and Cloud, Michael J},
  year={2009},
  publisher={SIAM},
  address={Philadelphia, PA}
}

@article{wachter2006implementation,
  title={On the implementation of an interior-point filter line-search algorithm for large-scale nonlinear programming},
  author={W{\"a}chter, Andreas and Biegler, Lorenz T},
  journal={Mathematical programming},
  volume={106},
  number={1},
  pages={25--57},
  year={2006},
  publisher={Springer}
}

@misc{gurobi,
  author = {{Gurobi Optimization, LLC}},
  title = {{Gurobi Optimizer Reference Manual}},
  year = 2026,
  url = "https://www.gurobi.com"
}

@article{Julia-2017,
    title={\text{Julia}: A fresh approach to numerical computing},
    author={Bezanson, Jeff and Edelman, Alan and Karpinski, Stefan and Shah, Viral B},
    journal={SIAM {R}eview},
    volume={59},
    number={1},
    pages={65--98},
    year={2017},
    publisher={SIAM},
    doi={10.1137/141000671},
    url={https://epubs.siam.org/doi/10.1137/141000671}
}

@article{bussieck2003minlplib,
  title={\text{MINLPLib}—a collection of test models for mixed-integer nonlinear programming},
  author={Bussieck, Michael R and Drud, Arne Stolbjerg and Meeraus, Alexander},
  journal={INFORMS Journal on Computing},
  volume={15},
  number={1},
  pages={114--119},
  year={2003},
  publisher={INFORMS}
}

@article{sherali1990hierarchy,
  title={A hierarchy of relaxations between the continuous and convex hull representations for zero-one programming problems},
  author={Sherali, Hanif D and Adams, Warren P},
  journal={SIAM Journal on Discrete Mathematics},
  volume={3},
  number={3},
  pages={411--430},
  year={1990},
  publisher={SIAM}
}

@inproceedings{lasserre2001explicit,
  title={An explicit exact SDP relaxation for nonlinear 0-1 programs},
  author={Lasserre, Jean B},
  booktitle={International Conference on Integer Programming and Combinatorial Optimization},
  pages={293--303},
  year={2001},
  organization={Springer}
}

@article{bonami2008algorithmic,
  title={An algorithmic framework for convex mixed integer nonlinear programs},
  author={Bonami, Pierre and Biegler, Lorenz T and Conn, Andrew R and Cornu{\'e}jols, G{\'e}rard and Grossmann, Ignacio E and Laird, Carl D and Lee, Jon and Lodi, Andrea and Margot, Fran{\c{c}}ois and Sawaya, Nicolas and others},
  journal={Discrete optimization},
  volume={5},
  number={2},
  pages={186--204},
  year={2008},
  publisher={Elsevier}
}

@article{kronqvist2019review,
  title={A review and comparison of solvers for convex MINLP},
  author={Kronqvist, Jan and Bernal, David E and Lundell, Andreas and Grossmann, Ignacio E},
  journal={Optimization and Engineering},
  volume={20},
  number={2},
  pages={397--455},
  year={2019},
  publisher={Springer}
}

@article{coey2020outer,
  title={Outer approximation with conic certificates for mixed-integer convex problems},
  author={Coey, Chris and Lubin, Miles and Vielma, Juan Pablo},
  journal={Mathematical Programming Computation},
  volume={12},
  number={2},
  pages={249--293},
  year={2020},
  publisher={Springer}
}

@article{gunluk2010perspective,
  title={Perspective reformulations of mixed integer nonlinear programs with indicator variables},
  author={G{\"u}nl{\"u}k, Oktay and Linderoth, Jeff},
  journal={Mathematical programming},
  volume={124},
  number={1},
  pages={183--205},
  year={2010},
  publisher={Springer}
}

@article{burer2009copositive,
  title={On the copositive representation of binary and continuous nonconvex quadratic programs},
  author={Burer, Samuel},
  journal={Mathematical Programming},
  volume={120},
  number={2},
  pages={479--495},
  year={2009},
  publisher={Springer}
}

@article{hojny2024detecting,
  title={Detecting and handling reflection symmetries in mixed-integer (nonlinear) programming},
  author={Hojny, Christopher},
  journal={arXiv preprint arXiv:2405.08379},
  year={2024}
}

@article{land1960automatic,
  title={An Automatic Method of Solving Discrete Programming Problems},
  author={Land, AH and Doig, AG},
  journal={Econometrica},
  volume={28},
  number={3},
  pages={497--520},
  year={1960}
}

@article{ryoo1996branch,
  title={A branch-and-reduce approach to global optimization},
  author={Ryoo, Hong S and Sahinidis, Nikolaos V},
  journal={Journal of global optimization},
  volume={8},
  number={2},
  pages={107--138},
  year={1996},
  publisher={Springer}
}

@article{tawarmalani2005polyhedral,
  title={A polyhedral branch-and-cut approach to global optimization},
  author={Tawarmalani, Mohit and Sahinidis, Nikolaos V},
  journal={Mathematical programming},
  volume={103},
  number={2},
  pages={225--249},
  year={2005},
  publisher={Springer}
}

@article{belotti2009branching,
  title={Branching and bounds tighteningtechniques for non-convex MINLP},
  author={Belotti, Pietro and Lee, Jon and Liberti, Leo and Margot, Fran{\c{c}}ois and W{\"a}chter, Andreas},
  journal={Optimization Methods \& Software},
  volume={24},
  number={4-5},
  pages={597--634},
  year={2009},
  publisher={Taylor \& Francis}
}

@article{rikun1997convex,
  title={A convex envelope formula for multilinear functions},
  author={Rikun, Anatoliy D},
  journal={Journal of Global Optimization},
  volume={10},
  number={4},
  pages={425--437},
  year={1997},
  publisher={Springer}
}

@book{rockafellar1998variational,
  title={Variational analysis},
  author={Rockafellar, R Tyrrell and Wets, Roger JB},
  year={1998},
  publisher={Springer},
  address={Berlin}
}

@book{tawarmalani2013convexification,
  title={Convexification and global optimization in continuous and mixed-integer nonlinear programming: theory, algorithms, software, and applications},
  author={Tawarmalani, Mohit and Sahinidis, Nikolaos V},
  volume={65},
  year={2013},
  publisher={Springer Science \& Business Media},
    address={Berlin}
}

@article{ben2021lectures,
  title={Lectures on modern convex optimization 2020},
  author={Ben-Tal, Aharon and Nemirovski, Arkadi},
  journal={SIAM, Philadelphia},
  year={2021}
}

@book{conforti2014integer,
  title={Integer Programming},
  author={Conforti, Michele and Cornu{\'e}jols, G{\'e}rard and Zambelli, Giacomo},
  series={Graduate Texts in Mathematics},
  volume={271},
  year={2014},
  publisher={Springer International Publishing},
  address={Cham, Switzerland}
}

@article{crama1993concave,
  title={Concave extensions for nonlinear 0--1 maximization problems},
  author={Crama, Yves},
  journal={Mathematical Programming},
  volume={61},
  number={1},
  pages={53--60},
  year={1993},
  publisher={Springer}
}

@article{meyer2005convex,
  title={Convex envelopes for edge-concave functions},
  author={Meyer, Clifford A and Floudas, Christodoulos A},
  journal={Mathematical programming},
  volume={103},
  number={2},
  pages={207--224},
  year={2005},
  publisher={Springer}
}

@article{locatelli2016polyhedral,
  title={Polyhedral subdivisions and functional forms for the convex envelopes of bilinear, fractional and other bivariate functions over general polytopes},
  author={Locatelli, Marco},
  journal={Journal of global optimization},
  volume={66},
  number={4},
  pages={629--668},
  year={2016},
  publisher={Springer}
}

\end{document}